\documentclass[11pt]{amsart}
\usepackage{soul}
\usepackage{cancel}
\usepackage[margin=1in]{geometry}
\usepackage[T1]{fontenc}
\usepackage[utf8]{inputenc}
\usepackage{amsmath,amssymb,amsthm,mathtools}
\usepackage{bbm}
\usepackage[shortlabels]{enumitem}
\usepackage{tikz}
\usetikzlibrary{decorations.pathreplacing,decorations.pathmorphing,arrows.meta}
\usepackage{caption}
\PassOptionsToPackage{hyphens}{url}
\usepackage[colorlinks=true, linkcolor=blue, citecolor=blue,urlcolor=blue]{hyperref}

\allowdisplaybreaks
\numberwithin{equation}{section}
\numberwithin{figure}{section}

\newtheorem{theorem}{Theorem}[section]
\newtheorem{corollary}[theorem]{Corollary}
\newtheorem{lemma}[theorem]{Lemma}
\newtheorem{proposition}[theorem]{Proposition}
\theoremstyle{definition}

\theoremstyle{remark}
\newtheorem{remark}[theorem]{Remark}

\newcommand{\R}{\mathbb{R}}

\newcommand{\Z}{\mathbb{Z}}
\newcommand{\EE}{\mathbb{E}}
\newcommand{\PP}{\mathbb{P}}
\newcommand{\DD}{\mathbb{D}}
\newcommand{\rmP}{{\mathrm P}}

\newcommand{\ind}{\mathbf{1}}

\newcommand{\cB}{\mathcal{B}}

\newcommand{\cF}{\mathcal{F}}
\newcommand{\cH}{\mathcal{H}}

\newcommand{\cR}{\mathcal{R}}
\newcommand{\cS}{\mathcal{S}}
\newcommand{\cU}{\mathcal{U}}
\newcommand{\cZ}{\mathcal{Z}}
\newcommand{\fB}{\mathfrak{B}}

\newcommand{\fc}{\mathfrak{c}}
\newcommand{\fH}{\mathfrak{H}}

\newcommand{\rfree}{\mathrm{free}}
\newcommand{\rstat}{\mathrm{stat}}

\newcommand{\al}{\alpha}
\newcommand{\be}{\beta}
\newcommand{\dl}{\delta}

\newcommand{\eps}{\varepsilon}
\newcommand{\ga}{\gamma}
\newcommand{\wt}{\widetilde}
\newcommand{\DN}{\Delta_n}
\newcommand{\noi}{\noindent}

\newcommand{\AND}{\quad{\rm and}\quad}

\newcommand{\Var}{\operatorname{Var}}
\newcommand{\supp}{\operatorname{supp}}
\newcommand{\OCY}{\textup{O'Connell--Yor}}
\DeclareMathOperator{\sgn}{sgn}
\DeclareMathOperator{\Cov}{Cov}
\DeclareMathOperator{\Exp}{Exp}

\tikzset{
  brownian background/.style={
    color={rgb, 255:red, 155; green, 155; blue, 155},
    draw opacity=0.66,
    line cap=round,
    decorate,
    decoration={random steps,segment length=7pt,amplitude=1.4pt}
  }
}

\tikzset{
  >=Latex,
  lev/.style={gray!65, dashed},
  redpath/.style={line width=1.1pt, red},
  bluepath/.style={line width=1.1pt, blue},
  labr/.style={red},
  labb/.style={blue},
  ptred/.style={circle, fill=red, inner sep=1.6pt},
  ptblue/.style={circle, fill=blue, inner sep=1.6pt},
  axes/.style={line width=.9pt, black},
  braceann/.style={decorate,decoration={brace,amplitude=5pt}}
}

\title[Polymer coalescence and fluctuations via free-energy correlation profile]{Coalescence and fluctuations of O'Connell--Yor polymers\\
via the free-energy correlation profile}

\author[F.~Rassoul-Agha, X.~Shen, R.~Zhang,  and G.~Zheng]
{Firas Rassoul-Agha, Xiao Shen, Ruixuan Zhang,  and Guangqu Zheng}

\thanks{F.\ Rassoul-Agha was partially supported by National Science Foundation grant DMS-2450951 and Simons Foundation grant MPS-TSM-00013661.}

\thanks{X.\ Shen was partially supported by Simons Foundation grant MPS-TSM-00024840.}

\address{F. Rassoul-Agha,  Department of Mathematics, University of Utah, Salt Lake City,
  UT 84112, USA}
\email{firas@math.utah.edu}

\address{X. Shen, Department of Mathematics, North Carolina State University,
  Raleigh, NC 27695, USA}
\email{xshen9@ncsu.edu}

\address{R. Zhang, Department of Mathematics, University of Utah, Salt Lake City,
  UT 84112, USA}
\email{ray.zhang@math.utah.edu}

\address{G. Zheng, Department of Mathematics and Statistics, Boston University,
  Boston, MA 02215, USA}
\email{gzheng90@bu.edu}

\date{\today}

\begin{document}

\begin{abstract}
We establish an identity relating the spatial derivative of a two-point free-energy correlation to two fundamental geometric quantities: the annealed probability that two independently sampled polymers, in the same environment, meet before reaching the prescribed terminal level, and the exit-point location of a single polymer. Our result is inspired in part by recent integration-by-parts work of Gu and Quastel for the KPZ equation. A crucial step in their continuous setting relies on an ingenious application of It\^o's formula, which has no direct analogue in our semi-discrete setting. Instead, we exploit intrinsic symmetries of the polymer model together with the memoryless property. 
We also provide two applications: (i) we recover the horizontal Burke property by proving that the anchored
stationary horizontal free-energy profile is a two-sided Brownian motion;
(ii)  in the zero-temperature limit, 
we obtain the corresponding identity for Brownian last-passage percolation.
\end{abstract}

\maketitle
\tableofcontents

\section{Introduction}
Among $(1+1)$-dimensional random growth models in the
Kardar--Parisi--Zhang (KPZ) universality class, the $\OCY$ polymer, introduced by O'Connell and Yor \cite{OCY01}, is
one of the fundamental exactly solvable models. Moriarty and O'Connell \cite{MO06}
identified the free-energy density of the $\OCY$ polymer. Its
integrable structure was subsequently elucidated through O'Connell's
connection between the polymer partition function, the quantum Toda
lattice, and Whittaker functions \cite{OC12,OC14survey}, as well as
through geometric and tropical RSK correspondences \cite{COSZ14}.
In parallel, Borodin and Corwin \cite{BC14} developed the
Macdonald-process framework, whose $q\to1$ degeneration to Whittaker
processes provides another route to the integrable structure of the
$\OCY$ polymer. These developments led to exact Fredholm determinant representations for transforms of the partition function and the free-energy distribution, both through the Macdonald-process framework, and, from a different determinantal perspective, in the work of Imamura and Sasamoto \cite{IS16,IS17}. Building on the
former representation, Borodin, Corwin, and Ferrari \cite{BCF14}
performed the asymptotic analysis of the Fredholm determinant and proved
the convergence of the rescaled free energy to the GUE Tracy--Widom
distribution. More broadly, at the level of large-scale scaling
limits, the model is connected with the KPZ fixed point
\cite{MQR21} and the directed landscape \cite{DOV22}; see the
works of Dauvergne and Zhang \cite{DZ24} and Vir\'ag and Wu \cite{VW25}.

Alongside these developments in integrable probability, a substantial probabilistic theory of the $\OCY$ polymer has emerged from its stationary structure, which facilitates the use of coupling techniques.  Burke's theorem  in   \cite{OCY01} provides a stationary version of the polymer and has
become a fundamental tool for studying its geometry and fluctuations.
Sepp\"al\"ainen and Valk\'o \cite{SV10} exploited this stationary
structure, together with exit-point estimates, to identify the KPZ
fluctuation and transversal fluctuation exponents. Building on the same
Burke property, Alberts, Rassoul-Agha, and Simper \cite{ARS20}
constructed Busemann functions as almost-sure limits of
partition-function ratios and used them to construct semi-infinite
$\OCY$ polymer measures in prescribed asymptotic directions. For the
fluctuations of the free energy, Janjigian \cite{Jan15} obtained the
large-deviation rate function through moment Lyapunov exponents, while
Noack and Sosoe \cite{NS22} obtained upper bounds for the higher-order central moments. Building on a generating-function identity
originating in the work of Emrah, Janjigian, and Sepp\"al\"ainen
\cite{EJS20}, Landon and Sosoe \cite{LS24} established left and right tail bounds for the free energy
of the correct order throughout the moderate-deviation regime. More
recently, Groathouse, Rassoul-Agha, Sepp\"al\"ainen, and Sorensen
\cite{GRSS25} studied the jointly stationary measure
for the $\OCY$ polymer.

Beyond coupling techniques, percolation arguments provide another
important set of tools for understanding random growth models. These
approaches exploit fine information about optimal paths to study 
properties of the random growth and have been particularly successful in
zero-temperature models. In the zero-temperature setting, the coalescence of geodesics
plays an important role, and various coalescence estimates have been
obtained in \cite{Pimentel16, BSS19,SS20, Zha20}. At positive temperature, analogous geometric information is more difficult to access because the partition function is defined through a collection of paths, rather than being determined by a single optimizing geodesic. Nevertheless,
recent work has begun to develop percolation-type methods for
positive-temperature models. In particular, a coalescence estimate was
obtained for the inverse-gamma polymer in \cite{RSS24}, and a related
BK-type inequality was established for the KPZ equation and the KPZ line ensemble in \cite{GHZ25}.

In this broader context, we establish an identity relating three
natural objects associated with the $\OCY$ polymer: the spatial
derivative of a two-point free-energy correlation, the first meeting level of two independent polymer replicas in the same environment (which is equivalent to coalescence), and
the terminal location of a single polymer.
More precisely, fix $\beta > 0$ and $\alpha \in \mathbb{R}$. Let us define the free energy of the
point-to-line polymer started at $(x,0)$ and terminating on the line
$y=n$ as follows
$$
  F_{x}^{\al, \be}
  =
  \be^{-1}\log Z_{x}^{\al, \be}.
$$
For two distinct starting points $\wt x$ and $x$, we obtain an
exact formula \eqref{main_eq} for the spatial derivative of the two-point free-energy
correlation, 
\[
  \frac{\partial}{\partial x}
  \EE\bigl[F_{\wt x}^{\al,\be}F_x^{\al,\be}\bigr].
\]

In our notation for the free energy $F_{\wt x}^{\al,\be}$, $\beta$ denotes the inverse temperature, while the parameter $\alpha$ controls the strength of the Brownian terminal condition and interpolates between two particularly important regimes. When $\al=0$, the boundary term disappears and the identity
reduces to a formula for the coalescence probability.
When $\al=1$, the coalescence coefficient vanishes and the identity becomes
a formula for the terminal exit point. Differentiating the latter
relation in $\al$ recovers the stationary coalescence probability.
These consequences are recorded in Corollaries~\ref{flat_cor},
\ref{stat_exit}, and \ref{stat_coal}.

In the stationary regime, the covariance identity admits a nonlinear
finite-dimensional extension, whose resulting formula 
leads to Stein's characterization of 
the multivariate normal distribution.
Applied to consecutive spatial
increments, it proves that the anchored horizontal free-energy
profile is a two-sided standard Brownian motion; see
Proposition~\ref{stat_stein} and Corollary~\ref{buse_cor}, giving a
direct Gaussian integration-by-parts proof of the horizontal Burke
property.

Our argument is motivated in part by the Gaussian
integration-by-parts framework that Gu and Quastel developed for the
KPZ equation \cite{GQ26}, whose continuum proof exploits a subtle error estimate arising from an ingenious application of It\^o's formula.
Since this mechanism is not available in the present semi-discrete
setting, we instead mollify the Brownian environment, apply
Malliavin integration by parts directly to the driving white noises,
and exploit intrinsic symmetries within the polymer model
together with the memoryless property.
Lastly, when sending $\beta\to \infty$, we obtain its zero-temperature counterpart for point-to-line
Brownian last-passage percolation, where the polymer measures
concentrate on geodesics.

\subsection{Point-to-line polymers with Brownian terminal data}

\begin{figure}[t]
  \centering
  \pgfmathsetseed{20260406}

  \begin{tikzpicture}[x=1cm,y=.72cm]

    \foreach \yy in {0,1,2,4,5}{
      \draw[brownian background] (0.2,\yy) -- (8.2,\yy);
    }

    \node[left] at (0.2,0) {$0$};
    \node[left] at (0.2,1) {$1$};
    \node[left] at (0.2,2) {$2$};
    \node[left] at (0.2,3) {$\vdots$};
    \node[left] at (0.2,4) {$n-1$};
    \node[left] at (0.2,5) {$n$};

    \node[right] at (8.25,0) {$y=0$};
    \node[right] at (8.25,5) {$y=n$};


    \draw[bluepath]
      (1.0,0)
      -- (2.1,0)
      -- (2.1,1)
      -- (3.25,1)
      -- (3.25,2);

    \draw[bluepath,densely dotted]
      (3.25,2)
      -- (4.15,2)
      -- (4.15,3.05)
      -- (5.25,3.05)
      -- (5.25,4);

    \draw[bluepath]
      (5.25,4)
      -- (6.75,4)
      -- (6.75,5);

    \fill[blue] (1.0,0) circle (1.7pt);
    \fill[blue] (2.1,1) circle (1.7pt);
    \fill[blue] (3.25,2) circle (1.7pt);
    \fill[blue] (4.15,3.05) circle (1.7pt);
    \fill[blue] (5.25,4) circle (1.7pt);
    \fill[blue] (6.75,5) circle (1.7pt);

    \draw[blue,fill=white,line width=.9pt]
      (2.1,0) circle (2.0pt);

    \draw[blue,fill=white,line width=.9pt]
      (3.25,1) circle (2.0pt);

    \draw[blue,fill=white,line width=.9pt]
      (4.15,2) circle (2.0pt);

    \draw[blue,fill=white,line width=.9pt]
      (5.25,3.05) circle (2.0pt);

    \draw[blue,fill=white,line width=.9pt]
      (6.75,4) circle (2.0pt);

    \node[anchor=north east] at (0.95,-0.10)
      {$\ga_0$};

    \node[anchor=south east] at (2.05,1.08)
      {$\ga_1$};

    \node[anchor=south east] at (3.20,2.08)
      {$\ga_2$};

    \node[anchor=south east] at (5.20,4.08)
      {$\ga_{n-1}$};

    \node[anchor=south] at (6.75,5.15)
      {$\ga_n$};

  \end{tikzpicture}

  \captionsetup{width=.84\linewidth}
  \caption{
    The gray curves schematically represent the Brownian environment
    on the horizontal levels.  The blue up-right path represents a
    polymer trajectory with jump locations
    $\ga_1,\dots,\ga_n$.  At the jump location $\ga_i$, the open
    circle is the endpoint of the horizontal segment on level $i-1$,
    while the solid circle is the starting point of the horizontal
    segment on level $i$.
  }
  \label{fig_jump}
\end{figure}
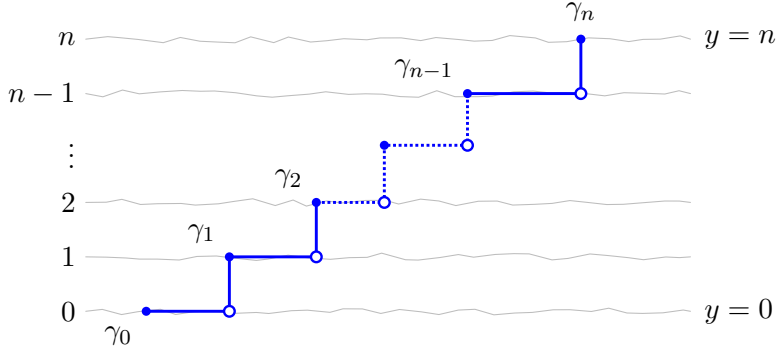

Fix $n\in\Z_{>0}$ and an inverse temperature $\be>0$.  Let
$\{X_i\}_{i\geq1}$ be i.i.d.\ exponential random variables with rate
$\be$, and set
\begin{align}
  \ga_0=0,
  \qquad
  \ga_m=\sum_{i=1}^m X_i,
  \qquad m\geq1.
\end{align}
Thus $(\ga_m)_{m\geq0}$ are the jump times of a rate-$\be$ Poisson
process.  Write
\[
  \DN
  =\{0<\ga_1<\cdots<\ga_n\}.  
\]
We denote by $\mu_\be(d\ga)$ the joint law of
$(\ga_1,\dots,\ga_n)$, namely
\begin{align}\label{mu}
  \mu_\be(d\ga)
  =\be^n e^{-\be\ga_n}
  \ind_{\DN}(\ga)\,
  d\ga_1\cdots d\ga_n.
\end{align}

Let $\{B_i\}_{i\geq0}$ be independent two-sided standard Brownian
motions with $B_i(0)=0$.  For a measurable terminal profile
$h:\R\to\R$, a strength parameter $\al\in\R$, spatial location $x\in \mathbb{R}$, and a vertical level $n \in \mathbb{Z}_{>0}$, define the Hamiltonian as follows
\[
  H_x^{h,\al}(\ga)
  =
  \al h(x+\ga_n)
  +\sum_{i=0}^{n-1}
  \bigl(
    B_i(x+\ga_{i+1})-B_i(x+\ga_i)
  \bigr),
\]
and note that we suppress the vertical level $n$ from our notation, as our results and proofs hold for any fixed $n\geq 1$. 
The corresponding point-to-line partition function is
\begin{align}\label{Z_gen}
  Z_x^{h,\al,\be}
  =
  \be^{-n}\int_{\DN}
    e^{\be H_x^{h,\al}(\ga)}\,
    \mu_\be(d\ga).
\end{align}
And we note the partition function above is almost surely finite when $h$ is a scaled Brownian motion
by Lemma \ref{B_poly}. The associated quenched polymer measure is 
\begin{align*}
  Q_x^{h,\al,\be}(d\ga)
  =
  \frac{\be^{-n}e^{\be H_x^{h,\al}(\ga)}}
       {Z_x^{h,\al,\be}}
  \mu_\be(d\ga).
\end{align*}
Equivalently, $Q_x^{h,\al,\be}$ is a random change of measure of the
Poisson jump-time law.  It may be viewed as a measure on up-right paths
that start from $(x,0)$, make one vertical jump at each of the locations
$\ga_1,\dots,\ga_n$, and arrive at the terminal level $y=n$; see
Figure~\ref{fig_jump}.

For $x\in\R$ and $\ga\in\DN$, we write $x+\ga$ for the
path that starts from $(x,0)$ and passes through
$(x+\ga_i,i)$ for $1\leq i\leq n$.  For two such paths, define their
first meeting level by 
\[
  \tau_{\wt x,x}(\wt\ga,\ga)
  =
  \begin{cases}
    \displaystyle
    \min\bigl\{0\leq k\leq n-1:
      \wt x+\wt\ga_{k+1} >  x+\ga_k\bigr\},
      &\wt x<x,\\[6pt]
    \displaystyle
    \min\bigl\{0\leq k\leq n-1:
      x+\ga_{k+1}> \wt x+\wt\ga_k\bigr\},
      &x<\wt x,\\[6pt]
    0,&x=\wt x,
  \end{cases}
\]
with the convention that the minimum of the empty set is $\infty$.
Thus $\tau=\infty$ precisely when the two paths remain disjoint through
the terminal level.  We abbreviate
$\tau_{\wt x,x}(\wt\ga,\ga)$ to $\tau$ whenever the paths are clear
from the context.  See Figure~\ref{fig_tau}.

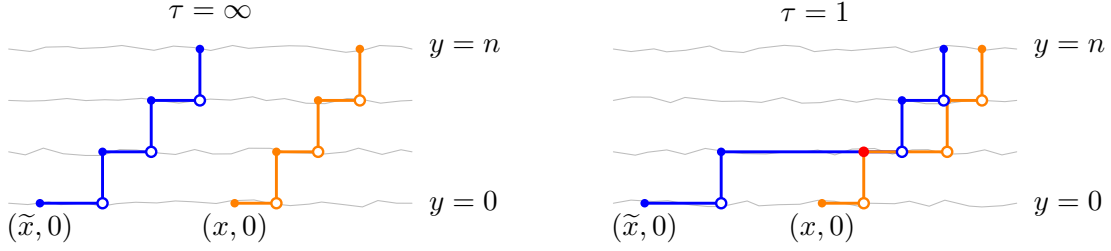
\begin{figure}[t]
  \centering
  \pgfmathsetseed{20260407}
  \begin{tikzpicture}[x=.92cm,y=.68cm,>=Stealth]

    \begin{scope}
      \foreach \yy in {0,1,2,3}{
        \draw[brownian background] (0,\yy) -- (5.8,\yy);
      }
      \node[right] at (5.9,0) {$y=0$};
      \node[right] at (5.9,3) {$y=n$};
      \node at (2.9,3.72) {$\tau=\infty$};

      \draw[bluepath]
        (.45,0) -- (1.35,0) -- (1.35,1)
        -- (2.05,1) -- (2.05,2)
        -- (2.75,2) -- (2.75,3);

      \draw[orange,line width=1.1pt]
        (3.25,0) -- (3.85,0) -- (3.85,1)
        -- (4.45,1) -- (4.45,2)
        -- (5.05,2) -- (5.05,3);

      \fill[blue]   (.45,0) circle (1.6pt);
      \fill[blue]   (1.35,1) circle (1.6pt);
      \fill[blue]   (2.05,2) circle (1.6pt);
      \fill[blue]   (2.75,3) circle (1.6pt);

      \fill[orange] (3.25,0) circle (1.6pt);
      \fill[orange] (3.85,1) circle (1.6pt);
      \fill[orange] (4.45,2) circle (1.6pt);
      \fill[orange] (5.05,3) circle (1.6pt);

      \draw[blue,fill=white,line width=.9pt]
        (1.35,0) circle (1.9pt);
      \draw[blue,fill=white,line width=.9pt]
        (2.05,1) circle (1.9pt);
      \draw[blue,fill=white,line width=.9pt]
        (2.75,2) circle (1.9pt);

      \draw[orange,fill=white,line width=.9pt]
        (3.85,0) circle (1.9pt);
      \draw[orange,fill=white,line width=.9pt]
        (4.45,1) circle (1.9pt);
      \draw[orange,fill=white,line width=.9pt]
        (5.05,2) circle (1.9pt);

      \node[below] at (.45,0) {$(\wt x,0)$};
      \node[below] at (3.25,0) {$(x,0)$};
    \end{scope}

    \begin{scope}[xshift=8.0cm]
      \foreach \yy in {0,1,2,3}{
        \draw[brownian background] (0,\yy) -- (5.8,\yy);
      }
      \node[right] at (5.9,0) {$y=0$};
      \node[right] at (5.9,3) {$y=n$};
      \node at (2.9,3.72) {$\tau=1$};

      \draw[bluepath]
        (.45,0) -- (1.55,0) -- (1.55,1)
        -- (4.15,1) -- (4.15,2)
        -- (4.75,2) -- (4.75,3);

      \draw[orange,line width=1.1pt]
        (3.00,0) -- (3.60,0) -- (3.60,1)
        -- (4.80,1) -- (4.80,2)
        -- (5.30,2) -- (5.30,3);

      \fill[blue]   (.45,0) circle (1.6pt);
      \fill[blue]   (1.55,1) circle (1.6pt);
      \fill[blue]   (4.15,2) circle (1.6pt);
      \fill[blue]   (4.75,3) circle (1.6pt);

      \fill[orange] (3.00,0) circle (1.6pt);
      \fill[orange] (4.80,2) circle (1.6pt);
      \fill[orange] (5.30,3) circle (1.6pt);

      \fill[red]    (3.60,1) circle (2.0pt);

      \draw[blue,fill=white,line width=.9pt]
        (1.55,0) circle (1.9pt);
      \draw[blue,fill=white,line width=.9pt]
        (4.15,1) circle (1.9pt);
      \draw[blue,fill=white,line width=.9pt]
        (4.75,2) circle (1.9pt);

      \draw[orange,fill=white,line width=.9pt]
        (3.60,0) circle (1.9pt);
      \draw[orange,fill=white,line width=.9pt]
        (4.80,1) circle (1.9pt);
      \draw[orange,fill=white,line width=.9pt]
        (5.30,2) circle (1.9pt);

      \node[below] at (.45,0) {$(\wt x,0)$};
      \node[below] at (3.00,0) {$(x,0)$};
    \end{scope}

  \end{tikzpicture}
  \captionsetup{width=.8\linewidth}
  \caption{
    An illustration of the first meeting level $\tau$.  The two
    paths in the left panel remain disjoint, whereas those in the
    right panel first meet at the red point on the line $y=1$.
    Open circles mark the endpoints of horizontal segments, and
    solid circles mark the starting points of the next horizontal
    segments after vertical jumps.
  }\label{fig_tau}
\end{figure}

\subsection{Main results}

We now specialize to the Brownian terminal profile $h=B_n$,  and abbreviate
\begin{align*}
  Z_x^{\al,\be}
  =Z_x^{B_n,\al,\be},
  \quad
  Q_x^{\al,\be}
  =Q_x^{B_n,\al,\be},
  \quad
  F_x^{\al,\be}
  =\be^{-1}\log Z_x^{\al,\be}.
\end{align*}
For arbitrary $\be>0$, we use the superscripts
``$\rfree,\be$'' and ``$\rstat,\be$'' for the cases $\al=0,1$,
respectively.

Throughout this paper, $\EE$ denotes expectation with respect to the
Brownian environment, while $Q_x^{\al,\be}$ denotes the quenched
polymer measure in a realization of the environment.  Let us define
$$\rmP_{\wt x,x}^{\al,\be}(\dots) = \mathbb{E}[Q_{\wt x}^{\al,\be}\otimes Q_x^{\al,\be} (\dots)],$$
and in particular,
\begin{align} \label{coal_prob}
  \rmP_{\wt x,x}^{\al,\be}(\tau\leq n-1)
  =
  \EE\left[
    \int_{\DN^2}
    \ind_{\{\tau\leq n-1\}}(\wt x + \wt\ga, x +\ga)
    Q_{\wt x}^{\al,\be}(d\wt\ga)
    Q_x^{\al,\be}(d\ga)
  \right].
\end{align}
This is the annealed probability that two independently sampled polymers starting at $x$ and $\wt x$ meet before reaching the terminal level $y=n$. Equivalently, rather than using the independent coupling, one may use a coupling analogous to the standard coupling of Markov chains, in which the two paths evolve independently until their first meeting and then coalesce; see \cite[Lemma~2.7]{Gin24}.
Lastly, we note that the random variable
$Q_0^{\al,\be}\{\ga_n>-x\}$
is the quenched probability that the terminal point $(\ga_n,n)$ of the polymer, which starts at $(0,0)$, lies to the right of $(-x,n)$.

\medskip

With the above definitions, we are ready to state our main result.

\begin{theorem}\label{main_thm}
  Fix $n\in\Z_{>0}$, $\al\in\R$, $\be>0$, and
  $x,\wt x\in\R$ with $x\neq\wt x$.  Then, 
  \begin{align}\label{main_eq}
    \frac{\partial}{\partial x}
    \EE\bigl[F_{\wt x}^{\al,\be}F_x^{\al,\be}\bigr]
    &=
    \frac{1-\al^2}{2}
    \sgn(\wt x-x)
    \rmP_{\wt x,x}^{\al,\be}(\tau\leq n-1)
    \notag\\
    &\qquad\qquad+
    \al^2\left(
      \EE\bigl[Q_0^{\al,\be}\{\ga_n>-x\}\bigr]
      -\ind_{\{\wt x<x\}}
    \right).
  \end{align}
\end{theorem}

\begin{remark}
  The restriction $x\neq\wt x$ is essential.  For some values of
  $\al$, the map
  $
    x\mapsto
    \EE[F_{\wt x}^{\al,\be}F_x^{\al,\be}]
  $
  is not differentiable at $x=\wt x$; see
  Remark~\ref{const_der} for a concrete example.
\end{remark}

When $\al=0$, the boundary contribution vanishes and the theorem gives
a direct coalescence identity.

\begin{corollary}[\textsf{Zero terminal condition}]\label{flat_cor}
  For $n\geq1$, $\be>0$, and $x\neq\wt x$,
  \begin{align}
    \rmP_{\wt x,x}^{\rfree,\be}(\tau\leq n-1)
    &=
    \frac{2}{\be^2}\sgn(\wt x-x)
    \frac{\partial}{\partial x}
    \EE\bigl[
      \log Z_{\wt x}^{\rfree,\be}\log Z_x^{\rfree,\be}
    \bigr]
    \label{flat_eq}
    \\
    &=
    \frac{2}{\be^2}\sgn(\wt x-x)
    \frac{\partial}{\partial x}
    \Cov\bigl(
      \log Z_{\wt x}^{\rfree,\be},
      \log Z_x^{\rfree,\be}
    \bigr)
    \notag\\
    &=
    -\frac{1}{\be^2}\sgn(\wt x-x)
    \frac{\partial}{\partial x}
    \EE\left[
      \bigl(
        \log Z_{\wt x}^{\rfree,\be}
        -\log Z_x^{\rfree,\be}
      \bigr)^2
    \right].
    \notag
  \end{align}
\end{corollary}

\begin{remark}
  The three formulas are equivalent because the free-terminal model is
  translation invariant in $x$.  In particular,
  $\EE[\log Z_x^{\rfree, \be}]$ and
  $\EE[(\log Z_x^{\rfree, \be})^2]$ do not depend on $x$.
\end{remark}

When $\al=1$, the coalescence coefficient vanishes and the theorem
reduces to an identity for the location of the terminal polymer point.

\begin{corollary}[\textsf{Stationary terminal condition}]\label{stat_exit}
  For $n\geq1$, $\be>0$, and $x\neq\wt x$,
  \begin{align}\label{stat_exit_eq}
    \frac{\partial}{\partial x}
    \EE\bigl[
      F_{\wt x}^{\rstat,\be}F_x^{\rstat,\be}
    \bigr]
    =
    \EE\bigl[Q_0^{\rstat,\be}\{\ga_n>-x\}\bigr]
    -\ind_{\{\wt x<x\}}.
  \end{align}
\end{corollary}

\begin{remark}\label{const_der}
  If $x>0$, then $\ga_n>-x$ holds identically, and therefore
$    \EE\bigl[Q_0^{\rstat,\be}\{\ga_n>-x\}\bigr]=1.$  Corollary~\ref{stat_exit} consequently gives
  \[
    \frac{\partial}{\partial x}
    \EE\bigl[
      F_{\wt x}^{\rstat,\be}F_x^{\rstat,\be}
    \bigr]
    =
    \begin{cases}
      1,&0<x<\wt x,\\[3pt]
      0,&x>\wt x\vee0.
    \end{cases}
  \]
  Thus, when $\wt x>0$, the derivative is piecewise constant and has a
  jump at $x=\wt x$.
\end{remark}

\begin{remark} 
  The classical Burke property for the stationary $\OCY$ polymer \cite{OCY01} provides a check of the vanishing derivative in Remark~\ref{const_der}. We assume the standard normalization $\be=1$. 

Fix $\wt x\in\R$ and $x>\wt x\vee0$. By the Burke property along down-right paths, the horizontal free-energy increment $\log Z_{x+\delta}^{\rstat} - \log Z_x^{\rstat}$ (for $\delta > 0$) is distributed as a centered Brownian increment. Furthermore, it is independent of the spatial history to its left and above, which includes $\log Z_{\wt x}^{\rstat}$. Consequently, taking the expectation yields
\begin{equation}\label{burke_cov_inc}
  \EE\left[
    \log Z_{\wt x}^{\rstat}
    \bigl(
      \log Z_{x+\delta}^{\rstat}
      -\log Z_x^{\rstat}
    \bigr)
  \right] = \EE\left[
    \log Z_{\wt x}^{\rstat}\right]
    \EE\left[
      \log Z_{x+\delta}^{\rstat}
      -\log Z_x^{\rstat}
  \right]
  = 0.
\end{equation}
This immediately implies that $x \mapsto \EE\left[ \log Z_{\wt x}^{\rstat}\log Z_x^{\rstat} \right]$ is constant on $(\wt x\vee0,\infty)$, recovering the zero derivative from Corollary~\ref{stat_exit}.
\end{remark}

The exit time identity in Corollary \ref{stat_exit} above admits the following nonlinear extension.
It is a  Stein identity for multivariate normal 
and leads directly to
the Brownian law of the horizontal stationary free-energy profile.

\begin{proposition}[\textsf{Stationary Stein identity}]
  \label{stat_stein}
  Fix $n\geq1$, $\be>0$, and $m\geq0$.  Let
  $\mathbf y=(y_0,\ldots,y_m)\in\R^{m+1}$ and set
  \[
    \mathbf F_{\mathbf y}^{\be}
    =
    \bigl(
      F_{y_0}^{\rstat,\be},\ldots,
      F_{y_m}^{\rstat,\be}
    \bigr).
  \]
  Suppose that $\varphi\in C^1(\R^{m+1})$ and
  $\|\nabla\varphi\|_\infty<\infty$.  Then the map
  $
    x\mapsto
    \EE\bigl[
      \varphi(\mathbf F_{\mathbf y}^{\be})
      F_x^{\rstat,\be}
    \bigr]
  $
  is continuously differentiable on every connected component of
  $\R\setminus\{y_0,\ldots,y_m\}$.  Moreover, for every
  $x\notin\{y_0,\ldots,y_m\}$,
  \begin{align}\label{stat_stein_eq}
    &\frac{\partial}{\partial x}
    \EE\bigl[
      \varphi(\mathbf F_{\mathbf y}^{\be})
      F_x^{\rstat,\be}
    \bigr]
  =
    \sum_{j=0}^{m}
    \EE\left[
      \partial_j\varphi(\mathbf F_{\mathbf y}^{\be})
      \left(
        Q_x^{\rstat,\be}\{\ga_n>-x\}
        -\ind_{\{x>y_j\}}
      \right)
    \right].
  \end{align}
\end{proposition}

\begin{remark} \label{rem19}
  The quenched terminal exit probability in
  \eqref{stat_stein_eq} is based at the variable point $x$.
  Inside the weighted expectation, it cannot in general be replaced by
  $Q_0^{\rstat,\be}\{\ga_n>-x\}$.  When $m=0$ and
  $\varphi(r)=r$, however, spatial stationarity gives
  \[
    \EE\bigl[Q_x^{\rstat,\be}\{x+ \ga_n> 0\}\bigr]
    =
    \EE\bigl[Q_0^{\rstat,\be}\{\ga_n>-x\}\bigr],
  \]
  where the equality follows by shifting the environment and using
  stationarity of its relative increment fields;
  see the beginning of Appendix \ref{moll_app}
  for more details.
  Consequently,
  Proposition~\ref{stat_stein} reduces to
  Corollary~\ref{stat_exit}.

\end{remark}

Corollary~\ref{buse_cor} below gives a direct Gaussian
integration-by-parts proof of the Brownian law of the horizontal
stationary profile. 
\begin{corollary}[\textsf{Brownian Busemann process}]
  \label{buse_cor}
  Fix $n\geq1$ and $\be>0$, and define the anchored horizontal
  stationary free-energy profile by
  \begin{align}\label{buse_eq}
    \cB^{\be}(x)
    =
    F_x^{\rstat,\be}-F_0^{\rstat,\be},
    \qquad x\in\R.
  \end{align}
  Then $(\cB^{\be}(x))_{x\in\R}$ is a two-sided standard Brownian
  motion.  More precisely, for every $m\geq1$ and every
  $x_0<x_1<\cdots<x_m$, the increments
   $\widehat\Delta_k
    =
    \cB^{\be}(x_k)-\cB^{\be}(x_{k-1})$,
  $1\leq k\leq m$,
  are independent with
    $\widehat\Delta_k\sim \text{Normal}(0,x_k-x_{k-1}).$
  Equivalently,
   $ x\in\R\mapsto
    \log Z_x^{\rstat,\be}-\log Z_0^{\rstat,\be}$
  is a two-sided Brownian motion with variance parameter $\be^2$.
\end{corollary}

The stationary coalescence probability can be recovered by
differentiating the main identity with respect to $\al$ at $\al=1$.

\begin{corollary}[\textsf{Stationary coalescence identity}]\label{stat_coal}
  Fix $n\geq1$, $\be>0$, and $0\leq\wt x<x$.  Then
  \begin{align*}
    \rmP_{\wt x,x}^{\rstat,\be}(\tau\leq n-1)
    &=
    \left.
      \frac{1}{\be^2}
      \frac{\partial}{\partial\al}
      \frac{\partial}{\partial x}
      \EE\bigl[
        \log Z_{\wt x}^{\al,\be}\log Z_x^{\al,\be}
      \bigr]
    \right|_{\al=1}
    \\
    &=
    2-
    \left.
      \frac{1}{\be^2}
      \frac{\partial}{\partial\al}
      \frac{\partial}{\partial\wt x}
      \EE\bigl[
        \log Z_{\wt x}^{\al,\be}\log Z_x^{\al,\be}
      \bigr]
    \right|_{\al=1}.
  \end{align*}
\end{corollary}

To conclude the main-results section, we discuss the zero-temperature
limit.  In this limit, the $\OCY$ polymer converges to Brownian
last-passage percolation (BLPP).  The free energy
is then replaced by a last-passage value, and the quenched polymer
measure concentrates on maximizing paths. 

Define 
$\overline{\DN}
  =\{0\leq\ga_1\leq\cdots\leq\ga_n\}.$
For $z\in\R$ and $\ga\in\overline{\DN}$, define
\[
  \cH_z^{\al}(\ga)
  =
  \al B_n(z+\ga_n)
  +\sum_{i=0}^{n-1}
  \bigl(
    B_i(z+\ga_{i+1})-B_i(z+\ga_i)
  \bigr)
  -\ga_n,
\]
and set
\begin{align*}
  L_z^{\al}
  =
  \sup_{\ga\in\overline{\DN}}
  \cH_z^{\al}(\ga).
\end{align*}
For any fixed $z$, the maximizer is almost surely unique; we denote it by 
\[
  \Gamma_z^{\al}
  =
  \bigl(
    \Gamma_z^{\al}(1),\dots,\Gamma_z^{\al}(n)
  \bigr)
  \in\overline{\DN}
\]
and call it the BLPP geodesic from $(z,0)$ to level $n$. Lastly, we may sometimes include the starting point of the geodesic by defining  $\Gamma_z^{\al}(0) = z$.

To match the definition of polymers, let us define 
$$\rmP_{\wt x,x}^{\al,\infty}
    (\tau\leq n-1) = \mathbb{P}\bigg(\bigcup_{i=0}^{n-1}\Big\{[\Gamma^\alpha_{\wt x}(i), \Gamma^\alpha_{\wt x}(i+1)]\cap  [\Gamma^\alpha_{x}(i), \Gamma^\alpha_{x}(i+1)] \neq \emptyset \Big\} \bigg),$$ 
    or in words, this is the the probability of two geodesics $\Gamma^\alpha_{\wt x}$ and $\Gamma^\alpha_x$ coalescing before level $n$. In addition, let us define 
    $$\rmP_0^{\al,\infty}(\ga_n>-x) = \mathbb{P}(\Gamma^\alpha_0(n)> -x).$$ With these definition, we will state our main result for the Brownian LPP, and its proof is given in Section~\ref{zero_sec}.
\begin{corollary}[\textsf{Zero-temperature BLPP identity}]\label{blpp_cor}
  Fix $n\in\Z_{>0}$, $\al\in\R$, and $x,\wt x\in\R$ with
  $x\neq\wt x$.  Then,
  \begin{align*}
    \frac{\partial}{\partial x}
    \EE\bigl[L_{\wt x}^{\al}L_x^{\al}\bigr]
    &=
    \frac{1-\al^2}{2}
    \sgn(\wt x-x)
    \rmP_{\wt x,x}^{\al,\infty}
    (\tau\leq n-1)
  +
    \al^2
    \rmP_0^{\al,\infty}(\ga_n>-x)
    -\frac{\al^2}{2}
    \bigl(1-\sgn(\wt x-x)\bigr).
  \end{align*}
\end{corollary}

\begin{remark}[Possible use of integrable formulas]
  Corollary~\ref{blpp_cor} expresses BLPP coalescence and terminal-exit
  probabilities through the spatial derivative of a two-point
  last-passage covariance.  Brownian LPP admits exact finite-dimensional
  formulas, including Johansson's multi-time formulas and Rahman's
  formulas with functional initial data
  \cite{Johansson17,Rahman24}.  These formulas do not directly cover the
  point-to-line Brownian-boundary covariance that appears here.  An
  explicit evaluation of the probabilities in
  Corollary~\ref{blpp_cor} would therefore require either a suitable
  point-to-line multi-point formula or an extension of the present
  identity to a geometry compatible with the available integrable
  formulas.  We leave this problem for future work.
\end{remark}


\subsection{Proof strategy}

We now briefly discuss the main steps of the proof.

\smallskip
\noindent
(1) \emph{Mollification.}
We regularize each Brownian environment by convolution with a standard
mollifier.  For this schematic discussion, we use the $\be=1$
notation.  The mollified partition function is differentiable in the
spatial variable, and
\begin{align}\label{eq_pre}
  \frac{\partial}{\partial x}
  \log Z_x^{\al,\eps}
  =
  \int_{\DN} Q_x^{\al,\eps}(d\ga)
  \bigg[
    \al\xi_n^\eps(x+\ga_n)
    +\sum_{i=0}^{n-1}
    \bigl(
      \xi_i^\eps(x+\ga_{i+1})
      -\xi_i^\eps(x+\ga_i)
    \bigr)
  \bigg],
\end{align}
where $\xi_i^\eps$ is the mollification of the white noise $\xi_i$
 on level $i$.

\smallskip
\noindent
(2) \emph{Gaussian integration by parts.}
We apply the  integration-by-parts formula in Malliavin calculus
to $\xi_i$ in \eqref{eq_pre}.  For the two-point free-energy
correlation, the result separates naturally into a bulk contribution
from the levels $0,\dots,n-1$ and a boundary contribution from the
terminal Brownian motion on level $n$.

\smallskip
\noindent
(3) \emph{Cancellation and explicit kernels.}
When the Malliavin derivative acts on the Gibbs density, additional
terms appear.  The bulk and boundary cancellation identities proved in
Appendix~\ref{cancel_app} show that these terms vanish after
integration in the white-noise variable.  The remaining expressions
are written in terms of
\[
  \Psi^\eps
  =
  \Phi^\eps\ast\phi^\eps,
  \qquad
  \Phi^\eps(r)
  =
  \int_0^r\phi^\eps(s)\,ds.
\]
The associated integration domains have a direct path interpretation:
the increment on level $i$ contributes the interval
$[x+\ga_i,x+\ga_{i+1}]$, while the terminal Brownian term contributes
the interval joining $0$ and $x+\ga_n$.

\smallskip
\noindent
(4) \emph{Removal of the mollification.}
As $\eps\downarrow0$, the kernel $\Psi^\eps$ converges, away from the
origin, to $\frac12\sgn$.  Appendix~\ref{moll_app} proves the
required local uniform convergence of the mollified polymer measures
and provides the estimates needed to pass to the limit away from the
diagonal $x=\wt x$.

\smallskip
\noindent
(5) \emph{Meeting-point symmetry.}
After the preceding steps, the bulk term is a signed functional of two
independent replicas.  Condition on their first meeting level, the
memoryless property of the exponential waiting times makes the two
continuations beyond the meeting point conditionally exchangeable.
Every contribution that is antisymmetric under interchange of these
continuations therefore has conditional mean zero.  The surviving term
is exactly the indicator that the two paths meet, which yields the
coalescence probability.  The analogous symmetry in the boundary term
leaves the one-point terminal-location of the polymer.

In the stationary case, the same calculation can be carried out with a
nonlinear test function of finitely many free energies. 
The Malliavin chain rule yields Proposition~\ref{stat_stein}.  
When the test function depends
only on consecutive spatial increments, the common terminal exit term
cancels telescopically, leading to the classical multivariate Gaussian
Stein identity.  This proves Corollary~\ref{buse_cor}.

\medskip 
\noindent\textbf{Frequently used notation and conventions.} The level $n$ is fixed throughout most of the paper and is suppressed
from the principal polymer notation.  The following conventions are
used repeatedly.

\begin{description}[leftmargin=4.45cm,labelwidth=4.10cm,
  style=multiline,itemsep=2pt,topsep=4pt]
  \item[$\R_+,\DN,\overline{\DN}$]
    $\R_+=[0,\infty)$,
    $\DN=\{0<\ga_1<\cdots<\ga_n\}$, and
    $\overline{\DN}=\{0\leq\ga_1\leq\cdots\leq\ga_n\}$.

  \item[$\al,\be,\eps$]
    The terminal-condition strength, inverse temperature, and
    mollification parameter, respectively.

  \item[$B_i,\xi_i,B_i^\eps,\xi_i^\eps$]
    The two-sided Brownian environments, their white-noise derivatives,
    and the corresponding mollified objects; see
    \eqref{not_moll}.

  \item[$\ga,\wt\ga,\eta,\mu_\be$]
    Polymer jump-location vectors, replicas of such vectors, and the
    joint law of the first $n$ jump times of a rate-$\be$ Poisson
    process.  We use $\ga_0=\wt\ga_0=\eta_0=0$.

  \item[$h;\ H,Z,\cZ,Q,F$]
    The terminal profile, Hamiltonian, normalized partition function,
    unnormalized partition function, quenched polymer measure, and
    free energy.  For a general terminal profile, we write
    $Z_x^{h,\al,\be}$ and $Q_x^{h,\al,\be}$; when $h=B_n$, the
    superscript $h$ is omitted.  When the deterministic factor
    $\be^{-n}$ is suppressed, we write
    $\cZ_x^{h,\al,\be}=\be^n Z_x^{h,\al,\be}$.  The full mollified
    notation is $H_x^{\al,\eps}$, $Z_x^{\al,\be,\eps}$,
    $\cZ_x^{\al,\be,\eps}=\be^n Z_x^{\al,\be,\eps}$, and
    $Q_x^{\al,\be,\eps}$, while
    $F_x^{\al,\be}=\be^{-1}\log Z_x^{\al,\be}$.  The value $\eps=0$
    is omitted, and $\be=1$ is likewise suppressed unless retaining it
    improves clarity.

  \item[$\rfree,\rstat$]
    The shorthand for $\al=0$ and $\al=1$, respectively; when $\be=1$, the parameter $\be$
    is omitted.  At general
    inverse temperature we write, for example,
    $F_x^{\rstat,\be}=F_x^{1,\be}$.  The same convention applies to $Z,Q$, and $\rmP$.

  \item[$\cB^{\be}$]
    The anchored horizontal stationary free-energy profile
    $\cB^{\be}(x)=F_x^{\rstat,\be}-F_0^{\rstat,\be}$; see
    \eqref{buse_eq}.

  \item[$\tau,\fB_k,\fB_\infty$]
    The first meeting level of two replicas and the corresponding
    first-meeting events $\fB_k=\{\tau=k\}$ and
    $\fB_\infty=\{\tau=\infty\}$.

  \item[$\rmP_{y,x}^{\al,\be}$]
    The averaged two-replica law: the Brownian environment is sampled
    first, followed by two independent paths under
    $Q_y^{\al,\be}\otimes Q_x^{\al,\be}$; see
    \eqref{coal_prob},
    \eqref{not_exp}.

  \item[$P_i,\Xi_{i,\eps}^{y,x},\mathfrak D_i^{y,x}$]
    The spatial coordinate on level $i$, the mollified four-point
    kernel, and its sign-kernel limit; see
    \eqref{not_path}.  The fraktur symbol
    $\mathfrak D$ is reserved for this path kernel and is distinct from
    the Malliavin derivative $D$.

  \item[$L_x^{\al},\Gamma_x^{\al}$]
    The zero-temperature BLPP passage value and its unique maximizing
    geodesic.  The corresponding quenched and averaged laws carry the
    superscript $\infty$, as in $Q_x^{\al,\infty}$ and
    $\rmP_{y,x}^{\al,\infty}$; see
    \eqref{not_zero}.

  \item[$D_{i,y},\EE,\PP,\ind$]
    The Malliavin derivative at level $i$ and position $y$, expectation,
    probability, and the indicator function.  Unless another law is
    specified, $\EE$ and $\PP$ refer to the Brownian environment.
\end{description}

For ease of reference, we collect the recurrent mollification notation
in one display.  Fix a nonnegative even function
$\phi\in C_c^\infty(\R)$ with $\int_\R\phi=1$.  Then
\begin{align}\label{not_moll}
  \begin{aligned}
    \phi^\eps(r)
      &=\eps^{-1}\phi(r/\eps),
    &
    \xi_i^\eps(x)
      &=\int_\R \phi^\eps(x-y)\,\xi_i(dy),
    \\
    B_i^\eps(t)
      &=\int_0^t\xi_i^\eps(s)\,ds,
    &
    \Phi^\eps(r)
      &=\int_0^r\phi^\eps(s)\,ds,
    \AND
    \Psi^\eps
      =\Phi^\eps\ast\phi^\eps.
  \end{aligned}
\end{align}
For $z\in\R$, $\ga\in\overline{\DN}$, and $0\leq i\leq n$, set
$P_i(z,\ga)=z+\ga_i$.  For $0\leq i\leq n-1$, define
\begin{align}\label{not_path}
  \begin{aligned}
    \Xi_{i,\eps}^{y,x}(\eta,\ga)
      &={}
      \Psi^\eps\bigl(P_{i+1}(y,\eta)-P_{i+1}(x,\ga)\bigr)
      -\Psi^\eps\bigl(P_i(y,\eta)-P_{i+1}(x,\ga)\bigr)
    \\
      &\qquad \quad
      -\Psi^\eps\bigl(P_{i+1}(y,\eta)-P_i(x,\ga)\bigr)
      +\Psi^\eps\bigl(P_i(y,\eta)-P_i(x,\ga)\bigr),
    \\
    \mathfrak D_i^{y,x}(\eta,\ga)
      &={}
      \sgn\bigl(P_{i+1}(y,\eta)-P_{i+1}(x,\ga)\bigr)
      -\sgn\bigl(P_i(y,\eta)-P_{i+1}(x,\ga)\bigr)
    \\
     &\qquad \quad
      -\sgn\bigl(P_{i+1}(y,\eta)-P_i(x,\ga)\bigr)
      +\sgn\bigl(P_i(y,\eta)-P_i(x,\ga)\bigr).
  \end{aligned}
\end{align}
Geometrically, $\mathfrak D_i^{y,x}(\eta,\ga)$ measures the oriented
interlacing of the two horizontal segments
\[
  [P_i(y,\eta),P_{i+1}(y,\eta)]
  \quad\text{and}\quad
  [P_i(x,\ga),P_{i+1}(x,\ga)].
\]
Away from endpoint equalities,
\[
  \mathfrak D_i^{y,x}(\eta,\ga)\in\{-2,0,2\},
\]
with value $-2$ in the ordering
\[
  P_i(y,\eta)<P_i(x,\ga)<P_{i+1}(y,\eta)<P_{i+1}(x,\ga),
\]
value $2$ in the reversed ordering
\[
  P_i(x,\ga)<P_i(y,\eta)<P_{i+1}(x,\ga)<P_{i+1}(y,\eta),
\]
and value $0$ otherwise.
See Figure \ref{fig:D_geom} for an illustration.

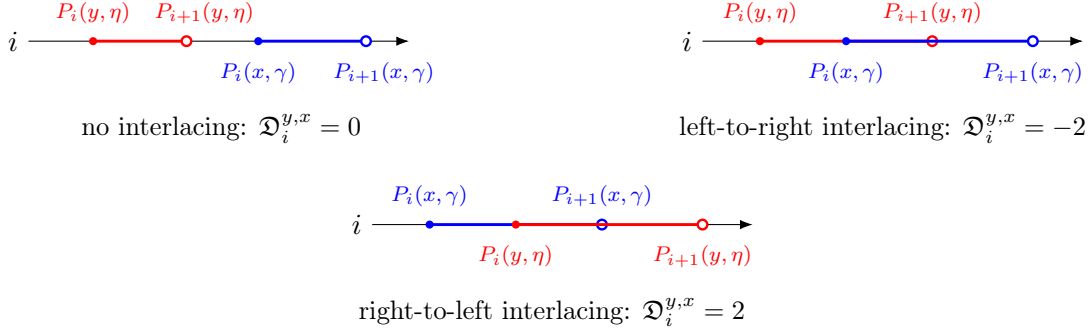
\begin{figure}[t]
\centering

\begin{minipage}{0.47\textwidth}
\centering
\begin{tikzpicture}[scale=0.95]
  \draw[->] (0,0) -- (5.3,0);
  \node[left] at (0,0) {$i$};

  \draw[red, line width=1.1pt] (0.9,0) -- (2.2,0);
  \fill[red] (0.9,0) circle (1.6pt);
  \draw[red,fill=white,line width=.9pt] (2.2,0) circle (1.9pt);

  \draw[blue, line width=1.1pt] (3.2,0) -- (4.7,0);
  \fill[blue] (3.2,0) circle (1.6pt);
  \draw[blue,fill=white,line width=.9pt] (4.7,0) circle (1.9pt);

  \node[red, above=3pt, font=\scriptsize] at (0.9,0) {$P_i(y,\eta)$};
  \node[red, above=3pt, font=\scriptsize] at (2.45,0) {$P_{i+1}(y,\eta)$};
  \node[blue, below=4pt, font=\scriptsize] at (3.2,0) {$P_i(x,\ga)$};
  \node[blue, below=4pt, font=\scriptsize] at (4.95,0) {$P_{i+1}(x,\ga)$};
\end{tikzpicture}

\vspace{1mm}
{\small no interlacing: $\mathfrak D_i^{y,x}=0$}
\end{minipage}
\hfill
\begin{minipage}{0.47\textwidth}
\centering
\begin{tikzpicture}[scale=0.95]
  \draw[->] (0,0) -- (5.3,0);
  \node[left] at (0,0) {$i$};

  \draw[red, line width=1.1pt] (0.8,0) -- (3.2,0);
  \fill[red] (0.8,0) circle (1.6pt);
  \draw[red,fill=white,line width=.9pt] (3.2,0) circle (1.9pt);

  \draw[blue, line width=1.1pt] (2.0,0) -- (4.6,0);
  \fill[blue] (2.0,0) circle (1.6pt);
  \draw[blue,fill=white,line width=.9pt] (4.6,0) circle (1.9pt);

  \node[red, above=3pt, font=\scriptsize] at (0.8,0) {$P_i(y,\eta)$};
  \node[red, above=3pt, font=\scriptsize] at (3.2,0) {$P_{i+1}(y,\eta)$};
  \node[blue, below=4pt, font=\scriptsize] at (2.0,0) {$P_i(x,\ga)$};
  \node[blue, below=4pt, font=\scriptsize] at (4.6,0) {$P_{i+1}(x,\ga)$};
\end{tikzpicture}

\vspace{1mm}
{\small left-to-right interlacing: $\mathfrak D_i^{y,x}=-2$}
\end{minipage}

\vspace{4mm}

\makebox[\textwidth][c]{%
\begin{minipage}{0.47\textwidth}
\centering
\begin{tikzpicture}[scale=0.95]
  \draw[->] (0,0) -- (5.3,0);
  \node[left] at (0,0) {$i$};

  \draw[blue, line width=1.1pt] (0.8,0) -- (3.2,0);
  \fill[blue] (0.8,0) circle (1.6pt);
  \draw[blue,fill=white,line width=.9pt] (3.2,0) circle (1.9pt);

  \draw[red, line width=1.1pt] (2.0,0) -- (4.6,0);
  \fill[red] (2.0,0) circle (1.6pt);
  \draw[red,fill=white,line width=.9pt] (4.6,0) circle (1.9pt);

  \node[blue, above=3pt, font=\scriptsize] at (0.8,0) {$P_i(x,\ga)$};
  \node[blue, above=3pt, font=\scriptsize] at (3.2,0) {$P_{i+1}(x,\ga)$};
  \node[red, below=4pt, font=\scriptsize] at (2.0,0) {$P_i(y,\eta)$};
  \node[red, below=4pt, font=\scriptsize] at (4.6,0) {$P_{i+1}(y,\eta)$};
\end{tikzpicture}

\vspace{1mm}
{\small right-to-left interlacing: $\mathfrak D_i^{y,x}=2$}
\end{minipage}}

\caption{Geometric meaning of $\mathfrak D_i^{y,x}(\eta,\ga)$ on level $i$. In each panel, the solid point marks the starting point \(P_i\) of a horizontal segment, and the open circle marks its endpoint \(P_{i+1}\).}
\label{fig:D_geom}
\end{figure}

Thus, away from the four equality hyperplanes in
\eqref{not_path},
$\Xi_{i,\eps}^{y,x}(\eta,\ga)\to
\frac12\mathfrak D_i^{y,x}(\eta,\ga)$ as $\eps\downarrow0$.
When the starting points and the two underlying paths are fixed, we use
the shorter forms
\begin{align}\label{not_pathb}
  \Xi_{i,\eps}
  \bigl(P_i(y,\eta),P_i(x,\ga)\bigr)
  =\Xi_{i,\eps}^{y,x}(\eta,\ga),
  \qquad
  \mathfrak D_i
  \bigl(P_i(y,\eta),P_i(x,\ga)\bigr)
  =\mathfrak D_i^{y,x}(\eta,\ga).
\end{align}
The subscript $i$ records that the next-level coordinates
$P_{i+1}(y,\eta)$ and $P_{i+1}(x,\ga)$ enter these shorthand
expressions.

Two expectation conventions used in the proofs are
\begin{align}\label{not_exp}
  \begin{aligned}
    \EE_{\eta,\ga}^{\mu_\be\times2}[G]
      &=\int_{\DN^2} G(\eta,\ga)\,
        \mu_\be(d\eta)\mu_\be(d\ga),
    \\
    \EE_{(y;\eta),(x;\ga)}^{\mathrm{poly}\times2}[G]
      &=\int_{\DN^2} G(\eta,\ga)\,
        Q_y^{\al,\be}(d\eta)Q_x^{\al,\be}(d\ga),
    \\
    \rmP_{y,x}^{\al,\be}(A)
      &=\EE\left[
        \EE_{(y;\eta),(x;\ga)}^{\mathrm{poly}\times2}
        [\ind_A]
      \right].
  \end{aligned}
\end{align}
The second expectation in \eqref{not_exp} is quenched,
and hence is itself a random variable.
When the parameter $\be$ is suppressed in the bulk proof, we write
$\mu=\mu_\be$ and $\EE_{\wt\ga,\ga}^{\mu\times2}$ for the
first expectation.  Primed path variables and other starting points
are handled by the same convention.

Finally, the zero-temperature notation is
\begin{align}\label{not_zero}
    \cH_z^{\al}(\ga)
      &=H_z^{\al}(\ga)-\ga_n,
    \quad
    L_z^{\al}
      =\sup_{\ga\in\overline{\DN}}\cH_z^{\al}(\ga),
    \quad
    \Gamma_z^{\al}
      =\mathop{\rm argmax}_{\ga\in\overline{\DN}}
        \cH_z^{\al}(\ga),
   \AND
    Q_z^{\al,\infty}
      =\delta_{\Gamma_z^{\al}}.
\end{align}
For two sets,  $U\Subset V$ means that $U$ is a compact subset of $V$.

\medskip

For the reader's convenience, we indicate where the proofs of the
preceding results are located.  The proof of Theorem~\ref{main_thm}
occupies Sections~\ref{basic_sec}--\ref{main_done}, with the final
assembly of the bulk and boundary contributions carried out in
Section~\ref{main_done}.  The proofs of
Corollaries~\ref{flat_cor}, \ref{stat_exit}, and \ref{stat_coal}
are also given in Section~\ref{main_done}.  Proposition~\ref{stat_stein}
is proved in   Section \ref{stein_sec} on the stationary Stein identity,
and Corollary~\ref{buse_cor} in Section \ref{buse_sec} on the Brownian
Busemann process.  Finally, the proof of the zero-temperature
Corollary~\ref{blpp_cor} is given in Section \ref{zero_sec}.
We collect all technical and auxiliary  results in the Appendices.

\section{Proofs of the main results}\label{proof_sec}

\subsection{Mollified partition functions and the basic decomposition}
\label{basic_sec}

We begin with the proof of Theorem~\ref{main_thm}.  Throughout
Subsections~\ref{basic_sec}--\ref{main_done},
we fix $\al$ and omit it from superscripts.  Thus, for example,
$H_x^\eps=H_x^{\al,\eps}$ and
$Q_x^{\be,\eps}=Q_x^{\al,\be,\eps}$.

For the Gaussian integration-by-parts calculation,
it is convenient to
suppress the deterministic factor $\be^{-n}$
in the partition function.
Accordingly, within these subsections, we write
\begin{align}\label{moll_Q}
  \begin{aligned}
    H_x^\eps(\ga)
      &={}
      \al B_n^\eps(x+\ga_n)
      +\sum_{i=0}^{n-1}
        \bigl(
          B_i^\eps(x+\ga_{i+1})-B_i^\eps(x+\ga_i)
        \bigr),
    \\
    \cZ_x^{\be,\eps}
      &=\int_{\DN}e^{\be H_x^\eps(\ga)}\,\mu_\be(d\ga)
       =\be^n Z_x^{\al,\be,\eps},
    \AND
    Q_x^{\be,\eps}(d\ga)
      =\frac{e^{\be H_x^\eps(\ga)}}
        {\cZ_x^{\be,\eps}}\,\mu_\be(d\ga).
  \end{aligned}
\end{align}
The polymer measure $Q$ is unchanged by this normalization.  

The mollified fields and kernels in
\eqref{not_moll} satisfy
\begin{align*}
  B_i^\eps(t)
  =\int_\R\phi^\eps(r)
    \bigl(B_i(t-r)-B_i(-r)\bigr)\,dr.
\end{align*}
In particular, $x\mapsto H_x^\eps(\ga)$ is smooth.  Lemma~\ref{pf_uxeps1}
justifies differentiation under the $\mu_\be$-integral and gives,
almost surely,

\noi
\begin{align}\label{sp_der}
  \begin{aligned}
    \frac{\partial}{\partial x}
     \Big(  \frac1\be\log \cZ_x^{\be,\eps} \Big)
    &=    \int_{\DN} Q_x^{\be,\eps}(d\ga)
    \bigg[   \al\xi_n^\eps(x+\ga_n)
      +\sum_{i=0}^{n-1}
      \bigl(
        \xi_i^\eps(x+\ga_{i+1})
        -\xi_i^\eps(x+\ga_i)
      \bigr)
    \bigg].
  \end{aligned}
\end{align}
We separate the terminal and bulk contributions by setting
\begin{align}\label{T12}
  \begin{aligned}
    T_{1,\eps}(x,\wt x)
      &={}
      \EE\bigg[
        \Big(  \log \cZ_{\wt x}^{\be,\eps} \Big)
        \int_{\DN} Q_x^{\be,\eps}(d\ga)\,
          \al\xi_n^\eps(x+\ga_n)
      \bigg],
    \\
    T_{2,\eps}(x,\wt x)
      &={}
      \EE\bigg[
        \Big(  \log \cZ_{\wt x}^{\be,\eps} \Big)
        \int_{\DN} Q_x^{\be,\eps}(d\ga)
        \sum_{i=0}^{n-1}
        \bigl(
          \xi_i^\eps(x+\ga_{i+1})
          -\xi_i^\eps(x+\ga_i)
        \bigr)
      \bigg].
  \end{aligned}
\end{align}
For $x\neq\wt x$, Proposition~\ref{app_der} and
\eqref{sp_der} imply
\begin{align}
  \begin{aligned}
    \frac{\partial}{\partial x}
    \EE\bigl[
      F_{\wt x}^{\al,\be}F_x^{\al,\be}
    \bigr]
    =\lim_{\eps\downarrow0}
      \left
      \{
        \frac1\be T_{1,\eps}(x,\wt x)
        +\frac1\be T_{2,\eps}(x,\wt x)
      \right\}.
  \end{aligned}
\end{align}
It remains to identify the two limits on the right-hand side.

\subsection{The bulk term}

\begin{proposition}\label{bulk}
  For $x\neq\wt x$, we have 
  $
    \frac{1}{\be}T_{2,\eps}(x,\wt x)
    \xrightarrow{\eps\downarrow0}
    \frac12\sgn(\wt x-x)\rmP_{\wt x,x}^{\be}\bigl(\tau\le n-1\bigr).
  $
\end{proposition}

\begin{proof}
  For brevity, set 
  $G_x^\eps(\ga) = \frac{e^{\be H_x^\eps(\ga)}}{\cZ_x^{\be,\eps}}$
      and for $i\in\{0, \ldots, n-1\}$,

      \noi
  \begin{align*}
    a_{i,\eps}^{x,\ga}(y)
      &=
      \phi^\eps\bigl(P_{i+1}(x,\ga)-y\bigr)
      -\phi^\eps\bigl(P_i(x,\ga)-y\bigr),
      \\
    A_{i,\eps}^{z,\eta}(y)
      &={}
      \Phi^\eps\bigl(P_{i+1}(z,\eta)-y\bigr)
      -\Phi^\eps\bigl(P_i(z,\eta)-y\bigr).
  \end{align*}
  By \eqref{not_moll} and \eqref{MD_H},
  we have 
  $D_{i,y}H_z^\eps(\eta)
      =A_{i,\eps}^{z,\eta}(y)$
      and
  \begin{align}\label{bulk_der}
    \begin{aligned}
      \xi_i^\eps\bigl(P_{i+1}(x,\ga)\bigr)
      -\xi_i^\eps\bigl(P_i(x,\ga)\bigr)
      =\int_\R a_{i,\eps}^{x,\ga}(y)\,\xi_i(dy).
    \end{aligned}
  \end{align}

 From the term $T_{2,\eps}$ in \eqref{T12}, using
  \eqref{bulk_der}, Fubini's theorem, and Gaussian
  integration by parts from Lemma~\ref{ibp}, we obtain
    \begin{align} \label{bulk_ibp}
    \begin{aligned}
      T_{2,\eps}(x,\wt x)
      =
      \sum_{i=0}^{n-1}
      \EE\bigg[
        \int_{\DN}\mu_\be(d\ga)
        \int_\R
          a_{i,\eps}^{x,\ga}(y)
        D_{i,y}\Bigl(  G_x^\eps(\ga)\log \cZ_{\wt x}^{\be,\eps}   \Bigr)
        \,dy    \bigg].
    \end{aligned}
  \end{align}
In the calculation above, the Sobolev facts needed follow from
  Lemma~\ref{lem_MD1}, Proposition~\ref{prop_MD}, and
  \eqref{cancel_DQ}.  A fully formal justification is
  obtained by first replacing the polymer measure and the free energy by their truncated
  versions.  For the truncated versions, the product
  rule, Fubini's theorem, and Gaussian integration by parts apply
  directly.
  One then lets the truncation parameter $R\to\infty$, using   Lemma~\ref{lem_MD1},
  Proposition~\ref{prop_MD}, and dominated convergence.

 To continue from \eqref{bulk_ibp}, expanding the Malliavin derivative by the product rule gives
  \begin{align}\label{bulk_split}
    \begin{aligned}
      T_{2,\eps}(x,\wt x)
      ={}&
      \sum_{i=0}^{n-1}
      \EE\bigg[
        \int_\R
         \big(  D_{i,y}\log \cZ_{\wt x}^{\be,\eps} \big)
          \int_{\DN}
            a_{i,\eps}^{x,\ga}(y)
            Q_x^{\be,\eps}(d\ga)
        \,dy
      \bigg]
      +\cR_{2,\eps},
\end{aligned}
\end{align}
where
$$
      \cR_{2,\eps}
    =
      \sum_{i=0}^{n-1}
      \EE\bigg[
        \big(  \log \cZ_{\wt x}^{\be,\eps} \big)
        \int_{\DN}\mu_\be(d\ga)
        \int_\R
          a_{i,\eps}^{x,\ga}(y)
          D_{i,y}G_x^\eps(\ga)
        \,dy
      \bigg].
$$
  Lemma~\ref{cancel_b} applies pathwise to the inner two
  integrals in $\cR_{2,\eps}$ and yields
  $\cR_{2,\eps}=0$.

  The deterministic factor by which the partition function in
  \eqref{moll_Q} differs from the normalized one
  has zero Malliavin derivative.  Hence Proposition~\ref{prop_MD},
  together with \eqref{MD_H}, gives
  \begin{align}\label{bulk_DZ}
    \begin{aligned}
      D_{i,y}\log \cZ_{\wt x}^{\be,\eps}
      =\be\int_{\DN}
        A_{i,\eps}^{\wt x,\eta}(y)
        Q_{\wt x}^{\be,\eps}(d\eta),
      \qquad 0\leq i\leq n-1.
    \end{aligned}
  \end{align}
  Substituting \eqref{bulk_DZ} into
  \eqref{bulk_split} and using Fubini once more, we find
  \begin{align}\label{bulk_conv}
    \begin{aligned}
      \frac1\be T_{2,\eps}(x,\wt x)
      =\EE\bigg[
        \int_{\DN^2}
        \sum_{i=0}^{n-1}
          \left(
            \int_\R
              A_{i,\eps}^{\wt x,\eta}(y)
              a_{i,\eps}^{x,\ga}(y)
            \,dy
          \right)
        Q_{\wt x}^{\be,\eps}(d\eta)
        Q_x^{\be,\eps}(d\ga)
      \bigg].
    \end{aligned}
  \end{align}

  Since $\phi^\eps$ is even and
  $\Psi^\eps=\Phi^\eps\ast\phi^\eps$, one has
  \begin{align*}
    \int_\R
      \Phi^\eps(a-y)\phi^\eps(b-y)\,dy
    =\Psi^\eps(a-b).
  \end{align*}
  Expanding the product in the inner integral of
  \eqref{bulk_conv} into four terms and applying this
  identity to each term gives, by the definition in
  \eqref{not_path},
  \begin{align*}
    \int_\R
      A_{i,\eps}^{\wt x,\eta}(y)
      a_{i,\eps}^{x,\ga}(y)
    \,dy
    =\Xi_{i,\eps}^{\wt x,x}(\eta,\ga).
  \end{align*}
  We have therefore proved the exact prelimit identity
  \begin{align}\label{T2_Xi}
    \begin{aligned}
      \frac1\be T_{2,\eps}(x,\wt x)
      =\EE\bigg[
        \int_{\DN^2}
        \sum_{i=0}^{n-1}
          \Xi_{i,\eps}^{\wt x,x}(\eta,\ga)
        Q_{\wt x}^{\be,\eps}(d\eta)
        Q_x^{\be,\eps}(d\ga)
      \bigg].
    \end{aligned}
  \end{align}

  Apply Proposition~\ref{T_uc}, or equivalently
  \eqref{T2_unif}, with a compact set
  $K\Subset\R\setminus\{\wt x\}$ containing $x$.  This yields
  \begin{align}\label{T2_sgn}
    \begin{aligned}
      \lim_{\eps\downarrow0}
        \frac1\be T_{2,\eps}(x,\wt x)
      =\frac12\EE\bigg[
        \int_{\DN^2}
        \sum_{i=0}^{n-1}
          \mathfrak D_i^{\wt x,x}(\eta,\ga)
        Q_{\wt x}^{\be}(d\eta)
        Q_x^{\be}(d\ga)
      \bigg].
    \end{aligned}
  \end{align}
  The remaining step is pathwise.  By the quenched switching identity
  \eqref{quen_b} in
  Lemma~\ref{quen_lem}, proved in the next subsection,
  the inner double integral in \eqref{T2_sgn} equals
  $ \sgn(\wt x-x)
    \bigl(Q_{\wt x}^{\be}\otimes Q_x^{\be}\bigr)
    \{\tau\leq n-1\}
$
for almost every environment.  Taking expectation and using the
  definition of $\rmP_{\wt x,x}^{\be}$ in
  \eqref{not_exp} completes the proof.
\end{proof}

\subsection{A quenched switching identity}

The following pathwise identity isolates the symmetry that was used at the end of the previous section. In addition, it will also be
used in the proofs of Proposition~\ref{bdry} and the stationary
Stein identity.  

\begin{lemma}[\textsf{Quenched switching at the first meeting}]
  \label{quen_lem}
  Fix $\al\in\R$, $\be>0$, and $x,y\in\R$ with $x\neq y$.  Almost
  surely with respect to the Brownian environment,
  \begin{align}\label{quen_b}
      \int_{\DN^2}    \sum_{i=0}^{n-1}    \mathfrak D_i^{y,x}(\eta,\ga)\,   Q_y^{\al,\be}(d\eta)Q_x^{\al,\be}(d\ga)
  =   \sgn(y-x)   \bigl(Q_y^{\al,\be}\otimes Q_x^{\al,\be}\bigr)   \{\tau\leq n-1\},
  \end{align}
  and
  \begin{align}\label{quen_t}
      \int_{\DN^2}  \sgn\bigl(P_n(y,\eta)-P_n(x,\ga)\bigr)\,  Q_y^{\al,\be}(d\eta)Q_x^{\al,\be}(d\ga)
=   \sgn(y-x)        \bigl(Q_y^{\al,\be}\otimes Q_x^{\al,\be}\bigr)   \{\tau=\infty\}.
  \end{align}
  Consequently,
  \begin{align}\label{quen_c}
      \int_{\DN^2}
      \left[
        \sum_{i=0}^{n-1}\mathfrak D_i^{y,x}(\eta,\ga)
        +\sgn\bigl(P_n(y,\eta)-P_n(x,\ga)\bigr)
      \right]
      Q_y^{\al,\be}(d\eta)Q_x^{\al,\be}(d\ga) =\sgn(y-x).
   \end{align}
\end{lemma}

\begin{proof}
  The argument is quenched.  Fix a realization of the Brownian
  environment for which all the Brownian paths are continuous.  We
  first treat the case $y<x$.  The case $y>x$ will follow by
  interchanging the two paths.

  For this proof, it is convenient to use the normalization
    $\cZ_z^{\al,\be}
    =
    \int_{\DN}
      e^{\be H_z^\al(\ga)}\,\mu_\be(d\ga)
    =\be^n Z_z^{\al,\be}.$
  Then
  \begin{align*}
    Q_z^{\al,\be}(d\ga)
    =
    \frac{e^{\be H_z^\al(\ga)}}{\cZ_z^{\al,\be}}
    \,\mu_\be(d\ga).
  \end{align*}
  In the remainder of the proof, we suppress $\al,\be$ from
  $\cZ_z^{\al,\be}$ and write simply $\cZ_z$.

  Let $\{X_j\}_{j\geq1}$ and $\{Y_j\}_{j\geq1}$ be two independent
  sequences of i.i.d.\ $\Exp(\be)$ random variables.  We realize the
  two reference paths as
  \begin{align*}
    \eta_0=\ga_0=0,
    \qquad
    \eta_m=\sum_{j=1}^m X_j,
    \qquad
    \ga_m=\sum_{j=1}^m Y_j,
    \qquad m\geq1.
  \end{align*}
  Set
  \begin{align}
    \cF_k
    =
    \sigma\bigl(
      X_1,\ldots,X_k,Y_1,\ldots,Y_k
    \bigr),
    \qquad 0\leq k\leq n.
  \end{align}

  Recall that, for $0\leq k\leq n-1$,
  \begin{align}\label{sw_Bk}
    \begin{aligned}
      \fB_k
      &=
      \bigcap_{r=1}^{k}
      \{P_r(y,\eta)<P_{r-1}(x,\ga)\}
      \cap
      \{P_{k+1}(y,\eta)>P_k(x,\ga)\},
      \\
      \fB_\infty
      &=
      \bigcap_{r=1}^{n}
      \{P_r(y,\eta)<P_{r-1}(x,\ga)\}.
    \end{aligned}
  \end{align}
  Since $Q_y^{\al,\be}\otimes Q_x^{\al,\be}$ is absolutely
  continuous with respect to
  $\mu_\be\otimes\mu_\be$, equality events have quenched probability
  zero.  Thus, up to a quenched null set, it holds that 
 $\fB_k=\{\tau=k\}$,
    $\fB_\infty=\{\tau=\infty\}$.
  Recall the notation $\mathbb{E}^{\textup{poly}\times 2}$ from \eqref{not_exp}, we start from
  \begin{align}\label{sw_start}
  \begin{aligned}
    \EE_{(y;\eta),(x;\ga)}^{\mathrm{poly}\times2}
    \left[
      \sum_{i=0}^{n-1}
      \mathfrak D_i^{y,x}(\eta,\ga)
    \right]
    &=
    \sum_{k=0}^{n-1}\sum_{i=0}^{n-1}
    \EE_{\eta,\ga}^{\mu_\be\times2}
    \left[
      \ind_{\fB_k}
      \frac{e^{\be H_y^\al(\eta)}}{\cZ_y}
      \frac{e^{\be H_x^\al(\ga)}}{\cZ_x}
      \mathfrak D_i^{y,x}(\eta,\ga)
    \right]
     \\
    &\qquad+
    \sum_{i=0}^{n-1}
    \EE_{\eta,\ga}^{\mu_\be\times2}
    \left[
      \ind_{\fB_\infty}
      \frac{e^{\be H_y^\al(\eta)}}{\cZ_y}
      \frac{e^{\be H_x^\al(\ga)}}{\cZ_x}
      \mathfrak D_i^{y,x}(\eta,\ga)
    \right].
  \end{aligned} \end{align}

  If the two paths first meet at level $k$, then, for every $i<k$,
  \begin{align*}
    P_i(y,\eta)
    <
    P_{i+1}(y,\eta)
    <
    P_i(x,\ga)
    <
    P_{i+1}(x,\ga).
  \end{align*}
  Hence all four sign functions in
  $\mathfrak D_i^{y,x}(\eta,\ga)$ are equal to $-1$, and therefore
  $\ind_{\fB_k}\mathfrak D_i^{y,x}(\eta,\ga)=0$,
    for $i<k$.
  The same argument shows that 
    $\ind_{\fB_\infty}
    \mathfrak D_i^{y,x}(\eta,\ga)=0$,
   for $0\leq i\leq n-1$.
  Consequently, \eqref{sw_start} reduces to
  \begin{align}\label{sw_red}
    &\EE_{(y;\eta),(x;\ga)}^{\mathrm{poly}\times2}
    \left[
      \sum_{i=0}^{n-1}
      \mathfrak D_i^{y,x}(\eta,\ga)
    \right]
   =
    \sum_{k=0}^{n-1}\sum_{i=k}^{n-1}
    \EE_{\eta,\ga}^{\mu_\be\times2}
    \left[
      \ind_{\fB_k}
      \frac{e^{\be H_y^\al(\eta)}}{\cZ_y}
      \frac{e^{\be H_x^\al(\ga)}}{\cZ_x}
      \mathfrak D_i^{y,x}(\eta,\ga)
    \right].
  \end{align}

  We now fix $k\in\{0,\ldots,n-1\}$ and compute the corresponding
  terms exactly.  Since
  $P_k(y,\eta)<P_k(x,\ga)$
  on the pre-meeting part of $\fB_k$, set
  \begin{align}
    s_k
    =
    P_k(x,\ga)-P_k(y,\eta)
    =
    x+\ga_k-y-\eta_k>0.
  \end{align}
   
  \noi
\begin{figure}[t]
  \centering
  \begin{tikzpicture}[x=.92cm,y=.76cm,>=Stealth,font=\small]

    \draw[axes,->] (0,0) -- (0,5.65) node[above left] {level};
    \draw[axes,->] (0,0) -- (11.7,0)
      node[below right] {horizontal coordinate};
    \foreach \yy/\lab in {1/1,2/2,4/k,5/{k+1}}{
      \draw[lev] (0,\yy) -- (9.1,\yy);
      \node[left] at (0,\yy) {$\lab$};
    }
    \node[left] at (0,0) {$0$};
    \node[left] at (0,3) {$\vdots$};

    \draw[redpath]
      (.75,0) -- (1.35,0) -- (1.35,1)
      -- (1.95,1) -- (1.95,2);

    \draw[redpath,densely dotted]
      (1.95,2) -- (2.30,2) -- (2.30,3)
      -- (2.65,3) -- (2.65,4);

    \draw[redpath]
      (2.65,4) -- (7.15,4) -- (7.15,5);

    \draw[bluepath]
      (4.55,0) -- (5.05,0) -- (5.05,1)
      -- (5.45,1) -- (5.45,2);

    \draw[bluepath,densely dotted]
      (5.45,2) -- (5.75,2) -- (5.75,3)
      -- (6.15,3) -- (6.15,4);

    \draw[bluepath]
      (6.15,4) -- (8.45,4) -- (8.45,5);


    \fill[red] (.75,0) circle (1.6pt);
    \fill[red] (1.35,1) circle (1.6pt);
    \fill[red] (1.95,2) circle (1.6pt);
    \fill[red] (2.30,3) circle (1.6pt);
    \fill[red] (2.65,4) circle (1.6pt);
    \fill[red] (7.15,5) circle (1.6pt);

    \fill[blue] (4.55,0) circle (1.6pt);
    \fill[blue] (5.05,1) circle (1.6pt);
    \fill[blue] (5.45,2) circle (1.6pt);
    \fill[blue] (5.75,3) circle (1.6pt);
    \fill[blue] (6.15,4) circle (1.6pt);
    \fill[blue] (8.45,5) circle (1.6pt);


    \draw[red,fill=white,line width=.9pt] (1.35,0) circle (1.9pt);
    \draw[red,fill=white,line width=.9pt] (1.95,1) circle (1.9pt);
    \draw[red,fill=white,line width=.9pt] (2.30,2) circle (1.9pt);
    \draw[red,fill=white,line width=.9pt] (2.65,3) circle (1.9pt);
    \draw[red,fill=white,line width=.9pt] (7.15,4) circle (1.9pt);

    \draw[blue,fill=white,line width=.9pt] (5.05,0) circle (1.9pt);
    \draw[blue,fill=white,line width=.9pt] (5.45,1) circle (1.9pt);
    \draw[blue,fill=white,line width=.9pt] (5.75,2) circle (1.9pt);
    \draw[blue,fill=white,line width=.9pt] (6.15,3) circle (1.9pt);
    \draw[blue,fill=white,line width=.9pt] (8.45,4) circle (1.9pt);

    \node[labr,anchor=south east] at (2.58,4.08) {$y+\eta_k$};
    \node[labb,anchor=north west] at (6.28,3.92) {$q=x+\ga_k$};
    \node[labr,anchor=south east] at (7.08,5.07)
      {$y+\eta_{k+1}$};
    \node[labb,anchor=south west] at (8.52,5.07)
      {$x+\ga_{k+1}$};
    \node[labr,fill=white,inner sep=1pt] at (4.55,4.42)
      {$X_{k+1}=\eta_{k+1}-\eta_k$};

    \draw[<->] (2.65,3.52) -- (6.15,3.52);
    \node[fill=white,inner sep=1pt] at (4.40,3.25)
      {$s_k$};

    \node[align=center,text width=3.6cm] at (10.00,3.10)
      {$\displaystyle X_{k+1}>s_k$
       \\[2pt] $\Downarrow$
       \\[2pt] $\displaystyle y+\eta_{k+1}>x+\ga_k=q$
       \\[3pt] first meeting at level $k$};

    \node[labr,below] at (.75,0) {$y$};
    \node[labb,below] at (4.55,0) {$x$};
    \node at (2.70,-.48) {$y<x$};

  \end{tikzpicture}
  \caption{If $k$ is the first meeting level, then
  $s_k=x+\ga_k-y-\eta_k>0$. The condition $X_{k+1}>s_k$ is equivalent to
  $y+\eta_{k+1}>x+\ga_k=q$, which yields first meeting at level $k$.}
  \label{fig_skXk}
\end{figure}
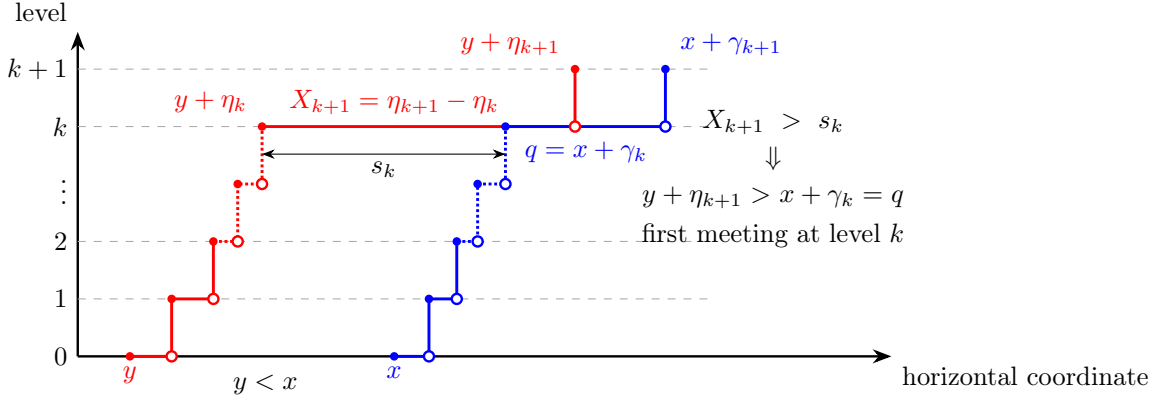

\noi
  Then, as illustrated in Figure 
  \ref{fig_skXk},
  we have
  \begin{align*}
    \fB_k
    =
    \bigcap_{r=1}^{k}
    \{P_r(y,\eta)<P_{r-1}(x,\ga)\}
    \cap
    \{X_{k+1}>s_k\},
  \end{align*}
  up to a null set.  Notice that $s_k$ is $\cF_k$-measurable,
  whereas $X_{k+1}$ is independent of $\cF_k$.
  
  We split the two Hamiltonians into two parts:
  \begin{align}
    H_y^\al(\eta)
    =
    V_1+V_2
    \AND
    H_x^\al(\ga)
    =
    W_1+W_2,
  \end{align}
  where
  \begin{align*}
    V_1
    &=
    \sum_{j=0}^{k-1}
    \bigl[
      B_j(P_{j+1}(y,\eta))
      -B_j(P_j(y,\eta))
    \bigr]
   +
    B_k(P_k(x,\ga))-B_k(P_k(y,\eta)),
    \\
    V_2
    &=
    \al B_n(P_n(y,\eta))
    +
    B_k(P_{k+1}(y,\eta))-B_k(P_k(x,\ga))\\
   & \qquad \qquad +
    \sum_{j=k+1}^{n-1}
    \bigl[
      B_j(P_{j+1}(y,\eta))
      -B_j(P_j(y,\eta))
    \bigr],
  \end{align*}
  and
  \begin{align*}
    W_1
    &=
    \sum_{j=0}^{k-1}
    \bigl[
      B_j(P_{j+1}(x,\ga))
      -B_j(P_j(x,\ga))
    \bigr],
    \\
    W_2
    &=
    \al B_n(P_n(x,\ga))
    +
    \sum_{j=k}^{n-1}
    \bigl[
      B_j(P_{j+1}(x,\ga))
      -B_j(P_j(x,\ga))
    \bigr].
  \end{align*}
  Recall the notation $\mathbb{E}^{\mu_\beta \times 2}$ from \eqref{not_exp}, and note that $V_1$ and $W_1$ are $\cF_k$-measurable.
  Hence, for $i\geq k$,
  \begin{align}\label{sw_cond}
    &\EE_{\eta,\ga}^{\mu_\be\times2}
    \left[
      \ind_{\fB_k}
      \frac{e^{\be H_y^\al(\eta)}}{\cZ_y}
      \frac{e^{\be H_x^\al(\ga)}}{\cZ_x}
      \mathfrak D_i^{y,x}(\eta,\ga)
    \right]
    \notag\\
    &=
    \EE_{\eta,\ga}^{\mu_\be\times2}
    \bigg[
      \ind_{
        \bigcap_{r=1}^{k}
        \{P_r(y,\eta)<P_{r-1}(x,\ga)\}
      }
      \frac{e^{\be V_1}}{\cZ_y}
      \frac{e^{\be W_1}}{\cZ_x}
      \EE_{\eta,\ga}^{\mu_\be\times2}
      \left[
        \ind_{\{X_{k+1}>s_k\}}
        e^{\be V_2}e^{\be W_2}
        \mathfrak D_i^{y,x}(\eta,\ga)
        \,\middle|\,\cF_k
      \right]
    \bigg].
  \end{align}

 \noi
 Let us recall the {\it memoryless property} of the exponential distribution.
  If $X\sim\Exp(\be)$, then, for $s>0$,
  \begin{align*}
    \EE\bigl[f(X)\ind_{\{X>s\}}\bigr]
    =
    \int_s^\infty
      f(t)\be e^{-\be t}\,dt
    &=
    e^{-\be s}
    \int_0^\infty
      f(s+r)\be e^{-\be r}\,dr
    \\
    &=
    e^{-\be s}\EE[f(s+X)].
  \end{align*}
  More generally, if $U_m=X_1+\cdots+X_m$ is an $\Exp(\be)$
  random walk, then
  by independence between $X_1$ and other steps,
  we have
  \begin{align}\label{sw_mem}
    &\EE\left[
      f(U_1,\ldots,U_\ell)
      \ind_{\{X_1>s\}}
    \right]
   =
    e^{-\be s}
    \EE\left[
      f(s+U_1,\ldots,s+U_\ell)
    \right].
  \end{align}

  We first consider $i=k+1,\ldots,n-1$.  Applying
  \eqref{sw_mem} to $X_{k+1}$
  with
  $s=s_k$, gives
  \begin{align} 
  \label{sw_post}
    \begin{aligned} 
    &\EE_{\eta,\ga}^{\mu_\be\times2}
    \left[
      \ind_{\{X_{k+1}>s_k\}}
      e^{\be V_2}e^{\be W_2}
      \mathfrak D_i^{y,x}(\eta,\ga)
      \,\middle|\,\cF_k
    \right]
   \\
    &=
    e^{-\be s_k}
    \EE_{\eta,\ga}^{\mu_\be\times2}
    \bigg[
      \exp\bigg\{
        \be\bigg(
          \al B_n(x+\ga_k-\eta_k+\eta_n)
          +B_k(x+\ga_k-\eta_k+\eta_{k+1})
          -B_k(x+\ga_k)
    \\
    &\hspace{3.6cm}
          +\sum_{j=k+1}^{n-1}
          \bigl[
            B_j(x+\ga_k-\eta_k+\eta_{j+1})
            -B_j(x+\ga_k-\eta_k+\eta_j)
          \bigr]
        \bigg)
      \bigg\}
     \\
    & \times
      \exp\bigg\{
        \be\bigg(
          \al B_n(x+\ga_n)
          +
          \sum_{j=k}^{n-1}
          \bigl[
            B_j(x+\ga_{j+1})
            -B_j(x+\ga_j)
          \bigr]
        \bigg)
      \bigg\} 
      \mathfrak{D}_i(\ga_k-\eta_k+\eta_i,\ga_i)
      \,\Bigm|\,\cF_k
    \bigg].
  \end{aligned}
    \end{align}
  Here we have used the 
  notation \eqref{not_pathb}
 for $\mathfrak{D}_i$
 and 
 \begin{align*}
 \mathfrak{D}_i(y+\eta_i + s_k, x+\ga_i) 
 =\mathfrak{D}_i( x+\ga_k-\eta_k+\eta_i, x+\ga_i)
 = \mathfrak{D}_i(\ga_k-\eta_k+\eta_i,\ga_i).
 \end{align*}

  Now use the (time homogeneous) Markov property of the two random walks to get rid of the conditional expectation on $\mathcal{F}_k$.
  Define, for $\ell\geq0$,
  \begin{align}\label{sw_prime}
    \eta'_\ell
    =
    \eta_{k+\ell}-\eta_k
    \AND
    \ga'_\ell
    =
    \ga_{k+\ell}-\ga_k,
  \end{align}
  so that $\eta'_0=\ga'_0=0$, and put
  \begin{align*}
    q=x+\ga_k,
    \qquad
    m=i-k.
  \end{align*}
  Then, we have 
  \begin{align*}
 &\mathfrak{D}_i(\ga_k-\eta_k+\eta_i,\ga_i)
 = \mathfrak{D}_i(\eta_i -\eta_k,\ga_i - \ga_k) \\
 &\quad = \sgn(\eta'_{m+1}-\ga'_{m+1})
        -\sgn(\eta'_m-\ga'_{m+1}) 
        -\sgn(\eta'_{m+1}-\ga'_m)
        +\sgn(\eta'_m-\ga'_m).
  \end{align*}
Let $\EE_{\eta',\ga'}^{\mu_\be\times2}$ denote the corresponding
  product reference expectation for the remaining $n-k$ exponential
  increments, then 
  the conditional expectation in
  \eqref{sw_post} becomes
  \begin{align}\label{sw_postp}
    &e^{-\be s_k}
    \EE_{\eta',\ga'}^{\mu_\be\times2}
    \bigg[
      \exp\bigg\{
        \be\bigg(
          \al B_n(q+\eta'_{n-k})
          +
          \sum_{\ell=0}^{n-k-1}
          \bigl[
            B_{k+\ell}(q+\eta'_{\ell+1})
            -B_{k+\ell}(q+\eta'_\ell)
          \bigr]
        \bigg)
      \bigg\}
    \notag\\
    &\hspace{1.1cm}\times
      \exp\bigg\{
        \be\bigg(
          \al B_n(q+\ga'_{n-k}) 
          +
          \sum_{\ell=0}^{n-k-1}
          \bigl[
            B_{k+\ell}(q+\ga'_{\ell+1})
            -B_{k+\ell}(q+\ga'_\ell)
          \bigr]
        \bigg)
      \bigg\}
    \notag\\
    &\hspace{1.1cm}\times
      \Big[
        \sgn(\eta'_{m+1}-\ga'_{m+1})
        -\sgn(\eta'_m-\ga'_{m+1}) 
        -\sgn(\eta'_{m+1}-\ga'_m)
        +\sgn(\eta'_m-\ga'_m)
      \Big]
    \bigg].
  \end{align}

  The product of the first two exponential factors in
  \eqref{sw_postp} is symmetric under
  $\eta'\leftrightarrow\ga'$,
  whereas the four-sign factor changes sign.
    Hence,  the expectation in
  \eqref{sw_postp} is zero.  Substituting this  into
  \eqref{sw_cond} yields
  \begin{align}\label{sw_post0}
    &\EE_{\eta,\ga}^{\mu_\be\times2}
    \left[
      \ind_{\fB_k}
      \frac{e^{\be H_y^\al(\eta)}}{\cZ_y}
      \frac{e^{\be H_x^\al(\ga)}}{\cZ_x}
      \mathfrak D_i^{y,x}(\eta,\ga)
    \right]
    =0,
    \qquad
    i=k+1,\ldots,n-1.
  \end{align}

  It remains to consider the term $i=k$.  On $\fB_k$,  we have
   $ P_k(y,\eta)
    <
    P_k(x,\ga)
    <
    P_{k+1}(y,\eta)$,
  and consequently by definition \eqref{not_path},
  we have
    $\mathfrak D_k^{y,x}(\eta,\ga)
    =
    \sgn\bigl(
      P_{k+1}(y,\eta)-P_{k+1}(x,\ga)
    \bigr)-1.$

  Applying the same memoryless calculation and Markov property as above gives
  \begin{align}\label{sw_kmem}
    &\EE_{\eta,\ga}^{\mu_\be\times2}
    \left[
      \ind_{\{X_{k+1}>s_k\}}
      e^{\be V_2}e^{\be W_2}
      \mathfrak D_k^{y,x}(\eta,\ga)
      \,\middle|\,\cF_k
    \right]
    \notag\\
    &=
    e^{-\be s_k}
    \EE_{\eta',\ga'}^{\mu_\be\times2}
    \bigg[
      \exp\bigg\{
        \be\bigg(
          \al B_n(q+\eta'_{n-k})
          +
          \sum_{\ell=0}^{n-k-1}
          \bigl[
            B_{k+\ell}(q+\eta'_{\ell+1})
            -B_{k+\ell}(q+\eta'_\ell)
          \bigr]
        \bigg)
      \bigg\}
    \notag\\
    &\hspace{1cm}\times
      \exp\bigg\{
        \be\bigg(
          \al B_n(q+\ga'_{n-k})
          +
          \sum_{\ell=0}^{n-k-1}
          \bigl[
            B_{k+\ell}(q+\ga'_{\ell+1})
            -B_{k+\ell}(q+\ga'_\ell)
          \bigr]
        \bigg)
      \bigg\}
    \notag\\
    &\hspace{1cm}\times
      \bigl[
        \sgn(\eta'_1-\ga'_1)-1
      \bigr]
    \bigg].
  \end{align}
  Again the product Gibbs weight is symmetric under
  $\eta'\leftrightarrow\ga'$.  Hence
  \begin{align*}
    &\EE_{\eta',\ga'}^{\mu_\be\times2}
    \bigg[
      e^{\be(\cdots\eta')}
      e^{\be(\cdots\ga')}
      \sgn(\eta'_1-\ga'_1)
    \bigg]
    =0.
  \end{align*}
  Therefore \eqref{sw_kmem} equals
  \begin{align}\label{sw_kminus}
    &-
    e^{-\be s_k}
    \EE_{\eta',\ga'}^{\mu_\be\times2}
    \bigg[
      \exp\bigg\{
        \be\bigg(
          \al B_n(q+\eta'_{n-k})
          +
          \sum_{\ell=0}^{n-k-1}
          \bigl[
            B_{k+\ell}(q+\eta'_{\ell+1})
            -B_{k+\ell}(q+\eta'_\ell)
          \bigr]
        \bigg)
      \bigg\}
    \notag\\
    &\hspace{2cm}\times
      \exp\bigg\{
        \be\bigg(
          \al B_n(q+\ga'_{n-k})
          +
          \sum_{\ell=0}^{n-k-1}
          \bigl[
            B_{k+\ell}(q+\ga'_{\ell+1})
            -B_{k+\ell}(q+\ga'_\ell)
          \bigr]
        \bigg)
      \bigg\}
    \bigg] \\
   &=-
    \EE_{\eta,\ga}^{\mu_\be\times2}
    \left[
      \ind_{\{X_{k+1}>s_k\}}
      e^{\be V_2}e^{\be W_2}
      \,\middle|\,\cF_k
    \right],  \notag
  \end{align}

\noi
where the last equality is obtained by 
reversing the above process, 
with the application of the Markov property which now reintroduces the conditional expectation on $\mathcal{F}_k$, then followed by the memoryless property \eqref{sw_mem}.
%
  Substituting this into \eqref{sw_cond} gives
  \begin{align}\label{sw_kfin}
  \begin{aligned}
    \EE_{\eta,\ga}^{\mu_\be\times2}
    \left[
      \ind_{\fB_k}
      \frac{e^{\be H_y^\al(\eta)}}{\cZ_y}
      \frac{e^{\be H_x^\al(\ga)}}{\cZ_x}
      \mathfrak D_k^{y,x}(\eta,\ga)
    \right] 
    &=
    -
    \EE_{\eta,\ga}^{\mu_\be\times2}
    \left[
      \ind_{\fB_k}
      \frac{e^{\be H_y^\al(\eta)}}{\cZ_y}
      \frac{e^{\be H_x^\al(\ga)}}{\cZ_x}
    \right] \\
    &=
    -
    \bigl(
      Q_y^{\al,\be}\otimes Q_x^{\al,\be}
    \bigr)(\fB_k).
  \end{aligned}
  \end{align}
  Plugging \eqref{sw_post0} and
  \eqref{sw_kfin} into \eqref{sw_red} and summing over $k$, we obtain
  \begin{align*}
    &\EE_{(y;\eta),(x;\ga)}^{\mathrm{poly}\times2}
    \left[
      \sum_{i=0}^{n-1}
      \mathfrak D_i^{y,x}(\eta,\ga)
    \right]
    =
    -
    \sum_{k=0}^{n-1}
    \bigl(
      Q_y^{\al,\be}\otimes Q_x^{\al,\be}
    \bigr)(\fB_k)
  =
    -
    \bigl(
      Q_y^{\al,\be}\otimes Q_x^{\al,\be}
    \bigr)\{\tau\leq n-1\}.
  \end{align*}
  Since $y<x$, $\sgn(y-x)=-1$, and this proves
  \eqref{quen_b}.

  We next prove the terminal identity.  On $\fB_k$, 
  we can apply exactly the
  same conditional computation, but replace the four-sign factor in
  \eqref{sw_postp} by
   $ \sgn\bigl( P_n(y,\eta)-P_n(x,\ga)  \bigr).$
  After the same process, which uses the memoryless shift and the Markov property, this sign becomes
   $ \sgn(\eta'_{n-k}-\ga'_{n-k}).$
  More precisely,
  \begin{align*}
    &\EE_{\eta,\ga}^{\mu_\be\times2}
    \left[
      \ind_{\{X_{k+1}>s_k\}}
      e^{\be V_2}e^{\be W_2}
      \sgn\bigl(
        P_n(y,\eta)-P_n(x,\ga)
      \bigr)
      \,\middle|\,\cF_k
    \right]
    \\
    &\quad=
    e^{-\be s_k}
    \EE_{\eta',\ga'}^{\mu_\be\times2}
    \bigg[
      \exp\bigg\{
        \be\bigg(
          \al B_n(q+\eta'_{n-k})
          +
          \sum_{\ell=0}^{n-k-1}
          \bigl[
            B_{k+\ell}(q+\eta'_{\ell+1})
            -B_{k+\ell}(q+\eta'_\ell)
          \bigr]
        \bigg)
      \bigg\}
    \\
    &\hspace{1cm}\times
      \exp\bigg\{
        \be\bigg(
          \al B_n(q+\ga'_{n-k})
          +
          \sum_{\ell=0}^{n-k-1}
          \bigl[
            B_{k+\ell}(q+\ga'_{\ell+1})
            -B_{k+\ell}(q+\ga'_\ell)
          \bigr]
        \bigg)
      \bigg\}
      \sgn(\eta'_{n-k}-\ga'_{n-k})
    \bigg].
  \end{align*}
  The integrand apart from the sign is invariant under
  $\eta'\leftrightarrow\ga'$, while the sign changes sign.
  Consequently,
  \begin{align}
    \EE_{(y;\eta),(x;\ga)}^{\mathrm{poly}\times2}
    \Big[
      \ind_{\fB_k}
      \sgn\bigl(
        P_n(y,\eta)-P_n(x,\ga)
      \bigr)
    \Big]
    =0.
  \end{align}

  On $\fB_\infty$, we have
  $
  P_n(y,\eta)
    <
    P_{n-1}(x,\ga)
    <
    P_n(x,\ga)$,
  so that
$ \sgn\bigl(
      P_n(y,\eta)-P_n(x,\ga)
    \bigr)
    =-1.$
  Hence, using the decomposition into
  $\fB_0,\ldots,\fB_{n-1},\fB_\infty$,
  \begin{align*}
    &\EE_{(y;\eta),(x;\ga)}^{\mathrm{poly}\times2}
    \left[
      \sgn\bigl(
        P_n(y,\eta)-P_n(x,\ga)
      \bigr)
    \right] =
    -
    \bigl(
      Q_y^{\al,\be}\otimes Q_x^{\al,\be}
    \bigr)(\fB_\infty) =
    -
    \bigl(
      Q_y^{\al,\be}\otimes Q_x^{\al,\be}
    \bigr)\{\tau=\infty\}.
  \end{align*}
  Since $\sgn(y-x)=-1$, this proves
  \eqref{quen_t} when $y<x$.

  Finally, if $y>x$, we can interchange $(y,\eta)$ and $(x,\ga)$. 
  From the
  definition \eqref{not_path},
  we have 
  \begin{align*}
    \mathfrak D_i^{y,x}(\eta,\ga)
    &=
    -\mathfrak D_i^{x,y}(\ga,\eta),
    \\
    \sgn\bigl(
      P_n(y,\eta)-P_n(x,\ga)
    \bigr)
    &=
    -\sgn\bigl(
      P_n(x,\ga)-P_n(y,\eta)
    \bigr).
  \end{align*}
  Thus both identities acquire the opposite sign, which is precisely
  $\sgn(y-x)=1$.  This proves
  \eqref{quen_b} and
  \eqref{quen_t} in all cases.

  Finally, adding the two identities yields \eqref{quen_c}.
\end{proof}

\subsection{The boundary term}

\begin{proposition}\label{bdry}
  For $x\neq\wt x$,
  \begin{align}\label{bdry_lim}
    \begin{aligned}
      \lim_{\eps\downarrow0}
        \frac1\be T_{1,\eps}(x,\wt x)
      =
      \frac{\al^2}{2}\sgn(\wt x-x)
        \rmP_{\wt x,x}^{\be}(\tau=\infty)
   +
      \al^2\EE\bigl[Q_0^{\be}\{\ga_n>-x\}\bigr]
      -\frac{\al^2}{2}.
    \end{aligned}
  \end{align}
\end{proposition}

\begin{proof}
  We continue to suppress $\al$ in the superscripts.  We first
  apply Gaussian integration by parts to the terminal white-noise
  term in \eqref{T12}.  Recall from \eqref{s_xi} that
  \begin{align*}
    \xi_n^\eps\bigl(P_n(x,\ga)\bigr)
    &=
    W_n\left(
      \phi^\eps\bigl(P_n(x,\ga)-\cdot\bigr)
    \right)
    =
    \int_\R
      \phi^\eps\bigl(P_n(x,\ga)-y\bigr)\,
      \xi_n(dy).
  \end{align*}
  Writing
  $
    G_x^\eps(\ga)
    = \frac{e^{\be H_x^\eps(\ga)}}{\cZ_x^{\be,\eps}},
  $
  we can therefore rewrite the first term in \eqref{T12} as
  \begin{align*}
    T_{1,\eps}(x,\wt x)
    &=
    \al\,
    \EE\bigg[
      \int_{\DN}\mu_\be(d\ga)\,
      G_x^\eps(\ga)
      \log\cZ_{\wt x}^{\be,\eps}
      \int_\R
        \phi^\eps\bigl(P_n(x,\ga)-y\bigr)
        \xi_n(dy)
    \bigg].
  \end{align*}
  By Fubini's theorem and the Gaussian integration-by-parts formula
  \eqref{Gibp1},
  \begin{align}\label{T1_ibp}
    \begin{aligned}
      T_{1,\eps}(x,\wt x)
      &=
      \al\,
      \EE\bigg[
        \int_{\DN}\mu_\be(d\ga)
        \int_\R
          \phi^\eps\bigl(P_n(x,\ga)-y\bigr)
          D_{n,y}\left(
            G_x^\eps(\ga)
            \log\cZ_{\wt x}^{\be,\eps}
          \right)
        \,dy
      \bigg].
    \end{aligned}
  \end{align}
  Expanding the Malliavin derivative by the product rule gives
  \begin{align}\label{T1_split}
    \begin{aligned}
      T_{1,\eps}(x,\wt x)
      = &
      \al\,
      \EE\bigg[
        \int_\R
        \bigl(
          D_{n,y}\log\cZ_{\wt x}^{\be,\eps}
        \bigr)
        \int_{\DN}
          \phi^\eps\bigl(P_n(x,\ga)-y\bigr)
          Q_x^{\be,\eps}(d\ga)
        \,dy
      \bigg]
      +\cR_{1,\eps},
    \end{aligned}
  \end{align}
where
$$\cR_{1,\eps}
      = 
      \EE\bigg[
        \bigl(\log\cZ_{\wt x}^{\be,\eps}\bigr)
        \int_\R dy
        \int_{\DN}\mu_\be(d\ga)\,
       \big[ D_{n,y}G_x^\eps(\ga) \big]
        \al\phi^\eps\bigl(P_n(x,\ga)-y\bigr)
      \bigg].$$
  By the cancellation in Lemma~\ref{cancel_bd}, 
  the inner two integrals in
  $\cR_{1,\eps}$ vanish almost surely.  Hence
  \begin{align}\label{T1_after_cancel}
    \begin{aligned}
      T_{1,\eps}(x,\wt x)
      =
      \al\,
      \EE\bigg[
        \int_\R
        \bigl(
          D_{n,y}\log\cZ_{\wt x}^{\be,\eps}
        \bigr)
        \int_{\DN}
          \phi^\eps\bigl(P_n(x,\ga)-y\bigr)
          Q_x^{\be,\eps}(d\ga)
        \,dy
      \bigg].
    \end{aligned}
  \end{align}

  The deterministic factor relating $\cZ$ to the normalized partition
  function has zero Malliavin derivative.  Thus Proposition~\ref{prop_MD}
  and the $j=n$ case of \eqref{MD_H} yield
  \begin{align}\label{T1_DZ}
    \begin{aligned}
      D_{n,y}\log\cZ_{\wt x}^{\be,\eps}
      =
      \be\al
      \int_{\DN}
        \Big[
          \Phi^\eps\bigl(P_n(\wt x,\wt\ga)-y\bigr)
          -\Phi^\eps(-y)
        \Big]
        Q_{\wt x}^{\be,\eps}(d\wt\ga).
    \end{aligned}
  \end{align}
  Substituting \eqref{T1_DZ} into
  \eqref{T1_after_cancel} and applying Fubini once more gives
  \begin{align}\label{T1_conv}
    \begin{aligned}
      \frac1\be T_{1,\eps}(x,\wt x)
      =
      \al^2\EE\bigg[
        \int_{\DN^2}
        \bigg(\int_\R
        &\big[
          \Phi^\eps\bigl(P_n(\wt x,\wt\ga)-y\bigr)  -\Phi^\eps(-y)  \big]
        \phi^\eps\bigl(P_n(x,\ga)-y\bigr)\,dy \bigg)\, \\
        &\qquad\qquad\qquad\times
        Q_{\wt x}^{\be,\eps}(d\wt\ga) Q_x^{\be,\eps}(d\ga)   \bigg].
    \end{aligned}
  \end{align}
  Consequently,
  \begin{align}\label{T1_pre}
    \begin{aligned}
      \frac1\be T_{1,\eps}(x,\wt x)
      =\al^2\EE\bigg(  \int_{\DN^2}   
      \Big[   \Psi^\eps\bigl( P_n(\wt x,\wt\ga)-P_n(x,\ga)    \bigr)
          +\Psi^\eps\bigl(P_n(x,\ga)\bigr) \Big]
        Q_{\wt x}^{\be,\eps}(d\wt\ga)
        Q_x^{\be,\eps}(d\ga)    \bigg).
    \end{aligned}
  \end{align}
  Here, we again used the convolution identity
  \[
    \int_\R\Phi^\eps(a-y)\phi^\eps(b-y)\,dy
    =\Psi^\eps(a-b)
  \]
  and the oddness of $\Psi^\eps$ for the term containing
  $\Phi^\eps(-y)$.
  Therefore, it follows from 
  Proposition~\ref{T_uc} that

  \noi
  \begin{align}\label{T1_sgn}
    \begin{aligned}
      \lim_{\eps\downarrow0}
        \frac1\be T_{1,\eps}(x,\wt x)
      =\frac{\al^2}{2}\EE\bigg(
        \int_{\DN^2} 
        \big[ \sgn( P_n(\wt x,\wt\ga)-P_n(x,\ga)  )
          +\sgn(P_n(x,\ga))  \big] \,
        Q_{\wt x}^{\be}(d\wt\ga)Q_x^{\be}(d\ga)   \bigg).
    \end{aligned}
  \end{align}
  By \eqref{quen_t}, the contribution
  of the first sign function under the Brownian expectation is
  $ \sgn(\wt x-x)\,
    \rmP_{\wt x,x}^{\be}(\tau=\infty).$
The second sign function gives 
  \[
    \int_{\DN}\sgn\bigl(P_n(x,\ga)\bigr)Q_x^{\be}(d\ga)
    =2Q_x^{\be}\{x+ \ga_n>0\}-1.
  \]
  Spatial stationarity of the Brownian increments gives
  $ \EE[ Q_x^{\be}\{x+\ga_n>0\}] =\EE [Q_0^{\be}\{\ga_n>-x\} ]$;
    see also Remark \ref{rem19}.
  Plugging the last three identities into
  \eqref{T1_sgn} proves \eqref{bdry_lim}.
\end{proof}

\subsection{Completion of the main theorem and immediate corollaries}
\label{main_done}

We now restore both $\al$ and $\be$ in the superscripts.  Combining
Proposition \ref{bulk} and Proposition \ref{bdry}, and using
  \[
  \rmP_{\wt x,x}^{\al,\be}(\tau=\infty)
  =1-\rmP_{\wt x,x}^{\al,\be}(\tau\leq n-1),
  \]
we obtain
\begin{align*}
  \frac{\partial}{\partial x}
  \EE\bigl[F_{\wt x}^{\al,\be}F_x^{\al,\be}\bigr]
  &={}
  \frac12\sgn(\wt x-x)
    \rmP_{\wt x,x}^{\al,\be}(\tau\leq n-1)
  \\
  &\quad+
  \frac{\al^2}{2}\sgn(\wt x-x)
  \bigl[1-\rmP_{\wt x,x}^{\al,\be}(\tau\leq n-1)\bigr]
   +
  \al^2\EE\bigl[Q_0^{\al,\be}\{\ga_n>-x\}\bigr]
  -\frac{\al^2}{2}.
\end{align*}
Since
 $ \ind_{\{\wt x<x\}}
  =\frac12\bigl(1-\sgn(\wt x-x)\bigr)$ for $x\neq \wt x$,
the preceding display is exactly \eqref{main_eq}.  This completes
the proof of
Theorem~\ref{main_thm}. \hfill $\square$

\begin{proof}[Proof of Corollary~\ref{flat_cor}]
  Set $\al=0$ in Theorem~\ref{main_thm} and use
  $F_x^{\rfree,\be}=\be^{-1}\log Z_x^{\rfree,\be}$.  This gives the
  first identity in \eqref{flat_eq}.  Spatial stationarity implies
  that both
  $\EE[\log Z_x^{\rfree,\be}]$ and
  $\EE[(\log Z_x^{\rfree,\be})^2]$ are independent of $x$.  Hence
  differentiating the covariance gives the same derivative as the raw
  two-point expectation.  Moreover,
  \[
    \EE\left[
      \bigl(
        \log Z_{\wt x}^{\rfree,\be}
        -\log Z_x^{\rfree,\be}
      \bigr)^2
    \right]
    =2\EE\bigl[(\log Z_0^{\rfree,\be})^2\bigr]
     -2\EE\bigl[
       \log Z_{\wt x}^{\rfree,\be}
       \log Z_x^{\rfree,\be}
     \bigr].
  \]
  Differentiating this identity proves the final equality in
  \eqref{flat_eq}.
\end{proof}

\begin{proof}[Proof of Corollary~\ref{stat_exit}]
  Set $\al=1$ in Theorem~\ref{main_thm}.  The coefficient of the
  coalescence probability vanishes, and the remaining terms give
  \eqref{stat_exit_eq}.
\end{proof}

\begin{proof}[Proof of Corollary~\ref{stat_coal}]
  Fix $0\leq\wt x<x$ and write
  \[
    p(\al)
    =\rmP_{\wt x,x}^{\al,\be}(\tau\leq n-1).
  \]
  We first note that $p(\al)\to p(1)$ as $\al\to1$.  Indeed, for each
  fixed environment, the Gibbs densities depend continuously on
  $\al$.  The almost-sure sublinear growth of the terminal Brownian
  path, together with the exponential factor $e^{-\be\ga_n}$ in the
  reference measure, supplies an integrable majorant for $\al$ in a
  neighborhood of $1$.  Dominated convergence therefore gives local
  total-variation continuity of each quenched measure
  $Q_z^{\al,\be}$, and bounded convergence gives the asserted
  continuity of $p$.

  Since $x>0$, the event $\{\ga_n>-x\}$ is certain.  Also
  $\sgn(\wt x-x)=-1$.  Theorem~\ref{main_thm} and the identity
  $\log Z_z^{\al,\be}=\be F_z^{\al,\be}$ consequently yield
  \begin{align*}
    \frac{\partial}{\partial x}
    \EE\bigl[
      \log Z_{\wt x}^{\al,\be}\log Z_x^{\al,\be}
    \bigr]
    =\frac{\be^2(\al^2-1)}{2}\,p(\al).
  \end{align*}
  Differentiating at $\al=1$ and using the continuity of $p$ gives the
  first identity of the corollary.

  For the second identity, interchange the roles of $x$ and $\wt x$ in
  Theorem~\ref{main_thm}.  Since $\wt x\geq0$, the event
  $\{\ga_n>-\wt x\}$ is again certain, and
  $\sgn(x-\wt x)=1$.  Hence
  \begin{align*}
    \frac{\partial}{\partial\wt x}
    \EE\bigl[
      \log Z_{\wt x}^{\al,\be}\log Z_x^{\al,\be}
    \bigr]
    =\be^2\left[
      \frac{1-\al^2}{2}\,p(\al)+\al^2
    \right].
  \end{align*}
  After division by $\be^2$, its derivative at $\al=1$ equals
  $2-p(1)$, which proves the second identity.
\qedhere

\end{proof}

\subsection{The stationary Stein identity}
 \label{stein_sec}

Throughout this subsection the terminal strength is $\al=1$.  The
quenched switching identity from
Lemma~\ref{quen_lem} is essential here,
because the polymer in \eqref{stat_stein_eq} is multiplied by the
random weight $\partial_j\varphi(\mathbf F_{\mathbf y}^{\be})$.

\begin{proof}[Proof of Proposition~\ref{stat_stein}]
  For $0\leq\eps<1$, write
  \[
    F_z^{\be, \eps}
    =F_z^{\rstat,\be,\eps}
    =\frac1\be\log Z_z^{\rstat,\be,\eps}
    \AND
    \mathbf F_{\mathbf y}^{\be, \eps}
    =\bigl(F_{y_0}^{\be, \eps},\ldots,F_{y_m}^{\be, \eps}\bigr).
  \]
  Here, we use the normalized partition function from
  \eqref{def_ZR} with $R=\infty$ and $\lambda=0$.
 The deterministic
  factor $\be^{-n}$ does not affect its spatial or Malliavin
  derivatives, but retaining it is important because $\varphi$ is
  nonlinear.  Set
  \[
    G_\eps(x)
    =\EE\bigl[
      \varphi(\mathbf F_{\mathbf y}^{\be, \eps})F_x^{\be, \eps}
    \bigr].
  \]

  Fix a compact interval
  $K\Subset\R\setminus\{y_0,\ldots,y_m\}$.  For fixed $\eps>0$,
  Lemma~\ref{pf_uxeps1}, the moment estimates
  \eqref{LP_logZRL} and \eqref{moll_Dmom}, and the
  fact that a function with bounded gradient has at most linear growth
  justify differentiation under the expectation.  Thus
  \begin{align}
    G_\eps'(x)
    =
    \EE\bigl[
      \varphi(\mathbf F_{\mathbf y}^{\be, \eps})
      \partial_xF_x^{\be, \eps}
    \bigr],
    \qquad x\in K.
  \end{align}

  By Proposition~\ref{prop_MD}, every component of
  $\mathbf F_{\mathbf y}^{\be, \eps}$ belongs to $\DD^{1,2}$.
 The multivariate Malliavin chain rule 
  (\cite[Proposition 1.2.4]{Nua06}) therefore gives
  \[
    D\varphi(\mathbf F_{\mathbf y}^{\be, \eps})
    =
    \sum_{j=0}^{m}
      \partial_j\varphi(\mathbf F_{\mathbf y}^{\be, \eps})
      DF_{y_j}^{\be, \eps}.
  \]
  Next we apply Lemma~\ref{ibp} to the white-noise terms in
  \eqref{sp_der}.  When the Malliavin derivative
  acts on $\varphi(\mathbf F_{\mathbf y}^{\be, \eps})$, the preceding chain rule
  produces one copy of $DF_{y_j}^{\be, \eps}$.  When it acts on the Gibbs
  density of $Q_x^{\rstat,\be,\eps}$, its contribution vanishes by
  Lemmas~\ref{cancel_b} and
  \ref{cancel_bd}; these cancellations hold pathwise and may
  therefore be multiplied by the random factor
  $\varphi(\mathbf F_{\mathbf y}^{\be, \eps})$.  All uses of Fubini and
  Gaussian integration by parts are justified by the same moment
  bounds and by $\|\nabla\varphi\|_\infty<\infty$.

  We use the mollified four-point kernel $\Xi_{i,\eps}^{y,x}$ from
  \eqref{not_path}.  The calculation in
  \eqref{T2_Xi} and
  \eqref{T1_pre}, with the derivative of the first factor now
  supplied by the multivariate chain rule, yields
  \begin{align}\label{stein_pre}
    G_\eps'(x)
    &=
    \sum_{j=0}^{m}
    \EE\bigg[
      \partial_j\varphi(\mathbf F_{\mathbf y}^{\be, \eps})
      \int_{\DN^2}
      \bigg\{ 
      \sum_{i=0}^{n-1}
        \Xi_{i,\eps}^{y_j,x}(\eta,\ga)
      +\Psi^\eps\bigl(
        P_n(y_j,\eta)-P_n(x,\ga)
      \bigr)
      +\Psi^\eps\bigl(P_n(x,\ga)\bigr)
      \bigg\}
    \notag\\
    &\hspace{42mm}\times
      Q_{y_j}^{\rstat,\be,\eps}(d\eta)
      Q_x^{\rstat,\be,\eps}(d\ga)
    \bigg].
  \end{align}

  By \eqref{LPcvg_eps},
  $\mathbf F_{\mathbf y}^{\be, \eps}\to
  \mathbf F_{\mathbf y}^{\be}$ in probability.  Since each
  $\partial_j\varphi$ is bounded and continuous,
  \begin{align}
    \partial_j\varphi(\mathbf F_{\mathbf y}^{\be, \eps})
    \longrightarrow
    \partial_j\varphi(\mathbf F_{\mathbf y}^{\be})
    \qquad\text{in }L^1(\Omega).
  \end{align}
  Lemma~\ref{ker_conv}, applied with first starting point $y_j$,
  gives almost-sure uniform convergence on $K$ of every quenched kernel
  in \eqref{stein_pre}.  The kernels are bounded by a
  deterministic constant.  More explicitly, if
  $\mathcal K_{j,\eps}(x)$ denotes the quenched double integral in
  \eqref{stein_pre} and $\mathcal K_j(x)$ its sign-kernel
  limit, then
  \begin{align*}
    &\sup_{x\in K}
    \left|
      \EE\left[
        \partial_j\varphi(\mathbf F_{\mathbf y}^{\be, \eps})
        \mathcal K_{j,\eps}(x)
      \right]
      -
      \EE\left[
        \partial_j\varphi(\mathbf F_{\mathbf y}^{\be})
        \mathcal K_j(x)
      \right]
    \right|
    \\
    &\qquad\leq
    C\,
    \EE\left[
      \left|
        \partial_j\varphi(\mathbf F_{\mathbf y}^{\be, \eps})
        -
        \partial_j\varphi(\mathbf F_{\mathbf y}^{\be})
      \right|
    \right]
    +
    \|\partial_j\varphi\|_\infty
    \EE\left[
      \sup_{x\in K}
      |\mathcal K_{j,\eps}(x)-\mathcal K_j(x)|
    \right],
  \end{align*}
  and the right-hand side tends to zero.  Summing over $j$ gives
  \begin{align}\label{stein_du}
    \sup_{x\in K}|G_\eps'(x)-g(x)|
    \longrightarrow0.
  \end{align}
  Here
  \begin{align}\label{stein_sgn}
    g(x)
    &={}
    \frac12\sum_{j=0}^{m}
    \EE\bigg[
      \partial_j\varphi(\mathbf F_{\mathbf y}^{\be})
      \int_{\DN^2}
      \bigg\{
        \sum_{i=0}^{n-1}\mathfrak D_i^{y_j,x}(\eta,\ga)
        +\sgn\bigl(
          P_n(y_j,\eta)-P_n(x,\ga)
        \bigr)
        +\sgn\bigl(P_n(x,\ga)\bigr)
      \bigg\}
    \notag\\
    &\hspace{45mm}\times
      Q_{y_j}^{\rstat,\be}(d\eta)
      Q_x^{\rstat,\be}(d\ga)
    \bigg].
  \end{align}

  Since $\varphi$ has at most linear growth, \eqref{LPcvg_eps},
  H\"older's inequality, and \eqref{LP_logZRL} imply, for every fixed
  $x$,
  \[
    G_\eps(x)
    \longrightarrow
    G(x)=
    \EE\bigl[
      \varphi(\mathbf F_{\mathbf y}^{\be})
      F_x^{\rstat,\be}
    \bigr].
  \]
  The uniform convergence theorem for derivatives
  \cite[Theorem~7.17]{Rudin}, applied on $K$ using
  \eqref{stein_du}, now shows that $G$ is
  continuously differentiable on $K$ and that $G'=g$ there.

  Apply the quenched identity
  \eqref{quen_c} with $y=y_j$.  The first two
  terms inside the braces in \eqref{stein_sgn} reduce to
  $\sgn(y_j-x)$.  Hence
  \begin{align*}
    g(x)
    =
    \frac12\sum_{j=0}^{m}
    \EE\left[
      \partial_j\varphi(\mathbf F_{\mathbf y}^{\be})
      \left\{
        \sgn(y_j-x)
        +\int_{\DN}\sgn(x+\ga_n)
          Q_x^{\rstat,\be}(d\ga)
      \right\}
    \right].
  \end{align*}
  The quenched endpoint law is absolutely continuous with respect to
  Lebesgue measure, so it assigns no mass to $x+\ga_n=0$.  Therefore
  \[
    \frac12\left\{
      \sgn(y_j-x)
      +\int_{\DN}\sgn(x+\ga_n)
        Q_x^{\rstat,\be}(d\ga)
    \right\}
    =
    Q_x^{\rstat,\be}\{\ga_n>-x\}
    -\ind_{\{x>y_j\}}.
  \]
  Substituting this identity into the preceding display proves
  \eqref{stat_stein_eq}.  Since $K$ was arbitrary, the same argument proves continuous
  differentiability on every connected component of
  $\R\setminus\{y_0,\ldots,y_m\}$.
\end{proof}

\subsection{The Brownian Busemann process}
 \label{buse_sec}

\begin{proof}[Proof of Corollary \ref{buse_cor}]
  Fix $x_0<x_1<\cdots<x_m$ and write
  \[
    \widehat\Delta_k
    =
    F_{x_k}^{\rstat,\be}-F_{x_{k-1}}^{\rstat,\be},
    \qquad 1\leq k\leq m.
  \]
  Let $\varphi\in C^1(\R^m)$ have bounded gradient and define
  \[
    \widehat\varphi(z_0,\ldots,z_m)
    =
    \varphi(z_1-z_0,\ldots,z_m-z_{m-1}).
  \]
  Its derivatives satisfy
  \begin{align*}
    \partial_0\widehat\varphi
    &=-\partial_1\varphi,
    \quad   \partial_m\widehat\varphi
    =\partial_m\varphi,
    \AND
    \partial_j\widehat\varphi
    =\partial_j\varphi-\partial_{j+1}\varphi
    \quad
    \text{for  $1\leq j\leq m-1$}.
  \end{align*}
  In particular,
  \begin{align}\label{buse_tel}
    \sum_{j=0}^{m}\partial_j\widehat\varphi
    =0
    \AND
    \sum_{j=0}^{k-1}\partial_j\widehat\varphi
    =-\partial_k\varphi.
  \end{align}

  Fix $k\in\{1,\ldots,m\}$ and let
  $s\in(x_{k-1},x_k)$.  
  Then, we can apply
  Proposition~\ref{stat_stein} with
    $(y_0,\ldots,y_m)=(x_0,\ldots,x_m)$
  and with the test function $\widehat\varphi$.  Since
  \[
    \widehat\varphi\bigl(
      F_{x_0}^{\rstat,\be},\ldots,F_{x_m}^{\rstat,\be}
    \bigr)
    =
    \varphi(\widehat\Delta_1,\ldots,\widehat\Delta_m),
  \]
  Proposition~\ref{stat_stein} gives
  \begin{align}\label{buse_loc1}
    &\frac{\partial}{\partial s}
    \EE\bigl[
      \varphi(\widehat\Delta_1,\ldots,\widehat\Delta_m)
      F_s^{\rstat,\be}
    \bigr]
    \notag\\
    &\quad=
    \sum_{j=0}^{m}
    \EE\left[
      \partial_j\widehat\varphi
      \bigl(
        F_{x_0}^{\rstat,\be},\ldots,F_{x_m}^{\rstat,\be}
      \bigr)
      \left(
        Q_s^{\rstat,\be}\{\ga_n>-s\}
        -\ind_{\{s>x_j\}}
      \right)
    \right].
  \end{align}
  We first consider the terms containing the quenched exit
  probability.  Since
  $Q_s^{\rstat,\be}\{\ga_n>-s\}$ does not depend on $j$,
  \begin{align*}
    &\sum_{j=0}^{m}
    \EE\left[
      \partial_j\widehat\varphi
      \bigl(
        F_{x_0}^{\rstat,\be},\ldots,F_{x_m}^{\rstat,\be}
      \bigr)
      Q_s^{\rstat,\be}\{\ga_n>-s\}
    \right]
    \\
    &\qquad=
    \EE\bigg[
      Q_s^{\rstat,\be}\{\ga_n>-s\}
      \sum_{j=0}^{m}
      \partial_j\widehat\varphi
      \bigl(
        F_{x_0}^{\rstat,\be},\ldots,F_{x_m}^{\rstat,\be}
      \bigr)
    \bigg]
    =0,
  \end{align*}
  where the last equality follows from the first identity in
  \eqref{buse_tel}.  Notice that this cancellation is pointwise in
  the environment; no independence between the exit probability and
  the free energies is being used.

  It remains to consider the indicator terms.  Since
  $s\in(x_{k-1},x_k)$,
  \[
    \ind_{\{s>x_j\}}
    =
    \begin{cases}
      1, & 0\leq j\leq k-1,\\
      0, & k\leq j\leq m.
    \end{cases}
  \]
  Therefore,  \eqref{buse_loc1} becomes
  \begin{align}
    \frac{\partial}{\partial s}
    \EE\bigl[
      \varphi(\widehat\Delta_1,\ldots,\widehat\Delta_m)
      F_s^{\rstat,\be}
    \bigr]
    &=
    -
    \EE\bigg[
      \sum_{j=0}^{k-1}
      \partial_j\widehat\varphi
      \bigl(
        F_{x_0}^{\rstat,\be},\ldots,F_{x_m}^{\rstat,\be}
      \bigr)
    \bigg]  \notag \\
   &=  \EE\bigl[
      \partial_k\varphi(
        \widehat\Delta_1,\ldots,\widehat\Delta_m
      )
    \bigr],  \label{buse_loc}
  \end{align}
where we used the second identity in \eqref{buse_tel}.
 Next, we integrate first over
  $[x_{k-1}+\delta,x_k-\delta]$ and then let $\delta\downarrow0$.
  Since
  $\cZ_x^{\rstat,\be}=\be^nZ_x^{\rstat,\be}$, the continuity of
  $x\mapsto\cZ_x^{\rstat,\be}$ from Lemma~\ref{conv_Z}, together with
  the moment bound
  \eqref{LP_logZRL}, justifies the endpoint limit.
  In this way, we have the following
  multivariate Stein identity
  \begin{align}
    \EE\bigl[
      \varphi(\widehat\Delta_1,\ldots,\widehat\Delta_m)\widehat\Delta_k
    \bigr]
    =
    (x_k-x_{k-1})
    \EE\bigl[
      \partial_k\varphi(\widehat\Delta_1,\ldots,\widehat\Delta_m)
    \bigr].
  \end{align}
  Thus, by Stein's lemma (see, e.g., \cite[Lemma 1]{Mec09}),
  the increments are independent centered Gaussian variables with
  the asserted variances.  Finally, the identity
  $\cZ_x^{\rstat,\be}=\be^nZ_x^{\rstat,\be}$ and
  Lemma~\ref{conv_Z} imply that
  $x\mapsto F_x^{\rstat,\be}$ is almost surely continuous.  Since
  $\cB^{\be}(0)=0$, this identifies $\cB^{\be}$ as a two-sided
  standard Brownian motion.
\end{proof}

\subsection{The zero-temperature limit}
\label{zero_sec}

For the zero-temperature argument, it is convenient to absorb the
exponential reference density into the Hamiltonian.  In the notation of
\eqref{not_zero}, the identities
\eqref{mu} and \eqref{Z_gen} give

\noi
\begin{align}\label{lap_ZQ}
  \begin{aligned}
    Z_z^{\al,\be}
      &=\int_{\DN}
        e^{\be\cH_z^{\al}(\ga)}\,d\ga
    \AND
    Q_z^{\al,\be}(d\ga)
      &=\frac{e^{\be\cH_z^{\al}(\ga)}}
        {Z_z^{\al,\be}}\,d\ga.
  \end{aligned}
\end{align}
Thus the limit $\be\to\infty$ is a Laplace-principle limit on the
unbounded simplex $\DN$.

\begin{lemma}[\textsf{Uniqueness of the BLPP geodesic}]
  \label{geo_uniq}
  Fix $n\in\Z_{>0}$, $\al\in\R$, and $z\in\R$.  Almost surely, the
  variational problem defining $L_z^{\al}$ has a unique maximizer in
  $\overline{\DN}$.
\end{lemma}

\begin{proof}
  Fix $\theta\in(1/2,1)$.  Lemma
  \ref{B_poly} gives, almost surely,
  \begin{align}\label{geo_coer}
    \cH_z^{\al}(\ga)
    \leq C_\omega(1+\ga_n^\theta)-\ga_n,
    \qquad \ga\in\overline{\DN}.
  \end{align}
  The right-hand side tends to $-\infty$ as $\ga_n\to\infty$.
  Since $\cH_z^{\al}$ is continuous, its supremum is finite and is
  attained in
  \[
    K_R=\{\ga\in\overline{\DN}:\ga_n\leq R\}
  \]
  for some finite random $R$.

  It remains to prove uniqueness.  For
  $\ga,\eta\in\overline{\DN}$, set
  \[
    I_i(\ga)=(z+\ga_i,z+\ga_{i+1}],
    \qquad 0\leq i\leq n-1.
  \]
  Independence of the Brownian motions on the different levels gives
  \begin{align*}
    \Var\bigl(
      H_z^{\al}(\ga)-H_z^{\al}(\eta)
    \bigr)
    &\geq
    \sum_{i=0}^{n-1}
    \bigl\|
      \ind_{I_i(\ga)}-\ind_{I_i(\eta)}
    \bigr\|_{L^2(\R)}^2.
  \end{align*}
  If the right-hand side were zero, then the corresponding intervals
  would agree up to null sets.  Starting from
  $\ga_0=\eta_0=0$ and proceeding level by level would give
  $\ga_i=\eta_i$ for every $i$, contrary to $\ga\neq\eta$.  Hence
  every nontrivial increment has strictly positive variance.

  The space $\overline{\DN}$ is $\sigma$-compact, and
  $\ga\mapsto\cH_z^{\al}(\ga)$ is a Gaussian process with continuous
  sample paths.  The deterministic term $-\ga_n$ does not affect
  increment variances, so
  \[
    \Var\bigl(
      \cH_z^{\al}(\ga)-\cH_z^{\al}(\eta)
    \bigr)>0,
    \qquad \ga\neq\eta.
  \]
  Kim--Pollard's uniqueness criterion
  \cite[Lemma~2.6]{KP90} implies that the supremum is attained at
  most one point.  The coercivity argument above shows that it is
  attained, and therefore the maximizer is almost surely unique.
\end{proof}

\begin{lemma}[\textsf{Quenched zero-temperature concentration}]
  \label{Q_geo}
  Fix $n\in\Z_{>0}$, $\al\in\R$, and $z\in\R$.  Then, almost surely,
  \begin{align}\label{Q_zero}
    Q_z^{\al,\be}
    \Longrightarrow
    \delta_{\Gamma_z^{\al}}
    \qquad\text{as }\be\to\infty
  \end{align}
  weakly on $\overline{\DN}$.
\end{lemma}

\begin{proof}
  Work on a Brownian realization for which
  Lemma~\ref{geo_uniq} and the coercive estimate
  \eqref{geo_coer} hold.  Write
  \[
    L=L_z^{\al},
    \qquad
    \Gamma=\Gamma_z^{\al},
    \AND
    \cH=\cH_z^{\al}.
  \]
  Let $U$ be an open neighborhood of $\Gamma$ in
  $\overline{\DN}$.  Continuity and coercivity imply that the supremum
  of $\cH$ on the closed set $U^c$ is attained.  By uniqueness,
  \[
    M_U=\sup_{\ga\in U^c}\cH(\ga)<L.
  \]
  Choose $\dl>0$ such that
  \begin{align}\label{comp_gap}
    M_U\leq L-5\dl.
  \end{align}
  By continuity of $\cH$ and density of $\DN$ in
  $\overline{\DN}$, there exists a bounded measurable set
  $V\subset U\cap\DN$ of positive Lebesgue measure such that
  \begin{align} \notag 
    \inf_{\ga\in V}\cH(\ga)\geq L-\dl.
  \end{align}

  The polynomial coercive bound also yields a linear tail estimate.
  After increasing $R$, we may assume that
  $\Gamma\in K_R$, $V\subset K_R$, and
  \begin{align}\label{tail_gap}
    \cH(\ga)
    \leq L-4\dl-\frac14(\ga_n-R),
    \qquad \ga_n\geq R.
  \end{align}
  Indeed, the derivative of
  $t\mapsto C_\omega(1+t^\theta)-t$ is eventually less than
  $-1/4$, while the function itself tends to $-\infty$.

  The denominator in \eqref{lap_ZQ} is bounded below by
  $|V|e^{\be(L-\dl)}$.  Splitting $U^c$ into
  $K_R\cap U^c$ and $K_R^c$, and then using
  \eqref{comp_gap} and \eqref{tail_gap}, gives
  \begin{align}\label{zero_conc}
    Q_z^{\al,\be}(U^c)
    &\leq
    \frac{|K_R|}{|V|}e^{-4\be\dl}
  +
    \frac{e^{-3\be\dl}}{|V|}
    \int_R^\infty
      \frac{t^{n-1}}{(n-1)!}
      e^{-\be(t-R)/4}\,dt.
  \end{align}
  The remaining integral has the explicit expansion
  \begin{align}  \notag   
    \int_R^\infty
      \frac{t^{n-1}}{(n-1)!}
      e^{-\be(t-R)/4}\,dt
    =
    \sum_{j=0}^{n-1}
      \frac{R^{n-1-j}}{(n-1-j)!}
      \left(\frac4\be\right)^{j+1}.
  \end{align}
  Hence both terms in \eqref{zero_conc} tend to
  zero.  Every open neighborhood of $\Gamma$ therefore receives
  asymptotic mass one, which is equivalent to
  \eqref{Q_zero}.
\end{proof}

Weak convergence can be applied to the indicator functions in the
main identity only after verifying that their boundaries carry no mass
under the limiting Dirac measures.  We record the two facts needed
below.

\begin{lemma}[\textsf{Continuity sets for the BLPP geodesics}]
  \label{cont_sets}
  Fix $n\in\Z_{>0}$, $\al\in\R$, and $x,\wt x\in\R$ with
  $x\neq\wt x$.  Then
  \begin{align}\label{exit_atom}
    \PP\bigl(\Gamma_0^{\al}(n)=-x\bigr)=0.
  \end{align}
  Moreover, if
  \[
    A_{\wt x,x}
    =\bigl\{(\wt\ga,\ga):
      \tau_{\wt x,x}(\wt\ga,\ga)\leq n-1
    \bigr\},
  \]
  then
  \begin{align}\label{coal_bd0}
    \PP\bigl(
      (\Gamma_{\wt x}^{\al},\Gamma_x^{\al})
      \in\partial A_{\wt x,x}
    \bigr)=0.
  \end{align}
\end{lemma}

\begin{proof}
  For
  $z\in\R$, write
  \[
    T_i(z)=z+\Gamma_z^{\al}(i),
    \qquad 0\leq i\leq n,
  \]
  with $T_0(z)=z$.  Every finite segment of the maximizing path
  $\Gamma_z^{\al}$ is a point-to-point BLPP geodesic: otherwise one
  could replace that segment by a path with larger weight and thereby
  increase $\cH_z^{\al}$.  Consequently,
  \cite[Proposition~5.1]{RS26} applies to these segments and rules out
  two consecutive vertical edges along $\Gamma_z^{\al}$.

  We claim that, for every deterministic $c\in\R$ and
  $1\leq r\leq n$,
  \begin{align}\label{fixed_jump}
    \PP\bigl(T_r(z)=c\bigr)=0.
  \end{align}
  First suppose that $n \geq 2$ and $1\leq r\leq n-1$.  If the neighboring jump
  locations are fixed and $t$ is varied in the interval
  $[T_{r-1}(z),T_{r+1}(z)]$, then the only part of the passage value that
  depends on $t$ is
  \[
    B_{r-1}(t)-B_r(t).
  \]
  Hence $T_r(z)$ maximizes this Brownian motion on the interval between
  its two neighboring jump locations. 
    On the event $\{T_r(z)=c\}$,
    we have
  \[
    T_{r-1}(z)\leq c\leq T_{r+1}(z).
  \]
  Except when $r=1$ and $c=z$, both inequalities are in fact
  strict.  Indeed, if $r\geq2$ and
  $T_{r-1}(z)=T_r(z)=c$, then the horizontal segment on level
  $r-1$ has zero length, and the maximizing path has two consecutive
  vertical edges.  Similarly,
  $T_r(z)=T_{r+1}(z)=c$ would give two consecutive vertical edges
  at the next pair of levels.  Both possibilities are excluded by
  the preceding paragraph.  The only exception is
  $r=1$ and $c=z$, since $T_0(z)=z$ is the prescribed starting
  point and there is no vertical edge below level $0$.

  Consider first the nonboundary case.  On $\{T_r(z)=c\}$ we then
  have
  $T_{r-1}(z)<c<T_{r+1}(z).$
  By density of the rationals, there exist $a,b\in\mathbb{Q}$ such that
  \[
    T_{r-1}(z)<a<c<b<T_{r+1}(z).
  \]
  Since $c=T_r(z)$ maximizes
  $B_{r-1}-B_r$ on
  $[T_{r-1}(z),T_{r+1}(z)]$, it also maximizes this process on
  the smaller interval $[a,b]$.  Consequently,
  \begin{align}\label{bm1}
    \{T_r(z)=c\}
    \subset
\bigcup_{\substack{a,b\in\mathbb{Q}\\a<c<b}}
    \Big\{
      c\in
      \operatorname*{argmax}_{t\in[a,b]}
      \bigl(B_{r-1}(t)-B_r(t)\bigr)
    \Big\},
  \end{align}
  up to the exceptional case $r=1$, $c=z$.

  In that exceptional case, choose $b\in\mathbb{Q}$ such that
  \[
    c<b<T_2(z).
  \]
  Since $T_0(z)=T_1(z)=c$, the point $c$ is a maximizer of
  $B_0-B_1$ on $[c,b]$.  Thus this case is contained in the
  countable union
\begin{equation}\label{bm2}
\bigcup_{\substack{b\in\mathbb{Q}\\b>c}}
    \Big\{
      c\in
      \operatorname*{argmax}_{t\in[c,b]}
      \bigl(B_0(t)-B_1(t)\bigr)
    \Big\}.
\end{equation}

Inside each of the argmax above in \eqref{bm1} and \eqref{bm2},
  $B_{r-1}-B_r$ is a Brownian motion (with variance parameter $2$).
In particular, the location of its maximum follows the arc-sine law, 
$$ \mathbb P\Big(\textup{argmax}_{t\in[a,b]}
      \bigl(B_{r-1}(t)-B_r(t)\bigr)\in dx\Big) = \frac{1}{\pi\sqrt{(x-a)(b-x)}}\ind_{(a,b)}(x)dx,$$
which is a continuous distribution. 
Hence every event in the two
  countable unions has probability zero.  This proves
  \eqref{fixed_jump} for $r<n$.

  For $r=n$, once the preceding coordinates are fixed, varying the
  terminal coordinate $t\geq T_{n-1}(z)$ changes the passage value only
  through
  \[
    B_{n-1}(t)+\al B_n(t)-t.
  \]
  This is a Brownian motion with variance parameter
  $1+\al^2$ and drift $-1$.  If $T_{n-1}(z)<c$, choose a rational
  $a$ with $T_{n-1}(z)<a<c$; then $c$ must be the maximizer of this
  drifted Brownian motion on $[a,\infty)$.  If
  $T_{n-1}(z)=c$, two consecutive vertical edges are present unless
  $n=1$ and $c=z$; in the latter case use the interval
  $[c,\infty)$.  In all cases, the location of the maximum has a
  continuous distribution, again by \cite{LP18}.  This proves
  \eqref{fixed_jump} for $r=n$.  Taking $z=0$, $r=n$, and $c=-x$
  gives \eqref{exit_atom}.

\begin{figure}[t]
  \centering

  \begin{tikzpicture}[
    x=.95cm,
    y=.78cm,
    >=Stealth,
    font=\small,
    bluegeo/.style={blue,line width=1.1pt},
    redgeo/.style={red,line width=1.1pt},
    levline/.style={gray!55,dashed,line width=.55pt}
  ]

    \begin{scope}

      \foreach \yy in {0,1,2,3}{
        \draw[levline] (0,\yy) -- (5.35,\yy);
      }
      \node[right] at (5.42,0) {$y=0$};
      \node[right] at (5.42,3) {$y=n$};
      \node at (2.70,3.55) {original geodesics};

      \draw[bluegeo]
        (.55,0) -- (1.15,0) -- (1.15,1)
        -- (3.00,1) -- (3.00,2)
        -- (3.75,2) -- (3.75,3);

      \draw[redgeo]
        (2.15,0) -- (3.00,0) -- (3.00,1)
        -- (4.05,1) -- (4.05,2)
        -- (4.80,2) -- (4.80,3);

      \fill[blue] (.55,0) circle (1.65pt);
      \fill[blue] (1.15,1) circle (1.65pt);
      \fill[blue] (3.00,2) circle (1.65pt);
      \fill[blue] (3.75,3) circle (1.65pt);

      \draw[blue,fill=white,line width=.9pt]
        (1.15,0) circle (2.0pt);
      \draw[blue,fill=white,line width=.9pt]
        (3.75,2) circle (2.0pt);

      \fill[red] (2.15,0) circle (1.65pt);
      \fill[red] (4.05,2) circle (1.65pt);
      \fill[red] (4.80,3) circle (1.65pt);

      \draw[red,fill=white,line width=.9pt]
        (3.00,0) circle (2.0pt);
      \draw[red,fill=white,line width=.9pt]
        (4.05,1) circle (2.0pt);
      \draw[red,fill=white,line width=.9pt]
        (4.80,2) circle (2.0pt);

      %
      %
      \fill[red] (3.00,1) circle (1.65pt);
      \draw[blue,line width=.9pt]
        (3.00,1) circle (2.45pt);

      \node[below] at (.55,0) {$\wt x$};
      \node[below] at (2.15,0) {$x$};

      \node[
        blue,
        font=\scriptsize,
        anchor=south east
      ] at (2.92,2.08)
        {$\wt x+\wt\ga_r=q$};

      \node[
        red,
        font=\scriptsize,
        anchor=north west
      ] at (3.08,.92)
        {$x+\ga_{r-1}=q$};

      \draw[gray!55,densely dotted]
        (3.00,-.05) -- (3.00,2.95);

    \end{scope}

    \begin{scope}[xshift=7.15cm]

      \foreach \yy in {0,1,2,3}{
        \draw[levline] (0,\yy) -- (5.35,\yy);
      }
      \node[right] at (5.42,0) {$y=0$};
      \node[right] at (5.42,3) {$y=n$};
      \node at (2.70,3.55) {exchanged continuations};

      %
      \draw[bluegeo]
        (.55,0) -- (1.15,0) -- (1.15,1)
        -- (4.05,1) -- (4.05,2)
        -- (4.80,2) -- (4.80,3);

      %
      \draw[redgeo]
        (2.15,0) -- (3.00,0) -- (3.00,1);

      \draw[redgeo,densely dashed]
        (3.00,1) -- (3.00,2);

      \draw[redgeo]
        (3.00,2) -- (3.75,2) -- (3.75,3);

      \fill[blue] (.55,0) circle (1.65pt);
      \fill[blue] (1.15,1) circle (1.65pt);
      \fill[blue] (4.05,2) circle (1.65pt);
      \fill[blue] (4.80,3) circle (1.65pt);

      \draw[blue,fill=white,line width=.9pt]
        (1.15,0) circle (2.0pt);
      \draw[blue,fill=white,line width=.9pt]
        (4.05,1) circle (2.0pt);
      \draw[blue,fill=white,line width=.9pt]
        (4.80,2) circle (2.0pt);

      \fill[red] (2.15,0) circle (1.65pt);
      \fill[red] (3.00,2) circle (1.65pt);
      \fill[red] (3.75,3) circle (1.65pt);

      \draw[red,fill=white,line width=.9pt]
        (3.00,0) circle (2.0pt);
      \draw[red,fill=white,line width=.9pt]
        (3.75,2) circle (2.0pt);

      \fill[red] (3.00,1) circle (1.65pt);
      \draw[red,line width=.9pt]
        (3.00,1) circle (2.45pt);

      \node[below] at (.55,0) {$\wt x$};
      \node[below] at (2.15,0) {$x$};
      \node[red,anchor=west] at (3.12,1.05) {$q$};

      \node[
        red,
        font=\scriptsize,
        align=center
      ] at (1.90,2.48)
        {two consecutive\\vertical jumps};

      \draw[red,->,line width=.7pt]
        (2.48,2.27) -- (2.94,1.72);

      \draw[gray!55,densely dotted]
        (3.00,-.05) -- (3.00,2.95);

    \end{scope}

  \end{tikzpicture}

  \captionsetup{width=.90\linewidth}
  \caption{As shown on the left, the two BLPP geodesics meet on the event
    $\{\wt x+\wt\ga_r=x+\ga_{r-1}\}$.
    In this case, the red and blue paths on the right-hand side are
    also geodesics.  In particular, the red geodesic has two
    consecutive vertical jumps.
  }
  \label{fig_BLPP}
\end{figure}
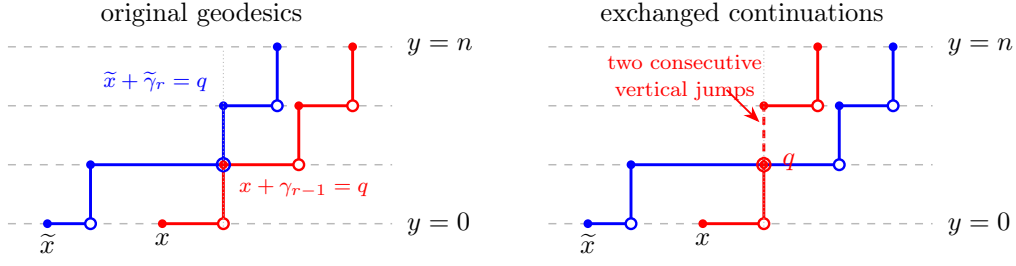

  We now prove the continuity-set statement for the meeting event. This part of the argument is illustrated in Figure \ref{fig_BLPP}. 
  Assume for definiteness that $\wt x<x$.  From the definition of
  $\tau$,
  \[
    A_{\wt x,x}
    =\bigcup_{k=0}^{n-1}
      \bigg[
        \bigcap_{r=1}^{k}
          \{\wt x+\wt\ga_r<x+\ga_{r-1}\}
        \cap
          \{\wt x+\wt\ga_{k+1}\geq x+\ga_k\}
      \bigg].
  \]
  Consequently,
  \begin{align}\label{coal_inc}
    \partial A_{\wt x,x}
    \subset
    \bigcup_{r=1}^{n}
      \{\wt x+\wt\ga_r=x+\ga_{r-1}\}.
  \end{align}
Set
  \[
    \wt t_i=\wt x+\Gamma_{\wt x}^{\al}(i),
    \qquad
    t_i=x+\Gamma_x^{\al}(i).
  \]
  The event in \eqref{coal_inc} with $r=1$ is
  $\{\wt t_1=x\}$ and has probability zero by
  \eqref{fixed_jump}.  Fix $2\leq r\leq n$ and suppose, toward a
  contradiction, that
  \[
    \wt t_r=t_{r-1}=q.
  \]
  For a nondecreasing continuation
  $q\leq u_r\leq\cdots\leq u_n$, define its future score from
  $(q,r-1)$ by
  \begin{align*}
    \cS_{r-1,q}(u_r,\ldots,u_n)
    ={}&
    B_{r-1}(u_r)-B_{r-1}(q)
    \\
    &+\sum_{i=r}^{n-1}
      \bigl(B_i(u_{i+1})-B_i(u_i)\bigr)
    +\al B_n(u_n)-u_n.
  \end{align*}
  Splitting each passage values at $(q,r-1)$ gives
  \begin{align*}
    \cH_{\wt x}^{\al}(\Gamma_{\wt x}^{\al})
      &=C_{\wt x}
        +\cS_{r-1,q}(\wt t_r,\ldots,\wt t_n),
    \\
    \cH_x^{\al}(\Gamma_x^{\al})
      &=C_x
        +\cS_{r-1,q}(t_r,\ldots,t_n),
  \end{align*}
  where $C_{\wt x}$ and $C_x$ contain only the respective prefixes up
  to $(q,r-1)$.  Exchanging the two continuations produces admissible
  point-to-line paths from both starting points.  Optimality of the
  original two paths therefore gives
  \begin{align*}
    \cS_{r-1,q}(\wt t_r,\ldots,\wt t_n)
      &\geq
      \cS_{r-1,q}(t_r,\ldots,t_n),
    \\
    \cS_{r-1,q}(t_r,\ldots,t_n)
      &\geq
      \cS_{r-1,q}(\wt t_r,\ldots,\wt t_n).
  \end{align*}
  Hence equality holds, and the path from $(x,0)$ obtained by keeping
  its prefix and attaching the continuation of
  $\Gamma_{\wt x}^{\al}$ is also maximizing.  This new path has two
  consecutive vertical edges at the same location $q$, since its
  $(r-1)$st and $r$th jump locations are both $q$.  This contradicts
  \cite[Proposition~5.1]{RS26}.  Thus every event in the finite union
  on the right-hand side of \eqref{coal_inc} has probability zero,
  which proves \eqref{coal_bd0}.  The case $\wt x>x$ follows by
  interchanging the two paths.
\end{proof}

\begin{lemma}[\textsf{Zero-temperature convergence of the free energy}]
  \label{FE_blpp}
  Fix $n\in\Z_{>0}$, $\al\in\R$, and a compact interval
  $J\subset\R$.  For every $p<\infty$,
  \begin{align}\label{unif_Lp0}
    \EE\Big[
      \sup_{u\in J}
      |F_u^{\al,\be}-L_u^{\al}|^p
    \Big]
    \longrightarrow0
    \qquad\text{as }\be\to\infty.
  \end{align}
  In particular, for fixed $u,v\in\R$,
  \[
    \EE\bigl[F_u^{\al,\be}F_v^{\al,\be}\bigr]
    \longrightarrow
    \EE\bigl[L_u^{\al}L_v^{\al}\bigr].
  \]
\end{lemma}

\begin{proof}
  Fix a Brownian realization for which
  Lemma~\ref{B_poly} holds.  Write
  \[
    \cH_u(\ga)=\cH_u^{\al}(\ga),
    \qquad
    L(u)=L_u^{\al},
\AND    K_R=\{\ga\in\overline{\DN}:\ga_n\leq R\}.
  \]
  Uniformly in $u\in J$,
  \begin{align}\label{unif_coer}
    \cH_u(\ga)
    \leq C_JM_\theta(1+\ga_n^\theta)-\ga_n.
  \end{align}
  Hence all maximizers with $u\in J$ lie in one random compact simplex
  $K_R$ once $R$ is sufficiently large.

  We first prove a uniform Laplace principle on a fixed $K_R$.  Set
  \[
    L_R(u)=\max_{\ga\in K_R}\cH_u(\ga).
  \]
  The map $(u,\ga)\mapsto\cH_u(\ga)$ is uniformly continuous on
  $J\times K_R$.  Given $\eta>0$, choose $r>0$ so that
  \[
    |\cH_u(\ga)-\cH_u(\eta')|\leq\eta
  \]
  whenever $u\in J$, $\ga,\eta'\in K_R$, and
  $|\ga-\eta'|\leq r$.  Since $K_R$ is a compact simplex, there exists
  $c_{R,r}>0$ such that
  \begin{align}
    |B(\ga,r)\cap K_R\cap\DN|
    \geq c_{R,r},
    \qquad \ga\in K_R.
  \end{align}
  If $\ga_u$ maximizes $\cH_u$ on $K_R$, then
  \begin{align}\label{low_B13}
    c_{R,r}e^{\be(L_R(u)-\eta)}
    &\leq
    \int_{K_R\cap\DN}
      e^{\be\cH_u(\ga)}\,d\ga
    \leq
    |K_R|e^{\be L_R(u)}.
  \end{align}
  Taking logarithms gives
  \begin{align}\label{comp_lap}
    \sup_{u\in J}
    \left|
      \frac1\be\log
      \int_{K_R\cap\DN}
        e^{\be\cH_u(\ga)}\,d\ga
      -L_R(u)
    \right|
    \longrightarrow0.
  \end{align}

  We next control the unbounded part of the simplex.  The common
  compactness of all maximizers and uniform continuity on
  $J\times K_R$ show that $u\mapsto L(u)$ is continuous.  Put
  $m_J=\inf_{u\in J}L(u)$.  By increasing $R$, the argument used for
  \eqref{tail_gap} gives
  \begin{align}
    \sup_{u\in J}\cH_u(\ga)
    \leq m_J-2-\frac14(\ga_n-R),
    \qquad \ga_n\geq R.
  \end{align}
  For this choice of $R$, $L_R(u)=L(u)$ for every $u\in J$, and
  \begin{align*}
    \int_{K_R^c\cap\DN}
      e^{\be\cH_u(\ga)}\,d\ga
    &\leq
    e^{\be(m_J-2)}
    \int_R^\infty
      \frac{t^{n-1}}{(n-1)!}
      e^{-\be(t-R)/4}\,dt.
  \end{align*}
  On the other hand, \eqref{low_B13} with $\eta=1/2$ gives, uniformly
  in $u\in J$,
  \[
    \int_{K_R\cap\DN}
      e^{\be\cH_u(\ga)}\,d\ga
    \geq
    c_{R,r}e^{\be(m_J- \frac12)}.
  \]
  Therefore, for $\be\geq1$,
  \begin{align*}
    \sup_{u\in J}
    \frac{
      \int_{K_R^c\cap\DN}e^{\be\cH_u(\ga)}\,d\ga
    }{
      \int_{K_R\cap\DN}e^{\be\cH_u(\ga)}\,d\ga
    }
    &\leq C_{R,n}e^{-3\be/2}
    \longrightarrow0.
  \end{align*}
  Combining this estimate with
  \eqref{comp_lap} and
  \eqref{lap_ZQ} yields
  \begin{align}\label{as_zero}
    \sup_{u\in J}
    |F_u^{\al,\be}-L_u^{\al}|
    \longrightarrow0
    \qquad\text{almost surely}.
  \end{align}

  It remains to obtain an integrable domination uniform in
  $\be\geq1$.  Young's inequality applied to
  \eqref{unif_coer} gives
  \begin{align}
    \sup_{u\in J}\cH_u(\ga)
    \leq Y_J-\frac12\ga_n,
    \quad {\rm with}
    \quad
    Y_J=C\bigl(1+M_\theta^{1/(1-\theta)}\bigr).
  \end{align}
  Hence
  \begin{align*}
    \sup_{u\in J}L(u)
      \leq Y_J
    \AND
    Z_u^{\al,\be}
      \leq
      e^{\be Y_J}
      \int_{\DN}e^{-\be\ga_n/2}\,d\ga
      =e^{\be Y_J}\left(\frac2\be\right)^n.
  \end{align*}
  For a lower bound, fix a deterministic bounded open set
  $U_0\Subset\DN$ of positive Lebesgue measure.  Then
  \begin{align*}
    F_u^{\al,\be}
      \geq
      \inf_{\ga\in U_0}\cH_u(\ga)
      +\frac1\be\log|U_0|
    \AND
    L(u)
      \geq
      \inf_{\ga\in U_0}\cH_u(\ga).
  \end{align*}
  Since $J\times\overline{U_0}$ is compact, the absolute value of the
  last infimum is bounded by $CM_\theta$.  It follows that there exists
  a random variable $Y_J'$ with moments of every finite order such that
  \begin{align}
    \sup_{\be\geq1}\sup_{u\in J}
    \bigl(
      |F_u^{\al,\be}|+|L_u^{\al}|
    \bigr)
    \leq Y_J'.
  \end{align}
  Dominated convergence applied to \eqref{as_zero} proves
  \eqref{unif_Lp0}.  The convergence of mixed second
  moments follows from Cauchy--Schwarz and the cases $p=2$ and $p=4$.
\end{proof}

The next elementary stability fact will be used to prove continuity of
the limiting right-hand side.

\begin{lemma}[\textsf{Stability of a unique geodesic}]
  \label{geo_stab}
  Let $z_m\to z$.  On the event that
  $\Gamma_{z_m}^{\al}$ and $\Gamma_z^{\al}$ are unique for all $m$,
  $
    \Gamma_{z_m}^{\al}
    \longrightarrow
    \Gamma_z^{\al}$
    in $\overline{\DN}$.
\end{lemma}

\begin{proof}
  The coercive estimate is uniform for $z_m$ in a compact interval, so
  all the maximizers lie in a common compact set.  On that set,
  $\cH_{z_m}^{\al}\to\cH_z^{\al}$ uniformly by uniform continuity of
  the Brownian paths.  Every subsequential limit of the maximizers
  therefore maximizes $\cH_z^{\al}$.  Uniqueness at $z$ identifies the
  limit as $\Gamma_z^{\al}$.
\end{proof}

In addition to the quenched law
$Q_z^{\al,\infty}=\delta_{\Gamma_z^{\al}}$ from
\eqref{not_zero}, define
\begin{align}
  \begin{aligned}
    \rmP_{\wt x,x}^{\al,\infty}(A)
      =\EE\left[
        \bigl(   Q_{\wt x}^{\al,\infty}   \otimes Q_x^{\al,\infty}   \bigr)(A)   \right]
    \AND
    \rmP_0^{\al,\infty}(C)
      =\EE\bigl[Q_0^{\al,\infty}(C)\bigr].
  \end{aligned}
\end{align}

\begin{proof}[Proof of Corollary~\ref{blpp_cor}]
  Fix $\wt x\in\R$, and let
  $I\Subset\R\setminus\{\wt x\}$ be an interval containing $x$ in its
  interior.  Define
  \begin{align*}
    G_\be(u)
      &=\EE\bigl[
        F_{\wt x}^{\al,\be}F_u^{\al,\be}
      \bigr]
    \AND
    G(u)
      =\EE\bigl[
        L_{\wt x}^{\al}L_u^{\al}
      \bigr].
  \end{align*}
  By Theorem~\ref{main_thm}, $G_\be'(u)=g_\be(u)$ on $I$, where
  \begin{align*}
    g_\be(u)
    &={}
    \frac{1-\al^2}{2}\sgn(\wt x-u)
      \rmP_{\wt x,u}^{\al,\be}(\tau\leq n-1)
  +
      \al^2\EE\bigl[
        Q_0^{\al,\be}\{\ga_n>-u\}
      \bigr]
      -\frac{\al^2}{2}
        \bigl(1-\sgn(\wt x-u)\bigr).
  \end{align*}
  Lemma~\ref{FE_blpp} gives
  $G_\be(u)\to G(u)$ for every $u\in I$.

  Lemma~\ref{Q_geo} gives almost-sure weak
  convergence of the relevant one- and two-polymer measures.
  Lemma~\ref{cont_sets} and the continuity-set form of
  the Portmanteau theorem \cite[Theorem~2.1]{Billingsley1999} therefore
  imply
  \begin{align*}
    \EE\bigl[Q_0^{\al,\be}\{\ga_n>-u\}\bigr]
      &\longrightarrow
      \rmP_0^{\al,\infty}(\ga_n>-u),
    \\
    \rmP_{\wt x,u}^{\al,\be}(\tau\leq n-1)
      &\longrightarrow
      \rmP_{\wt x,u}^{\al,\infty}(\tau\leq n-1).
  \end{align*}
  Hence $g_\be(u)\to g(u)$ pointwise, where $g$ is the right-hand side
  of the identity in Corollary~\ref{blpp_cor}, with $u$ in place
  of $x$.  The functions $g_\be$ are bounded uniformly in
  $u\in I$ and $\be>0$ by a constant depending only on $\al$.

  For $a,b\in I$, dominated convergence gives
  \[
    G(b)-G(a)
    =\lim_{\be\to\infty}
      \int_a^b g_\be(u)\,du
    =\int_a^b g(u)\,du.
  \]
  Finally, $g$ is continuous on $I$.  The exit probability is
  continuous by \eqref{exit_atom}.  For the coalescence term, let
  $u_m\to u$ and use Lemma~\ref{geo_stab}.  The associated
  meeting indicators converge almost surely unless the limiting pair
  lies in the boundary event from \eqref{coal_bd0}, which
  has probability zero.  Dominated convergence therefore gives
  continuity of the coalescence probability.  It follows that
  $G'(x)=g(x)$, proving the corollary.
\end{proof}

\appendix

\section{Gaussian integration by parts and Malliavin
  derivatives}

This appendix collects the Gaussian and Malliavin-calculus facts used in
Section~\ref{proof_sec}.  We also prove the moment bounds needed to
remove the spatial mollification.  The only growth input is the
almost-sure polynomial growth of finitely many two-sided Brownian motions.

\medskip
\noi
$\bullet$ {\bf The isonormal Gaussian framework.}
Fix $n\geq1$ and let
\begin{align}
  \fH=L^2\bigl(\{0,\dots,n\}\times\R,\fc\otimes dy\bigr),
\end{align}
where $\fc$ is counting measure on $\{0,\dots,n\}$.  Thus
\[
  \langle h,g\rangle_{\fH}
  =\sum_{i=0}^n\int_\R h_i(y)g_i(y)\,dy.
\]
Let $W=\{W(h):h\in\fH\}$ be an isonormal Gaussian process over $\fH$.
Writing $\xi_i$ for the white noise on level $i$, we have
\begin{align}
  W(h)=\sum_{i=0}^n\int_\R h_i(y)\,\xi_i(dy)
  \AND
  W_i(f)=\int_\R f(y)\,\xi_i(dy).
\end{align}
The associated two-sided Brownian motions are
\[
  B_i(t)=
  \begin{cases}
    W_i(\ind_{(0,t]}),&t\geq0,\\
    -W_i(\ind_{(t,0]}),&t<0.
  \end{cases}
\]

For a smooth cylindrical functional
$  F=f\bigl(W(h^{(1)}),\dots,W(h^{(m)})\bigr),$
its Malliavin derivative is
\begin{align}\label{exp_DF}
  DF=\sum_{k=1}^m
  \partial_kf\bigl(W(h^{(1)}),\dots,W(h^{(m)})\bigr)h^{(k)}.
\end{align}
We write $D_{i,y}F$ for the value of $DF$ at $(i,y)$.  The operator $D$
is closable from $L^2(\Omega)$ to $L^2(\Omega;\fH)$, and $\DD^{1,2}$
denotes its domain, equipped with
\begin{align}
  \|F\|_{1,2}^2
  =\EE[F^2]+\EE\bigl[\|DF\|_{\fH}^2\bigr].
\end{align}
We use the standard product and chain rules in $\DD^{1,2}$.  In
particular, if $F\in\DD^{1,2}$ and $\varphi\in C^1(\R)$ has bounded
derivative, then
\begin{align}\label{DphiF}
  D\varphi(F)=\varphi'(F)DF.
\end{align}
We also use the following logarithmic version: if $G>0$ almost surely,
$G\in\DD^{1,2}$, $\log G\in L^2(\Omega)$, and
$DG/G\in L^2(\Omega;\fH)$, then
\[
  \log G\in\DD^{1,2},
  \qquad
  D\log G=\frac{DG}{G}.
\]
Indeed, apply \eqref{DphiF} to $\log(G+\dl)$ and let $\dl\downarrow0$;
the stated assumptions and the closedness of $D$ justify the limit.
See \cite[Chapter 1]{Nua06} for these standard facts.

\begin{lemma}[\textsf{Gaussian integration by parts}]\label{ibp}
Let $W=\{W(h):h\in\fH\}$ be the isonormal Gaussian process.
For every $F\in\DD^{1,2}$ and $h\in\fH$,
\begin{align}\label{Gibp1}
  \EE[F W(h)]
  =
  \EE\bigl[\langle DF,h\rangle_{\fH}\bigr].
\end{align}
Equivalently,
\begin{align*}
  \EE[F W(h)]
  =
  \sum_{i=0}^n
  \int_\R
    \EE[D_{i,y}F]\,h_i(y)\,dy.
\end{align*}
\end{lemma}

\begin{proof}
The identity is standard for smooth cylindrical random variables.
Since the Malliavin derivative is closable and smooth cylindrical
random variables are dense in $\DD^{1,2}$, the identity extends to
all $F\in\DD^{1,2}$ by approximation. See, e.g., \cite[Lemma 1.2.1]{Nua06}. 
\end{proof}

\medskip
\noi
$\bullet$ {\bf Brownian growth and mollified objects.}
Let $\phi\in C_c^\infty(\R)$ be nonnegative and even, with
$\int_\R\phi=1$, and set
\[
  \phi^\eps(r)=\eps^{-1}\phi(r/\eps),
  \qquad
  \Phi^\eps(r)=\int_0^r\phi^\eps(s)\,ds.
\]
Then $\Phi^\eps$ is odd and $|\Phi^\eps|\leq1/2$.  For $\eps>0$, define
\begin{align}
  \xi_i^\eps(x)
  &=W_i\bigl(\phi^\eps(x-\cdot)\bigr),\label{s_xi}\\
  B_i^\eps(t)
  &=\int_0^t\xi_i^\eps(s)\,ds
    =W_i\bigl(\phi^\eps\ast\ind_{[0,t]}\bigr).
    \label{s_B}
\end{align}
Here and below, the oriented indicator is understood when $t<0$.  Since
$\phi^\eps$ is even,
\begin{align}\label{fora}
  (\phi^\eps\ast\ind_{[0,t]})(y)
  =\Phi^\eps(t-y)-\Phi^\eps(-y).
\end{align}
Stochastic Fubini also gives the pathwise representation
\begin{align}\label{Beps_path}
  B_i^\eps(t)
  =\int_\R\phi^\eps(r)
    \bigl(B_i(t-r)-B_i(-r)\bigr)\,dr.
\end{align}
For $\eps=0$, we put $B_i^0=B_i$.

Let us record a useful estimate. 

\begin{lemma}[\textsf{Polynomial growth of the Brownian paths}]
  \label{B_poly}
  Fix $\theta\in(1/2,1)$ and define
  \[
    M_\theta
    =1+\max_{0\leq i\leq n}
      \sup_{t\in\R}\frac{|B_i(t)|}{1+|t|^\theta}.
  \]
  Then $M_\theta<\infty$ almost surely,  and for some constants
  $c,C>0$ that depend only on $\theta, n$, 
  we have 
  
  \noi
  \begin{align} \label{Mt_bdd}
    \PP(M_\theta>u)\leq C e^{-cu^2},
    \qquad u\geq1.
  \end{align}
  In particular, $M_\theta$ has moments of every finite order.  If
  $\supp(\phi)\subset[-R_\phi,R_\phi]$, then, after changing a
  deterministic constant,
  \begin{align}\label{Beps_grow}
    \sup_{0\leq\eps<1}|B_i^\eps(t)|
    \leq C M_\theta(1+|t|^\theta),
    \qquad t\in\R,
  \end{align}
  simultaneously for $0\leq i\leq n$.
\end{lemma}

\begin{proof}
  By the reflection principle and Brownian scaling, for $k\geq0$ and
  $u\geq1$,
  \[
    \PP\bigg(   \sup_{|t|\leq2^{k+1}}|B_i(t)|   >u2^{k\theta}   \bigg)
    \leq
    C\exp\left\{-cu^2 2^{k(2\theta-1)}\right\},
  \]
for some absolute constants $c, C > 0$.
  Since $2\theta-1>0$, we have
  for $u\geq1$,
  \[
 \PP(M_\theta > u ) \leq  C n \sum_{k=0}^\infty
    \exp\left\{-cu^2 2^{k(2\theta-1)}\right\},
  \]
  which gives us the bound \eqref{Mt_bdd}
  and 
  also proves the
  almost-sure finiteness of $M_\theta$.  Finally,
  \eqref{Beps_path} and
  $\supp(\phi^\eps)\subset[-R_\phi\eps,R_\phi\eps]$ imply
  \[
    |B_i^\eps(t)|
    \leq
    2\sup_{|s|\leq |t|+R_\phi}|B_i(s)|,
  \]
  which yields \eqref{Beps_grow}.
\end{proof}

For $\ga=(\ga_1,\dots,\ga_n)\in\DN$, with $\ga_0=0$, define
\begin{align}\label{H_moll}
  H_{x}^{\al,\eps}(\ga)
  =\al B_n^\eps(x+\ga_n)
  +\sum_{i=0}^{n-1}
    \bigl(B_i^\eps(x+\ga_{i+1})-B_i^\eps(x+\ga_i)\bigr).
\end{align}
For a compact set $K\subset\R$, Lemma~\ref{B_poly}
gives
\begin{align}\label{H_grow}
  \sup_{\substack{x\in K\\0\leq\eps<1}}
  |H_{x}^{\al,\eps}(\ga)|
  \leq C_K M_\theta(1+\ga_n^\theta).
\end{align}

For $R\in[1,\infty]$ and $\lambda\in[0,\be)$, set
\begin{align}\label{def_ZR}
  Z_{x,R}^{\al,\be,\eps}(\lambda)
  =\be^{-n}\int_{\R_+^n}
    e^{\lambda\ga_n+\be H_{x}^{\al,\eps}(\ga)}
    \ind_{\{\ga_n<R\}}\,\mu_\be(d\ga).
\end{align}
Thus
$Z_{x,\infty}^{\al,\be,\eps}(0)=Z_{x}^{\al,\be,\eps}$.
For $R<\infty$, let
\begin{align}\label{QxR_abe}
  Q_{x,R}^{\al,\be,\eps}(d\ga)
  =
  \frac{   \be^{-n}e^{\be H_{x}^{\al,\eps}(\ga)}   \ind_{\{\ga_n<R\}}  }{
    Z_{x,R}^{\al,\be,\eps}(0)  }   \,\mu_\be(d\ga),
\end{align}
and use the same notation with $R=\infty$ for the full polymer
measure.

More generally, for $\lambda\in[0,\be)$ and $R<\infty$, define the
$\lambda$-tilted truncated polymer probability measure by
\begin{align}\label{def_QR_tilted}
  Q_{x,R}^{\al,\be,\eps,\lambda}(d\ga)
  &=
  \frac{
    \be^{-n}
    e^{\lambda\ga_n+\be H_{x}^{\al,\eps}(\ga)}
    \ind_{\{\ga_n<R\}}
  }{
    Z_{x,R}^{\al,\be,\eps}(\lambda)
  }
  \,\mu_\be(d\ga).
\end{align}
Equivalently,
\begin{align}\label{eq_QR_tilted}
  Q_{x,R}^{\al,\be,\eps,\lambda}(d\ga)
  &=
  \frac{e^{\lambda\ga_n}}
  {\int_{\DN} e^{\lambda\eta_n}
    Q_{x,R}^{\al,\be,\eps}(d\eta)}
  Q_{x,R}^{\al,\be,\eps}(d\ga).
\end{align}
In particular,
$
  Q_{x,R}^{\al,\be,\eps,0}
  =
  Q_{x,R}^{\al,\be,\eps}.
$

\begin{lemma}\label{lem_MD1}
  The following statements hold.

  \smallskip
  \noindent
  \textup{(i) (\textsf{Explicit Malliavin derivatives})}  For $\eps>0$,
  \begin{align}
    D_{j,y}B_i^\eps(t)
    &=\dl_{ij}\bigl(\Phi^\eps(t-y)-\Phi^\eps(-y)\bigr),
      \\
    D_{j,y}\xi_i^\eps(x)
    &=\dl_{ij}\phi^\eps(x-y).
  \end{align}
  Consequently,
  \begin{align}\label{MD_H}
    D_{j,y}H_{x}^{\al,\eps}(\ga)
    =
    \begin{cases}
      \Phi^\eps(x+\ga_{j+1}-y)
      -\Phi^\eps(x+\ga_j-y),&0\leq j\leq n-1,\\[3pt]
      \al\bigl(\Phi^\eps(x+\ga_n-y)-\Phi^\eps(-y)\bigr),&j=n.
    \end{cases}
  \end{align}

  \smallskip
  \noindent
  \textup{(ii) (\textsf{Variance bound})}  Put
  $v_\eps(x,\ga)=\Var(H_{x}^{\al,\eps}(\ga)) =  \| D H_{x}^{\al,\eps}(\ga) \|_\fH^2  $.  Then
  \begin{align}\label{bdd_vga}
    -\al^2|x|+(1+\al^2)\ga_n
    -\eps(n+\al^2)\int_\R |r|(\phi\ast\phi)(r)\,dr
    \leq v_\eps(x,\ga)
    \leq \al^2|x|+(1+\al^2)\ga_n.
  \end{align}
  In particular, for every compact $K\subset\R$,
  \[
    \sup_{x\in K,\,0\leq\eps<1}v_\eps(x,\ga)
    \leq C_K(1+\ga_n).
  \]

  \smallskip
  \noindent
  \textup{(iii) (\textsf{Uniform free-energy moments and removal of the path
  truncation})}  For every compact $K\subset\R$, $p<\infty$, and
  $\lambda\in[0,\be)$,
  \begin{align}\label{LP_logZRL}
    \sup_{\substack{x\in K,\ 1\leq R\leq\infty\\0\leq\eps<1}}
    \bigl\|\log Z_{x,R}^{\al,\be,\eps}(\lambda)\bigr\|_{L^p(\Omega)}
    <\infty.
  \end{align}
  For fixed $x$ and $\eps\in[0,1)$,
  \begin{align}\label{LPcvg_ZR}
    \log Z_{x,R}^{\al,\be,\eps}(\lambda)
    \xrightarrow[R\to\infty]{L^p(\Omega)}
    \log Z_{x,\infty}^{\al,\be,\eps}(\lambda).
  \end{align}

  \smallskip
  \noindent
  \textup{(iv) (\textsf{Removal of the mollification})}  For every fixed $x$ and
  every $p<\infty$,
  \begin{align}\label{LPcvg_eps}
    \log Z_{x}^{\al,\be,\eps}
    \xrightarrow[\eps\downarrow0]{L^p(\Omega)}
    \log Z_{x}^{\al,\be}.
  \end{align}
\end{lemma}

\begin{proof}
  Part~(i) follows directly from \eqref{s_B}, \eqref{fora}, and
  \eqref{s_xi}.  Substituting these identities into
  \eqref{H_moll} gives \eqref{MD_H}.

  For part~(ii), let $V_\eps(t)=\Var(B_i^\eps(t))$ for $t\geq0$.
  Since
  \begin{align}
    V_\eps(t)
    &=\int_{-t}^t(t-|r|)(\phi^\eps\ast\phi^\eps)(r)\,dr,
      \\
    t-\eps\int_\R|r|(\phi\ast\phi)(r)\,dr
    &\leq V_\eps(t)\leq t,
  \end{align}
  independence across the levels yields
  \[
    v_\eps(x,\ga)
    =\al^2V_\eps(|x+\ga_n|)
      +\sum_{i=0}^{n-1}V_\eps(\ga_{i+1}-\ga_i).
  \]
  The bounds in \eqref{bdd_vga} follow from
  $\ga_n-|x|\leq|x+\ga_n|\leq\ga_n+|x|$.

  We now prove part~(iii).  Fix a compact $K$, and let
  $U_0\Subset\DN$ be a deterministic open set such that $\ga_n<1$ on
  $U_0$ and $\mu_\be(U_0)>0$.  For every $R\geq1$, Jensen's inequality
  for the normalized restriction of $\mu_\be$ to $U_0$ gives
  \begin{align}
    \log Z_{x,R}^{\al,\be,\eps}(\lambda)
    \geq c_0
    +\frac{1}{\mu_\be(U_0)}
      \int_{U_0}
      \bigl(\lambda\ga_n+\be H_{x}^{\al,\eps}(\ga)\bigr)
      \,\mu_\be(d\ga),
  \end{align}
  where $c_0$ is deterministic.  The Brownian part of the integral
  on the right is a centered Gaussian random variable, whereas the
  deterministic contribution involving $\lambda\ga_n$ is uniformly
  bounded on $U_0$.  By \eqref{bdd_vga}, the Gaussian variance is
  bounded uniformly over $x\in K$ and $0\leq\eps<1$.  Gaussian moment
  equivalence therefore gives a uniform $L^p$ bound for the negative
  part of the logarithm.

   For the positive part, use \eqref{H_grow} and the
  marginal density of $\ga_n$ under $\mu_\be$ to obtain
  \begin{align*}
    Z_{x,R}^{\al,\be,\eps}(\lambda)
    &\leq C\exp\{CM_\theta\}
      \int_0^\infty t^{n-1}
      \exp\bigl\{-(\be-\lambda)t
        +C\be M_\theta t^\theta\bigr\}\,dt.
  \end{align*}
  Since $\lambda<\be$, Young's inequality with conjugate exponents
  $1/\theta$ and $1/(1-\theta)$ gives
  \begin{align*}
    C\be M_\theta t^\theta
    &\leq \frac{\be-\lambda}{2}t
      +C_\theta
      (\be-\lambda)^{-\theta/(1-\theta)}
      (\be M_\theta)^{1/(1-\theta)}\\
    &=\frac{\be-\lambda}{2}t
      +C_\theta\be
      \left(\frac{\be}{\be-\lambda}\right)^{
        \theta/(1-\theta)}
      M_\theta^{1/(1-\theta)}.
  \end{align*}
  Since $\be$ and $\lambda$ are fixed throughout this lemma, the
  deterministic factor
  $
    \be
    (\frac{\be}{\be-\lambda} )^{
      \theta/(1-\theta)}
  $
  can be absorbed into the constant.  Hence
  \begin{align*}
    C\be M_\theta t^\theta
    \leq \frac{\be-\lambda}{2}t
      +C_{\be,\lambda,\theta}
      M_\theta^{1/(1-\theta)}.
  \end{align*}
  Therefore
  \begin{align*}
    Z_{x,R}^{\al,\be,\eps}(\lambda)
    &\leq
      C\exp\left\{
        CM_\theta
        +C_{\be,\lambda,\theta}
          M_\theta^{1/(1-\theta)}
      \right\}
      \int_0^\infty t^{n-1}
      \exp\left\{-\frac{\be-\lambda}{2}t\right\}\,dt\\
    &=
      C_{\be,\lambda,n}
      \exp\left\{
        CM_\theta
        +C_{\be,\lambda,\theta}
          M_\theta^{1/(1-\theta)}
      \right\}.
  \end{align*}
  Consequently,
  \begin{align}
    \bigl(
      \log Z_{x,R}^{\al,\be,\eps}(\lambda)
    \bigr)_+
    \leq
    C_{\be,\lambda,n,\theta}
    \left(
      1+M_\theta^{1/(1-\theta)}
    \right),
  \end{align}
  uniformly over $x\in K$, $R\in [1, \infty]$, and $0\leq\eps<1$.
  Lemma~\ref{B_poly} now proves
  \eqref{LP_logZRL}.

  For fixed $x$ and $\eps$, monotone convergence gives
  $Z_{x,R}^{\al,\be,\eps}(\lambda)\uparrow
  Z_{x,\infty}^{\al,\be,\eps}(\lambda)$ almost surely.  Applying
  \eqref{LP_logZRL} with exponent $2p$ makes the $p$th powers of the
  logarithmic differences uniformly integrable.  Vitali's theorem
  yields \eqref{LPcvg_ZR}.

  Finally, \eqref{Beps_path} and continuity of the Brownian paths
  imply $B_i^\eps\to B_i$ locally uniformly almost surely.  Thus
  $H_{x}^{\al,\eps}(\ga)\to H_{x}^{\al,0}(\ga)$ for every $\ga$.
  The pathwise bound \eqref{H_grow} supplies an
  integrable majorant against $\mu_\be$, and consequently
  \[
    Z_{x}^{\al,\be,\eps}
    \longrightarrow Z_{x}^{\al,\be}
    \qquad\text{almost surely}.
  \]
  The partition functions are strictly positive, so their logarithms
  converge almost surely.  Another application of \eqref{LP_logZRL}
  with exponent $2p$ and Vitali's theorem proves \eqref{LPcvg_eps}.
\end{proof}

We next justify differentiation of the partition function in the
Malliavin sense.  We use the following standard closed-operator fact.

\begin{lemma}\label{lem_D_int}
  Let $(E,\nu)$ be a finite measure space.  Suppose that
  $F(\eta)\in\DD^{1,2}$ for $\nu$-almost every $\eta$ and
  \[
    \EE\int_E
    \bigl(|F(\eta)|^2+\|DF(\eta)\|_{\fH}^2\bigr)\,\nu(d\eta)<\infty.
  \]
  Then $\int_EF(\eta)\nu(d\eta)\in\DD^{1,2}$ and
  \[
    D\int_EF(\eta)\nu(d\eta)
    =\int_EDF(\eta)\nu(d\eta).
  \]
\end{lemma}

\begin{proof}
  This is the Gaussian analogue of
  \cite[Lemma~2.6]{BZ24} 
  and the proof follows exactly the same lines. 
\end{proof}

\begin{proposition}\label{prop_MD}
 The following statements hold.

  \smallskip
  \noindent
  \textup{(i) (\textsf{Polymer endpoint moments})}  
  Recall the definition of $  Q_{x,R}^{\al,\be,\eps}$
  from  \eqref{QxR_abe}. 
  For every compact
  $K\subset\R$ and every $p\in[1,\infty)$,
  we have 
  
  \noi
  \begin{align}\label{expec_Q}
    \sup_{\substack{x\in K,\ 0<\eps<1\\1\leq R\leq\infty}}
    \bigg\|
      \int_{\R_+^n}\ga_n
      Q_{x,R}^{\al,\be,\eps}(d\ga)
    \bigg\|_{L^p(\Omega)}<\infty.
  \end{align}
  In particular, with
  \begin{align}
    \cU_{x}^{\eps}
    =\int_{\R_+^n}
     \big[  DH_{x}^{\al,\eps}(\ga) \big]
      Q_{x}^{\al,\be,\eps}(d\ga),
  \end{align}
  one has
  \[
    \sup_{x\in K,\,0<\eps<1}
    \|\cU_x^\eps\|_{L^2(\Omega;\fH)}<\infty.
  \]

  \smallskip
  \noindent
  \textup{(ii) (\textsf{Truncated Malliavin derivative})}  For $R<\infty$,
  $\lambda\in[0,\be)$, and $\eps>0$,
  $\log Z_{x,R}^{\al,\be,\eps}(\lambda)\in\DD^{1,2}$ and
  \begin{align}\label{eq_DlogZR}
    D\log Z_{x,R}^{\al,\be,\eps}(\lambda)
    =\frac{\be}{Z_{x,R}^{\al,\be,\eps}(\lambda)}
    \be^{-n}\int_{\ga_n<R}
      e^{\lambda\ga_n+\be H_{x}^{\al,\eps}(\ga)}
     \big[ DH_{x}^{\al,\eps}(\ga) \big]\,\mu_\be(d\ga).
  \end{align}

  \smallskip
  \noindent
  \textup{(iii) (\textsf{Full Malliavin derivative})}  For every fixed
  $\eps>0$, we have
  \begin{align}\label{MD_logZ}
    \log Z_{x}^{\al,\be,\eps}\in\DD^{1,2}
    \quad{\rm and}
    \quad
    D\log Z_{x}^{\al,\be,\eps}
    =\be\cU_x^\eps.
  \end{align}
  Moreover,
  $
    D\log Z_{x,R}^{\al,\be,\eps}(0)
    \longrightarrow \be\cU_x^\eps
    \quad\text{in }L^2(\Omega;\fH)
$
  as $R\to\infty$.
\end{proposition}

\begin{proof}
  Fix $\lambda\in(0,\be)$.  Jensen's inequality gives
  \begin{align}
    \int_{\DN}\ga_nQ_{x,R}^{\al,\be,\eps}(d\ga)
    \leq\frac{1}{\lambda}
      \bigl(
        \log Z_{x,R}^{\al,\be,\eps}(\lambda)
        -\log Z_{x,R}^{\al,\be,\eps}(0)
      \bigr).
  \end{align}
  The compact-uniform estimate \eqref{LP_logZRL} proves
  \eqref{expec_Q}.  Since
  $\|DH_{x}^{\al,\eps}(\ga)\|_{\fH}^2=v_\eps(x,\ga)$, Jensen's
  inequality and \eqref{bdd_vga} give
  
  \noi
  \begin{align}
    \EE\bigl[\|\cU_x^\eps\|_{\fH}^2\bigr]
    &\leq
    \EE\int_{\DN} v_\eps(x,\ga)
      Q_{x}^{\al,\be,\eps}(d\ga)
      \leq C_K\left(
      1+\EE\int_{\DN}\ga_nQ_{x}^{\al,\be,\eps}(d\ga)
    \right), 
  \end{align}
  which proves part~(i).

  For finite $R$, Lemma~\ref{lem_D_int}, the chain rule, and
  \eqref{bdd_vga} show that
  \begin{align}\label{MD_ZRL}
    DZ_{x,R}^{\al,\be,\eps}(\lambda)
    =
    \be^{1-n}\int_{\ga_n<R}
      e^{\lambda\ga_n+\be H_{x}^{\al,\eps}(\ga)}
      DH_{x}^{\al,\eps}(\ga)\,\mu_\be(d\ga).
  \end{align}
  Indeed, the required square-integrability is immediate on
  $\{\ga_n<R\}$ from
$
    \EE\bigl[
      e^{2\be H_{x}^{\al,\eps}(\ga)}
    \bigr]
    =
    e^{2\be^2v_\eps(x,\ga)}.
$
  By \eqref{def_QR_tilted}, \eqref{MD_ZRL} can be rewritten as
  \begin{align}
    \frac{
      DZ_{x,R}^{\al,\be,\eps}(\lambda)
    }{
      Z_{x,R}^{\al,\be,\eps}(\lambda)
    }
    &=
    \be\int_{\DN}
    \big[  DH_{x}^{\al,\eps}(\ga) \big]\, 
      Q_{x,R}^{\al,\be,\eps,\lambda}(d\ga),
  \end{align}
where $Q_{x,R}^{\al,\be,\eps,\lambda}$, given in  \eqref{eq_QR_tilted},  
is a probability.
Then, we deduce from 
  Jensen's inequality that 
  
  \noi
  \begin{align*}
    \bigg\|
      \frac{
        DZ_{x,R}^{\al,\be,\eps}(\lambda)
      }{
        Z_{x,R}^{\al,\be,\eps}(\lambda)
      }
    \bigg\|_{\fH}^2
    &\leq
    \be^2
    \int_{\DN}
      \|DH_{x}^{\al,\eps}(\ga)\|_{\fH}^2
      Q_{x,R}^{\al,\be,\eps,\lambda}(d\ga)
      =
    \be^2
    \int_{\DN}
      v_\eps(x,\ga)\,
      Q_{x,R}^{\al,\be,\eps,\lambda}(d\ga).
  \end{align*}
  By \eqref{bdd_vga} and the fact that
  $Q_{x,R}^{\al,\be,\eps,\lambda}$ is supported on
  $\{\ga_n<R\}$,
  \begin{align*}
    \bigg\| \frac{  DZ_{x,R}^{\al,\be,\eps}(\lambda)   }{   Z_{x,R}^{\al,\be,\eps}(\lambda) }   \bigg\|_{\fH}^2
    &\leq
    \be^2   \bigl(      \al^2|x|+(1+\al^2)R
    \bigr).
  \end{align*}
  Hence
  $
    \frac{
      DZ_{x,R}^{\al,\be,\eps}(\lambda)
    }{
      Z_{x,R}^{\al,\be,\eps}(\lambda)
    }
    \in L^2(\Omega;\fH).
$
  Lemma~\ref{lem_MD1} also gives
  $\log Z_{x,R}^{\al,\be,\eps}(\lambda)\in L^2(\Omega)$.
  The logarithmic chain rule therefore applies and yields
  \eqref{eq_DlogZR}.

  It remains to let $R\to\infty$ at $\lambda=0$.  Since
  $Q_{x,R}^{\al,\be,\eps,0}
  =Q_{x,R}^{\al,\be,\eps}$, put
  \[
    U_R
    =
    \int_{\DN} 
          \big[  DH_{x}^{\al,\eps}(\ga) \big]
      Q_{x,R}^{\al,\be,\eps}(d\ga),
    \qquad
    U_\infty=\cU_x^\eps.
  \]
  Then
  \begin{align}\label{UR_dec}
    U_R-U_\infty
    &=\left(1-
      \frac{Z_{x,R}^{\al,\be,\eps}(0)}
           {Z_{x,\infty}^{\al,\be,\eps}(0)}
      \right)U_R
      -\int_{\ga_n\geq R}
      \big[  DH_{x}^{\al,\eps}(\ga) \big]
        Q_{x}^{\al,\be,\eps}(d\ga).
  \end{align}
  The scalar factor in the first term converges almost surely to zero
  and is bounded by one.  By Jensen, \eqref{bdd_vga}, and
  \eqref{expec_Q}, the family $\{\|U_R\|_{\fH}^2:R\geq1\}$ is uniformly
  integrable.  Hence the first term in \eqref{UR_dec}
  converges to zero in $L^2(\Omega;\fH)$.  For the second term,
  \begin{align*}
    \left\|
      \int_{\ga_n\geq R}
      \big[  DH_{x}^{\al,\eps}(\ga) \big]
      Q_{x}^{\al,\be,\eps}(d\ga)
    \right\|_{\fH}^2
    \leq
    \int_{\ga_n\geq R}v_\eps(x,\ga)
      Q_{x}^{\al,\be,\eps}(d\ga),
  \end{align*}
  whose expectation tends to zero by dominated convergence and
  \eqref{expec_Q}.  Thus $U_R\to U_\infty$ in
  $L^2(\Omega;\fH)$.

  Finally, \eqref{LPcvg_ZR} gives convergence of the logarithms in
  $L^2(\Omega)$.  The closedness of $D$, together with the convergence
  of the derivatives just proved, yields \eqref{MD_logZ}.
\end{proof}

\section{Cancellation identities}\label{cancel_app}

This appendix proves the two cancellations used in
Section~\ref{proof_sec}.  They
correspond to the two cases in the Malliavin derivative formula
\eqref{MD_H}: the indices $j=0,\dots,n-1$ correspond to the bulk
Brownian levels, whereas $j=n$ corresponds to the terminal Brownian
level.

We keep $\be>0$ arbitrary, because the lemmas are invoked at general
inverse temperature.  Fix $x\in\R$, $n\geq1$, $\al\in\R$, and
$\eps>0$, and write
\[
  H(\ga)=H_{x}^{\al,\eps}(\ga),
  \qquad
  \cZ=\int_{\DN} e^{\be H(\ga)}\,\mu_\be(d\ga),
  \qquad
  Q(d\ga)=\frac{e^{\be H(\ga)}}{\cZ}\,\mu_\be(d\ga).
\]
As in Section~\ref{proof_sec}, the deterministic factor $\be^{-n}$ is suppressed; it
has no effect on either cancellation.

For later use, note that the product rule gives, for every
$j\in\{0,\dots,n\}$,\footnote{To be more precise, 
using the estimates in Appendix A, 
we can first prove 
$e^{\be H(\ga)}/\cZ \in L^2(\Omega)$ and 
the right-hand side of \eqref{cancel_DQ} belongs to 
$L^2(\Omega; \fH)$, then we can conclude
$e^{\be H(\ga)}/\cZ\in\mathbb{D}^{1,2}$.  }

\noi
\begin{align}\label{cancel_DQ}
  D_{j,y}\bigg(\frac{e^{\be H(\ga)}}{\cZ}\bigg)
  =
  \be\left(  D_{j,y}H(\ga)    -   \int_{\DN} D_{j,y}H(\wt\ga)\,Q(d\wt\ga)   \right)
  \frac{e^{\be H(\ga)}}{\cZ}.
\end{align}
The form of $D_{j,y}H$ is different according as $j<n$ or $j=n$,
and we treat these two cases separately.

\begin{lemma}\label{cancel_b}
  For each $i\in\{0,\dots,n-1\}$, almost surely,
  \begin{align*}
    \int_\R dy\int_{\DN}\mu_\be(d\ga)\,
    D_{i,y}\bigg(\frac{e^{\be H(\ga)}}{\cZ}\bigg)
    \bigl(
      \phi^\eps(x+\ga_{i+1}-y)
      -\phi^\eps(x+\ga_i-y)
    \bigr)
    =0.
  \end{align*}
\end{lemma}

\begin{proof}
  Fix $i\in\{0,\dots,n-1\}$ and define
  \[
    F_\ga(y)
    =
    \Phi^\eps(x+\ga_{i+1}-y)
    -
    \Phi^\eps(x+\ga_i-y).
  \]
  By the first case of \eqref{MD_H},
 $    D_{i,y}H(\ga)=F_\ga(y).$
  Moreover,
  $
    F_\ga'(y)
    =
    -\phi^\eps(x+\ga_{i+1}-y)
    +
    \phi^\eps(x+\ga_i-y),
  $
  and hence
  $
    \phi^\eps(x+\ga_{i+1}-y)
    -
    \phi^\eps(x+\ga_i-y)
    =
    -F_\ga'(y).
  $

  Substituting the first case of \eqref{MD_H} into
  \eqref{cancel_DQ} gives
  \[
    D_{i,y}\bigg(\frac{e^{\be H(\ga)}}{\cZ}\bigg)
    =
    \be\left(
      F_\ga(y)
      -
      \int_{\DN} F_{\wt\ga}(y)\,Q(d\wt\ga)
    \right)
    \frac{e^{\be H(\ga)}}{\cZ}.
  \]
  Therefore, after integrating first with respect to $\mu_\be(d\ga)$,
  the quantity in the statement equals
  \begin{align*}
    &-\be\int_{\DN} Q(d\ga)
      \int_\R F_\ga(y)F_\ga'(y)\,dy
      +
    \be\int_{\DN^2} Q(d\ga)Q(d\wt\ga)
      \int_\R F_{\wt\ga}(y)F_\ga'(y)\,dy.
  \end{align*}

  For the first term,
  \[
    \int_\R F_\ga(y)F_\ga'(y)\,dy
    =
    \frac12
    \bigl[F_\ga(y)^2\bigr]_{-\infty}^{\infty}
    =0,
  \]
  because
  $
    F_\ga(y)\longrightarrow0
$ as $y\to\pm\infty.$
  For the second term, exchange $\ga$ and $\wt\ga$ and average the two
  expressions.  This gives
  \begin{align*}
    &\int_{\DN^2} Q(d\ga)Q(d\wt\ga)
      \int_\R F_{\wt\ga}(y)F_\ga'(y)\,dy
      =
    \frac12
    \int_{\DN^2} Q(d\ga)Q(d\wt\ga)
      \int_\R
      \frac{d}{dy}
      \bigl(
        F_\ga(y)F_{\wt\ga}(y)
      \bigr)\,dy
    =0,
  \end{align*}
  again because both factors tend to zero at $y=\pm\infty$.
  This proves the bulk cancellation.
\end{proof}

\begin{lemma}\label{cancel_bd}
  Almost surely,
  \[
    \int_\R dy\int_{\DN}\mu_\be(d\ga)\,
    D_{n,y}\bigg(\frac{e^{\be H(\ga)}}{\cZ}\bigg)
    \al\phi^\eps(x+\ga_n-y)
    =0.
  \]
\end{lemma}

\begin{proof}
  We now use the second case of \eqref{MD_H}, corresponding to the
  terminal Brownian level $j=n$.  Put
  \[
    G_\ga(y)
    =
    \al\Phi^\eps(x+\ga_n-y)
\AND
    K(y)
    =
    \al\Phi^\eps(-y).
  \]
  Then the $j=n$ case of \eqref{MD_H} reads
  $ D_{n,y}H(\ga)   =   G_\ga(y)-K(y)$.
  Also,
  $  G_\ga'(y)  =   -\al\phi^\eps(x+\ga_n-y),
$
  so that
  $
    \al\phi^\eps(x+\ga_n-y)
    =
    -G_\ga'(y).
  $

  Applying \eqref{cancel_DQ} with $j=n$ gives
  \[
    D_{n,y}\bigg(\frac{e^{\be H(\ga)}}{\cZ}\bigg)
    =
    \be\left[
      G_\ga(y)-K(y)
      -
      \int_{\DN}
      \bigl(G_{\wt\ga}(y)-K(y)\bigr)
      Q(d\wt\ga)
    \right]
    \frac{e^{\be H(\ga)}}{\cZ}.
  \]
  Since $K(y)$ does not depend on the polymer path and
  $Q$ is a probability measure,
  \[
    \int_{\DN}
      \bigl(G_{\wt\ga}(y)-K(y)\bigr)
      Q(d\wt\ga)
    =
    \int_{\DN} G_{\wt\ga}(y)Q(d\wt\ga)-K(y).
  \]
  Thus the two $K(y)$ terms cancel already at this stage, and
  \begin{align}
    D_{n,y}\bigg(\frac{e^{\be H(\ga)}}{\cZ}\bigg)
    =
    \be\left(
      G_\ga(y)
      -
      \int_{\DN} G_{\wt\ga}(y)Q(d\wt\ga)
    \right)
    \frac{e^{\be H(\ga)}}{\cZ}.
  \end{align}

  Consequently, after dividing the desired integral by $\be$, it is
  enough to show that
  \begin{align*}
    &-\int_{\DN} Q(d\ga)
      \int_\R G_\ga(y)G_\ga'(y)\,dy
 +
    \int_{\DN^2} Q(d\ga)Q(d\wt\ga)
      \int_\R G_{\wt\ga}(y)G_\ga'(y)\,dy
    =0.
  \end{align*}

  For the first term,
  \[
    \int_\R G_\ga(y)G_\ga'(y)\,dy
    =
    \frac12
    \bigl[G_\ga(y)^2\bigr]_{-\infty}^{\infty}
    =0.
  \]
  Indeed,
  \[
    G_\ga(y)\longrightarrow
    \begin{cases}
      \al/2,&y\to-\infty,\\
      -\al/2,&y\to+\infty,
    \end{cases}
  \]
  and hence $G_\ga(y)^2$ has the same limit $\al^2/4$ at both ends.

  Finally, exchanging $\ga$ and $\wt\ga$ and symmetrizing gives
  \begin{align*}
    &\int_{\DN^2} Q(d\ga)Q(d\wt\ga)
      \int_\R G_{\wt\ga}(y)G_\ga'(y)\,dy
      =
    \frac12
    \int_{\DN^2} Q(d\ga)Q(d\wt\ga)
      \int_\R
      \frac{d}{dy}
      \bigl(
        G_\ga(y)G_{\wt\ga}(y)
      \bigr)\,dy
    =0.
  \end{align*}
  The last equality holds because
  $
    G_\ga(y)G_{\wt\ga}(y)
    \longrightarrow \frac{\al^2}{4}$
as $y\to\pm\infty.$
Hence, the proof is completed. 
\end{proof}

\section{Justification of the mollification limit}
  \label{moll_app}

Throughout this appendix, $n$, $\al$, and $\be$ are fixed.  We use the
normalization adopted in Section~\ref{proof_sec} and set
\begin{align*}
  H_x^\eps=H_{x}^{\al,\eps},
  \quad
  \cZ_x^\eps
  =\int_{\DN}e^{\be H_x^\eps(\ga)}\,\mu_\be(d\ga)
    =\be^n Z_{x}^{\al,\be,\eps},
    \AND
  Q_x^\eps(d\ga)
  =\frac{e^{\be H_x^\eps(\ga)}}{\cZ_x^\eps}\,\mu_\be(d\ga).
\end{align*}
We write $H_x=H_x^0$, $\cZ_x=\cZ_x^0$, and $Q_x=Q_x^0$.  Thus the
partition function denoted by $\cZ_x^\eps$ in this appendix is the one
used in Section~\ref{proof_sec}; it differs from the normalized
partition function $Z_x^{\al,\be,\eps}$ in
\eqref{Z_gen} by the deterministic factor $\be^n$, while
the Gibbs measure is identical.

This deterministic normalization does not change any spatial
covariance derivative.  Indeed,
\[
  \frac1\be\log \cZ_x^\eps
  =\frac1\be\log Z_{x}^{\al,\be,\eps}
    +\frac n\be\log\be.
\]
Moreover, $x\mapsto\EE[\be^{-1}\log \cZ_x^\eps]$ is constant.  To see
this, write
\[
  B_n^\eps(x+\ga_n)
  =B_n^\eps(x)+
    \bigl(B_n^\eps(x+\ga_n)-B_n^\eps(x)\bigr).
\]
The relative increment fields
$\{B_i^\eps(x+t)-B_i^\eps(x):t\geq0\}$ have a law independent of $x$,
while the term $B_n^\eps(x)$ can be factored out from the partition function
with  $\EE[B_n^\eps(x)]=0$.  The same observation applies at
$\eps=0$.  We will use this normalization explicitly in the proof of
Proposition~\ref{app_der}.

\medskip

Fix $\theta\in(1/2,1)$ and let $\Omega_0$ be the full-probability event
on which all Brownian paths are continuous and the random variable
$M_\theta$ in Lemma~\ref{B_poly} is finite.
All pathwise statements below are made on $\Omega_0$.  For every
compact $K\subset\R$,
\begin{align}\label{D_grow}
  \sup_{\substack{u\in K\\0\leq\eps<1}}
  |H_u^\eps(\ga)|
  \leq C_{K,\omega}(1+\ga_n^\theta).
\end{align}
Furthermore, by the pathwise representation
\eqref{Beps_path},
\begin{align}\label{D_Hloc}
  \sup_{\substack{u\in K\\\ga_n\leq R}}
  |H_u^\eps(\ga)-H_u(\ga)|\longrightarrow0
\end{align}
for every finite $R$.

\begin{proposition}[\textsf{Uniform truncation principle}]\label{int_uc}
  Let $K\subset\R$ be compact and $q>0$.  Then
  \begin{align}\label{wt_L1}
    \sup_{u\in K}
    \int_{\DN}
      \left|e^{q\be H_u^\eps(\ga)}
        -e^{q\be H_u(\ga)}\right|\,\mu_\be(d\ga)
    \longrightarrow0.
  \end{align}
  More generally, suppose that
  $h_\eps:K\times\DN\to\R$ are uniformly bounded
  over $\eps\in [0, 1]$, and  for every
  $R<\infty$,
  \[
    \sup_{\substack{u\in K\\\ga_n\leq R}}
    |h_\eps(u,\ga)-h_0(u,\ga)|\longrightarrow0.
  \]
  Then
  \begin{align}\label{poly_unif}
    \sup_{u\in K}\left|
      \int_{\DN} e^{q\be H_u^\eps(\ga)}h_\eps(u,\ga)\,\mu_\be(d\ga)
      -\int_{\DN} e^{q\be H_u(\ga)}h_0(u,\ga)\,\mu_\be(d\ga)
    \right|\longrightarrow0.
  \end{align}
\end{proposition}

\begin{proof}
  Fix $R<\infty$.  On $K\times\{\ga_n\leq R\}$, the Hamiltonians are
  uniformly bounded and converge uniformly by
  \eqref{D_Hloc}.  Hence the contribution of
  this set to \eqref{wt_L1} tends to zero.  On its
  complement, from \eqref{D_grow}, we have the upper bound
  $
    2\exp\{C_{q,K,\omega}(1+\ga_n^\theta)\}.
  $
  This majorant is integrable against $\mu_\be$, whose density contains the
  factor $e^{-\be\ga_n}$, and its integral over $\{\ga_n>R\}$ tends to
  zero as $R\to\infty$.  This  proves
  \eqref{wt_L1} by first sending $\eps\downarrow0$ and
  then $R\to\infty$.

  For the second assertion, write
  \begin{align*}
    &e^{q\be H_u^\eps}h_\eps-e^{q\be H_u}h_0
    =
      h_\eps\bigl(e^{q\be H_u^\eps}-e^{q\be H_u}\bigr)
      +e^{q\be H_u}(h_\eps-h_0).
  \end{align*}
  The first term is controlled by
  \eqref{wt_L1}.  The second converges uniformly on each
  truncation, and its tail is controlled by the same integrable
  majorant.  This proves \eqref{poly_unif}.
\end{proof}

\begin{lemma}\label{conv_Z}
  For every compact $K\subset\R$,
  \begin{align}\label{Z_unif}
    \sup_{u\in K}|\cZ_u^\eps-\cZ_u|\longrightarrow0.
  \end{align}
  The map $u\mapsto \cZ_u$ is continuous and strictly positive.
  Consequently, there are $\eps_K>0$ and $c_{K,\omega}>0$ such that
  \begin{align}\label{Z_low}
    \inf_{\substack{u\in K\\0\leq\eps\leq\eps_K}}\cZ_u^\eps
    \geq c_{K,\omega}.
  \end{align}
\end{lemma}

\begin{proof}
  Formula \eqref{Z_unif} follows from
  Proposition~\ref{int_uc} with $q=1$.  Continuity of $\cZ_u$ follows by
  the same truncation argument, using continuity of
  $u\mapsto H_u(\ga)$.  Since $\cZ_u>0$, compactness of $K$ gives a
  positive minimum, and \eqref{Z_low} follows from
  \eqref{Z_unif}.
\end{proof}

For two probability measures $\nu$ and $\wt\nu$ on the same measurable
space, we use the convention
\[
  \|\nu-\wt\nu\|_{\mathrm{TV}}
  =
  \sup_{\|f\|_\infty\leq1}
  \left|
    \int f\,d\nu-\int f\,d\wt\nu
  \right|.
\]
Equivalently,
\[
  \|\nu-\wt\nu\|_{\mathrm{TV}}
  =
  2\sup_A|\nu(A)-\wt\nu(A)|.
\]
If $\nu$ and $\wt\nu$ have densities $p$ and $\wt p$ with respect to
a common dominating measure $\mu$, then
\[
  \|\nu-\wt\nu\|_{\mathrm{TV}}
  =
  \int |p-\wt p|\,d\mu.
\]

\begin{proposition}[\textsf{Local total-variation convergence}]\label{Q_uc}
  For every compact $K\subset\R$,
  \begin{align}\label{TV_unif}
    \sup_{u\in K}
    \|Q_u^\eps-Q_u\|_{\mathrm{TV}}
    \xrightarrow{\eps\to 0}0.
  \end{align}
  Consequently, if $K_1,K_2\subset\R$ are compact and
  $h:K_1\times K_2\times\DN\times\DN\to\R$ is bounded and
  measurable, then
  \begin{align}\label{prod_unif}
    \sup_{\substack{u\in K_1\\v\in K_2}}
    \left|
      \int_{\DN^2} h(u,v,\ga,\wt\ga)
        Q_u^\eps(d\ga)Q_v^\eps(d\wt\ga)
      -\int_{\DN^2} h(u,v,\ga,\wt\ga)
        Q_u(d\ga)Q_v(d\wt\ga)
    \right|\longrightarrow0.
  \end{align}
\end{proposition}

\begin{proof}
  With $w_u^\eps=e^{\be H_u^\eps}$ and $w_u=e^{\be H_u}$,
  \begin{align*}
    \int_{\DN}\left|
      \frac{w_u^\eps}{\cZ_u^\eps}-\frac{w_u}{\cZ_u}
    \right|d\mu_\be
    &\leq
      \frac{1}{\cZ_u^\eps}\int_{\DN}|w_u^\eps-w_u|\,d\mu_\be
      +\frac{|\cZ_u^\eps-\cZ_u|}{\cZ_u^\eps}.
  \end{align*}
Then, the desired total-variation convergence  \eqref{TV_unif}  
follows from Proposition~\ref{int_uc}, Lemma~\ref{conv_Z}, and
  \eqref{Z_low}.
   For fixed $u\in K_1$ and $v\in K_2$, write
  \begin{align*}
    Q_u^\eps\otimes Q_v^\eps-Q_u\otimes Q_v
    &=
    (Q_u^\eps-Q_u)\otimes Q_v^\eps
    +
    Q_u\otimes(Q_v^\eps-Q_v).
  \end{align*}
  Since $Q_v^\eps$ and $Q_u$ are probability measures,
  \[
    \|(Q_u^\eps-Q_u)\otimes Q_v^\eps\|_{\mathrm{TV}}
    =
    \|Q_u^\eps-Q_u\|_{\mathrm{TV}},
  \]
  and
  \[
    \|Q_u\otimes(Q_v^\eps-Q_v)\|_{\mathrm{TV}}
    =
    \|Q_v^\eps-Q_v\|_{\mathrm{TV}}.
  \]
  Therefore, by the triangle inequality,
  \[
    \|Q_u^\eps\otimes Q_v^\eps-Q_u\otimes Q_v\|_{\mathrm{TV}}
    \leq
    \|Q_u^\eps-Q_u\|_{\mathrm{TV}}
    +
    \|Q_v^\eps-Q_v\|_{\mathrm{TV}}.
  \]

  Let
 $
    M=\|h\|_\infty<\infty.
 $
  By the above definition of  total variation distance, 
  we have 
  
  \noi
  \begin{align*}
    &\left|
      \int_{\DN^2} h(u,v,\ga,\wt\ga)
        Q_u^\eps(d\ga)Q_v^\eps(d\wt\ga)
      -
      \int_{\DN^2} h(u,v,\ga,\wt\ga)
        Q_u(d\ga)Q_v(d\wt\ga)
    \right|\\
    &\qquad\leq
 M    \left(
      \|Q_u^\eps-Q_u\|_{\mathrm{TV}}
      +
      \|Q_v^\eps-Q_v\|_{\mathrm{TV}}
    \right).
  \end{align*}
  Hence
  \begin{align*}
    &\sup_{\substack{u\in K_1\\v\in K_2}}
    \left|
      \int_{\DN^2} h(u,v,\ga,\wt\ga)
        Q_u^\eps(d\ga)Q_v^\eps(d\wt\ga)
      -
      \int_{\DN^2} h(u,v,\ga,\wt\ga)
        Q_u(d\ga)Q_v(d\wt\ga)
    \right|\\
    &\qquad\leq
 M  \left(
      \sup_{u\in K_1}\|Q_u^\eps-Q_u\|_{\mathrm{TV}}
      +
      \sup_{v\in K_2}\|Q_v^\eps-Q_v\|_{\mathrm{TV}}
    \right)
    \longrightarrow0
  \end{align*}
  by \eqref{TV_unif}.  This proves
  \eqref{prod_unif}.
  \qedhere

\end{proof}

Put $ \Psi^\eps=\Phi^\eps\ast\phi^\eps$,
which is an odd function with
$(\Psi^\eps)'=\phi^\eps\ast\phi^\eps$.  Then, 
\begin{align}\label{1/2bd}
  |\Psi^\eps(r)|\leq\frac12,
  \qquad r\in\R.
\end{align}
If $\supp(\phi)\subset[-R_\phi,R_\phi]$, then
\begin{align}\label{sgn_eq}
  \Psi^\eps(r)=\frac12\sgn(r)
  \qquad\text{whenever }|r|>2R_\phi\eps.
\end{align}

The only obstruction to using \eqref{sgn_eq} directly is a possible
small amount of Gibbs mass near the discontinuity of the sign
function.  The next estimate rules this out uniformly on compact
spatial sets.

\begin{lemma}[\textsf{Uniform small-diagonal estimate}]
  \label{diag_lem}
  Fix $\wt x\in\R$ and a compact set
  $K\Subset\R\setminus\{\wt x\}$.  With
  $\ga_0=\wt\ga_0=0$, if $(i,j)\neq(0,0)$, then
  \begin{align}\label{diag_unif}
    \lim_{\eta\downarrow0}\ 
    \sup_{\substack{u\in K\\0\leq\eps\leq\eps_K}}
    &(Q_{\wt x}^\eps\otimes Q_u^\eps)
    \bigl(
      |\wt x+\wt\ga_i-u-\ga_j|\leq\eta
    \bigr)=0.
  \end{align}
  Here $\eps_K>0$ is sufficiently small.  Also, for every compact
  $K\subset\R$ and $j\in\{1,\dots,n\}$,
  \begin{align}\label{diag_1}
    \lim_{\eta\downarrow0}\ 
    \sup_{\substack{u\in K\\0\leq\eps\leq\eps_K}}
    Q_u^\eps\bigl(|u+\ga_j|\leq\eta\bigr)=0.
  \end{align}
\end{lemma}

\begin{proof}
  Under $\mu_\be$, each $\ga_j$ with $j\geq1$ has a gamma density that
  is bounded on $\R_+$.  Since the two replicas are independent under
  $\mu_\be\otimes\mu_\be$, for $(i,j)\neq(0,0)$,
$
    \sup_{u\in K}
    (\mu_\be\otimes\mu_\be)
    \bigl(
      |\wt x+\wt\ga_i-u-\ga_j|\leq\eta
    \bigr)
    \leq C\eta.
$
  By Cauchy--Schwarz,
  we have 
  
  \noi
  \begin{align*}
    &(Q_{\wt x}^\eps\otimes Q_u^\eps)
      \bigl(     |\wt x+\wt\ga_i-u-\ga_j|\leq\eta   \bigr)
    \leq
      \frac{   \left(    \int_{\DN} e^{2\be H_{\wt x}^\eps}\,d\mu_\be   
       \int_{\DN} e^{2\be H_u^\eps}\,d\mu_\be   \right)^{1/2}  }{\cZ_{\wt x}^\eps \cZ_u^\eps}
      (C\eta )^{1/2}.
  \end{align*}
  Proposition~\ref{int_uc} with $q=2$ and Lemma~\ref{conv_Z} show that
  the prefactor is bounded uniformly over $u\in K$ and all sufficiently
  small $\eps$, including $\eps=0$.  This proves
  \eqref{diag_unif}.  The one-point estimate follows
  in exactly the same way from
  $
    \sup_{u\in K}\mu_\be(|u+\ga_j|\leq\eta)\leq C\eta.
  $
  \qedhere
\end{proof}

\begin{lemma}[\textsf{Convergence of the mollified sign kernels}]
  \label{ker_conv}

  Fix $\wt x\in\R$ and a compact set
  $K\Subset\R\setminus\{\wt x\}$.  For every
  $i,j\in\{0,\dots,n\}$,
  
  \noi
  \begin{align}
  \label{ker_2}
  \begin{aligned}
   & \sup_{u\in K}\bigg|
      \int_{\DN^2}
      \Psi^\eps(\wt x+\wt\ga_i-u-\ga_j)
      Q_{\wt x}^\eps(d\wt\ga)Q_u^\eps(d\ga) \\
     &\qquad\qquad\qquad  -\frac12    \int_{\DN^2}    \sgn(\wt x+\wt\ga_i-u-\ga_j)   Q_{\wt x}(d\wt\ga)Q_u(d\ga)
    \bigg|\xrightarrow{\eps\to 0} 0. 
  \end{aligned}
  \end{align}
  In addition, for every $j\in\{1,\dots,n\}$ and every compact
  $K\subset\R$,
  \begin{align}\label{ker_1}
    \sup_{u\in K}
    \bigg|    \int_{\DN}\Psi^\eps(u+\ga_j)Q_u^\eps(d\ga)     -\frac12\int_{\DN}\sgn(u+\ga_j)Q_u(d\ga)    \bigg|
    \xrightarrow{\eps\to 0} 0. 
  \end{align}
\end{lemma}

\begin{proof}
  We prove \eqref{ker_2}.  If
  $(i,j)=(0,0)$, the argument of the kernel is $\wt x-u$, which is
  bounded away from zero on $K$.  For all sufficiently small $\eps$,
  \eqref{sgn_eq} is therefore exact, and the conclusion follows from
  Proposition~\ref{Q_uc}.

  Suppose that $(i,j)\neq(0,0)$, fix $\eta>0$, and let 
   $ A_{\eta,u}
    =\{|\wt x+\wt\ga_i-u-\ga_j|\leq\eta\}.$
  If $2R_\phi\eps<\eta$, then on $A_{\eta,u}^c$ the mollified kernel
  equals one half of the sign kernel.  Proposition~\ref{Q_uc}, applied
  to the bounded measurable function
  $
    \frac12\sgn(\wt x+\wt\ga_i-u-\ga_j)
    \ind_{A_{\eta,u}^c},$
  gives uniform convergence on this complement.  The two contributions
  from $A_{\eta,u}$ are bounded by the corresponding probabilities,
  using \eqref{1/2bd}; they tend to zero uniformly as
  $\eta\downarrow0$ by Lemma~\ref{diag_lem}.  Let
  first $\eps\downarrow0$ 
  and then 
  $\eta\downarrow0$.
  The proof of \eqref{ker_1} is identical,
  using \eqref{diag_1}.
\end{proof}

We can now pass to the limit in precisely the two prelimit expressions
that occur in Section~\ref{proof_sec}.

\begin{proposition}\label{T_uc}
  Fix $\wt x\in\R$ and a compact set
  $K\Subset\R\setminus\{\wt x\}$.  Then,
  we have 
  
  \noi
  \begin{align}
    \sup_{u\in K}\bigg|
      \frac1\be T_{2,\eps}(u,\wt x)  - 
   \frac12\EE\bigg[  \int_{\DN^2} Q_{\wt x}(d\wt\ga)Q_u(d\ga) 
    \sum_{i=0}^{n-1}   \mathfrak D_i(\wt x+\wt\ga_i,u+\ga_i)    \bigg]    \bigg|
    \xrightarrow{\eps\to 0} 0 .
    \label{T2_unif}
    \end{align}
Recall from \eqref{T1_pre} that
\[
      \frac1\be T_{1,\eps}(u,\wt x)
    =\al^2\EE\bigg[\int_{\DN^2}
      \Bigl(
        \Psi^\eps(\wt x+\wt\ga_n-u-\ga_n)
        +\Psi^\eps(u+\ga_n)
      \Bigr)
      Q_{\wt x}^\eps(d\wt\ga)Q_u^\eps(d\ga)
    \bigg].
        \]
Then, we have
    
    \noi
  \begin{align}  
  \begin{aligned} 
    \sup_{u\in K}
      \bigg| \frac{\al^2}{2}\EE\bigg[
        \int_{\DN^2} Q_{\wt x}(d\wt\ga)Q_u(d\ga)
        & \Bigl(   \sgn(\wt x+\wt\ga_n-u-\ga_n)   +\sgn(u+\ga_n)     \Bigr)      \bigg]  \\
     & \qquad  -   \frac1\be T_{1,\eps}(u,\wt x)
        \bigg|
      \xrightarrow{\eps\to 0} 0. 
      \end{aligned} 
    \label{T1_unif}
  \end{align}
\end{proposition}

\begin{proof}
  For $0\leq i\leq n-1$, define
  \begin{align*}
    U_{i,\eps}(u)
    &=\int_{\DN^2}
      \Xi_{i,\eps}(\wt x+\wt\ga_i,u+\ga_i)
      Q_{\wt x}^\eps(d\wt\ga)Q_u^\eps(d\ga),\\
    U_i(u)
    &=\frac12\int_{\DN^2}
      \mathfrak D_i(\wt x+\wt\ga_i,u+\ga_i)
      Q_{\wt x}(d\wt\ga)Q_u(d\ga).
  \end{align*}
  The four applications of Lemma~\ref{ker_conv} corresponding to
  the four terms in $\Xi_{i,\eps}$ give
  \[
    \sup_{u\in K}|U_{i,\eps}(u)-U_i(u)|     \xrightarrow{\eps\to 0} 0
    \quad\text{almost surely}.
  \]
  The supremum is bounded by a deterministic constant because
  $|\Psi^\eps|\leq1/2$ and $|\mathfrak D_i|\leq4$.  Dominated convergence over
  the environment and the identity \eqref{T2_Xi} prove
  \eqref{T2_unif}.

 To prove \eqref{T1_unif},
 it suffices to apply the two-point and one-point conclusions of
  Lemma~\ref{ker_conv} and the dominated convergence
  with  the fact that the kernels are bounded by $1/2$.
  \qedhere
  
\end{proof}

We finish by justifying differentiation under the path integral and
then the exchange of the mollification limit with the spatial
derivative.

\begin{lemma}\label{pf_uxeps1}
  For every fixed $\eps>0$, almost surely the map
  $u\mapsto \cZ_u^\eps$ is continuously differentiable and
  \begin{align}\label{app_spder}
    \frac{\partial}{\partial u}\frac1\be\log \cZ_u^\eps
    =\int_{\DN} Q_u^\eps(d\ga)\bigg[
      \al\xi_n^\eps(u+\ga_n)
      +\sum_{i=0}^{n-1}
      \bigl(
        \xi_i^\eps(u+\ga_{i+1})
        -\xi_i^\eps(u+\ga_i)
      \bigr)
    \bigg].
  \end{align}
\end{lemma}

\begin{proof}
  Fix a compact interval $K$.  Integration by parts in the convolution
  defining $\xi_i^\eps$ gives
  \[
    |\xi_i^\eps(t)|
    \leq \|\phi'\|_{L^1}\eps^{-1}
      \sup_{|s-t|\leq R_\phi\eps}|B_i(s)|
    \leq C_{\eps,\phi}M_\theta(1+|t|^\theta).
  \]
  Together with \eqref{D_grow}, this yields, uniformly
  for $u\in K$,
  \begin{align*}
    e^{\be H_u^\eps(\ga)}
    \left|\frac{\partial}{\partial u}H_u^\eps(\ga)\right|
    \leq C_{\omega,\eps,K}
      (1+\ga_n^\theta)
      \exp\{C_{\omega,K}(1+\ga_n^\theta)\}.
  \end{align*}
  The right-hand side is integrable against $\mu_\be$ because
  $\theta<1$.  Dominated differentiation under the path integral gives
  \eqref{app_spder}; the same domination proves
  continuity of the derivative.
\end{proof}

\begin{proposition}\label{app_der}
  For every $x,\wt x\in\R$ with $x\neq\wt x$,
  \begin{align}\label{app_derlim}
    \frac{\partial}{\partial x}
    \EE\bigl[
      F_{\wt x}^{\al,\be}F_{x}^{\al,\be}
    \bigr]
    =\lim_{\eps\downarrow0}
    \EE\left[
      \frac1\be\log \cZ_{\wt x}^\eps
      \frac{\partial}{\partial x}
      \frac1\be\log \cZ_x^\eps
    \right].
  \end{align}
\end{proposition}

\begin{proof}
  Fix a compact interval
  $I\Subset\R\setminus\{\wt x\}$ containing $x$ in its interior, and
  put
  \[
    \widehat F_u^\eps=\frac1\be\log \cZ_u^\eps,
    \qquad
    \widehat F_u=\frac1\be\log \cZ_u,
    \qquad
    G_\eps(u)=\EE[\widehat F_{\wt x}^\eps\widehat F_u^\eps].
  \]
  For fixed $\eps>0$, Lemma~\ref{pf_uxeps1} gives the pathwise
  derivative.  Its proof and Jensen's inequality also give
  \begin{align}\label{moll_Dmom}
    \left|
      \frac{\partial}{\partial u}\widehat F_u^\eps
    \right|
    \leq C_{\eps,I}M_\theta
      \left(1+\int_{\DN}\ga_nQ_u^\eps(d\ga)\right),
    \qquad u\in I.
  \end{align}
  Hence, for every $p<\infty$,
  $
    \sup_{u\in I}  \big\|  \frac{\partial}{\partial u}\widehat F_u^\eps   \big\|_{L^p(\Omega)}<\infty
  $
  by Lemma~\ref{B_poly}, H\"older's
  inequality, and \eqref{expec_Q}.  The moment estimate
  \eqref{LP_logZRL} applies to $\widehat F_u^\eps$ as well, because it differs from the normalization in Appendix~A by the
  deterministic constant $n\log\be$.  The preceding bounds make the
  difference quotients uniformly integrable, so differentiation under
  $\EE$ is justified.  Using the definitions in Section~\ref{proof_sec}, we obtain
  \begin{align}\label{G_der}
    G_\eps'(u)
    =\EE\left[  \widehat F_{\wt x}^\eps   \frac{\partial}{\partial u}\widehat F_u^\eps  \right] 
    =\frac1\be T_{1,\eps}(u,\wt x)    +\frac1\be T_{2,\eps}(u,\wt x).
  \end{align}

  Proposition~\ref{T_uc} shows that $G_\eps'$ converges uniformly on
  $I$ as $\eps\downarrow0$.  At one fixed $u_0\in I$,
  \eqref{LPcvg_eps} and Cauchy--Schwarz give
  $
    G_\eps(u_0) \to   \EE[\widehat F_{\wt x}\widehat F_{u_0}]
  $
  as $\eps\to 0$.
  The uniform convergence theorem for derivatives
  \cite[Theorem~7.17]{Rudin} implies that $G_\eps$ converges uniformly
  on $I$ to a differentiable function and that its derivative is the
  uniform limit of $G_\eps'$.  On the other hand,
  \eqref{LPcvg_eps} gives, for every $u\in I$,
  $
    G_\eps(u) \to 
    \widehat G(u)=\EE[\widehat F_{\wt x}\widehat F_u]
  $
    as $\eps\to 0$.
  Hence the limiting function is $\widehat G$, and
  \begin{align}\label{Ghat_lim}
    \widehat G'(x)
    =\lim_{\eps\downarrow0}G_\eps'(x).
  \end{align}

  Finally,
    $F_{u}^{\al,\be}
    =\widehat F_u-\frac n\be\log\be.$
  As observed at the beginning of this appendix,
  $u\mapsto\EE[\widehat F_u]$ is constant.  Therefore the spatial
  derivative of
  $\EE[F_{\wt x}^{\al,\be}F_{u}^{\al,\be}]$ equals
  $\widehat G'(u)$.  Combining this with
  \eqref{G_der} and
  \eqref{Ghat_lim} proves
  \eqref{app_derlim}.
\end{proof}

\medskip
\noi
$\bullet$ {\bf Acknowledgments.}
We thank Yu Gu for insightful discussions
and for his suggestion of providing
another proof of Burke's property,
which is now the content of 
Corollary \ref{buse_cor}. Generative AI tools were used to assist with grammar editing, LaTeX formatting, and preliminary literature searches. The authors verified all calculations and references, and take full responsibility for the content of the paper.

\end{document}